\documentclass[11pt]{article}
\usepackage{color}
\usepackage[colorlinks=true]{hyperref}
\usepackage[margin=1in]{geometry}
\usepackage{amssymb}
\usepackage{amsmath}
\usepackage{amsthm}
\usepackage{mathtools}
\usepackage{epsfig, graphicx}
\usepackage[boxed, noline, ruled, linesnumbered]{algorithm2e}
\newtheorem{theorem}{Theorem}
\newtheorem{lemma}{Lemma}

\newtheorem{remark}{Remark}
\newcommand{\R}{\mathbb{R}}

\allowdisplaybreaks 
\newcommand{\revise}[1]{{\color{black}#1}}  
\begin{document}
\title{Tensor Neural Network and Its Numerical Integration\footnote{This work was
supported in part by the National Key Research and Development Program of China
(2019YFA0709601), the National Center for Mathematics and Interdisciplinary Science, CAS.}}
\author{
Yifan Wang\footnote{LSEC, NCMIS, Institute
of Computational Mathematics, Academy of Mathematics and Systems
Science, Chinese Academy of Sciences, Beijing 100190,
China,  and School of Mathematical Sciences, University
of Chinese Academy of Sciences, Beijing 100049, China (wangyifan@lsec.cc.ac.cn).}, \ \
Pengzhan Jin\footnote{School of Mathematical Sciences, Peking University, Beijing 100871, China (jpz@pku.edu.cn).},
\ \  and \ \
Hehu Xie\footnote{LSEC, NCMIS, Institute of Computational Mathematics, Academy of Mathematics and Systems
Science, Chinese Academy of Sciences, Beijing 100190,
China,  and School of Mathematical Sciences, University
of Chinese Academy of Sciences, Beijing 100049, China (hhxie@lsec.cc.ac.cn).}}
\date{}
\maketitle
\begin{abstract}
In this paper, we introduce a type of tensor neural network. For the first time, we propose
its numerical integration scheme and prove the computational complexity to be the polynomial
scale of the dimension.
Based on the tensor product structure, we develop an efficient numerical integration method by
using fixed quadrature points for the functions of the tensor neural network.
The corresponding machine learning method is also introduced for solving high-dimensional
problems.
Some numerical examples are also provided to validate the theoretical results and the
numerical algorithm.

\vskip0.3cm {\bf Keywords.} Tensor neural network, numerical integration, fixed quadrature points,
machine learning, high-dimensional eigenvalue problem.
\vskip0.2cm {\bf AMS subject classifications.} 65N30, 65N25, 65L15, 65B99
\end{abstract}

\section{Introduction}

Partial differential equations (PDEs) appear in many scientific and industrial applications since they can describe physical and engineering phenomena or processes.  So far, many types of numerical methods have been developed such as the finite difference method, finite element method, and spectral method for solving PDEs in three spatial dimensions plus the temporal dimension. But there exist many high-dimensional PDEs such as many-body Schr\"{o}dinger, Boltzmann equations, Fokker-Planck equations, and stochastic PDEs (SPDEs), which are almost impossible to be solved using traditional numerical methods. Recently, many numerical methods have been proposed based on machine learning to solve the high-dimensional PDEs (\cite{BaymaniEffati,WeinanE, EYu,HanJentzenE,LagarisLikasPapageorgiou, LitsarevOseledets, RaissiPerdikarisKarniadakis, DGM, WAN}). Among these machine learning methods, neural network-based methods attract more and more attention. Neural networks can be used to build approximations of the exact solutions of PDEs by machine learning methods. The reason is that neural networks can approximate any function given enough parameters. This type of method provides a possible way to solve many useful high-dimensional PDEs from physics, chemistry, biology, engineering, and so on.

Due to its universal approximation property, the fully-connected neural network (FNN) is the most widely used architecture to build the functions for solving high-dimensional PDEs. There are several types of FNN-based methods such as well-known the deep Ritz \cite{EYu}, deep Galerkin method \cite{DGM}, PINN \cite{RaissiPerdikarisKarniadakis}, and weak adversarial networks \cite{WAN}
for solving high-dimensional PDEs by designing different loss functions. Among these methods, the loss functions always include computing high-dimensional integration for the functions defined by FNN. For example,  the loss functions of the deep Ritz method require computing the integrations on the high-dimensional domain for the functions which is constructed by FNN. Direct numerical integration for the high-dimensional functions also meets the ``curse of dimensionality''. Always, the Monte-Carlo method is adopted to do the high-dimensional integration with some types of sampling methods
\cite{EYu,HanZhangE}.  Due to the low convergence rate of the Monte-Carlo method,
the solutions obtained by the FNN-based numerical methods are difficult to obtain high accuracy and stable convergence process. In other words, the Monte-Carlo method decreases computational work in each forward propagation by decreasing the simulation efficiency and stability of the FNN-based numerical methods for solving high-dimensional PDEs.

The CANDECOMP/PARAFAC (CP) tensor decomposition builds a low-rank approximation method and is a widely
used way to cope with the curse of dimensionality. The CP method decomposes a tensor
as a sum of rank-one tensors which can be considered as the higher-order
extensions of the singular value decomposition (SVD)
for the matrices.
This means the SVD idea can be generalized to the decomposition of the high-dimensional Hilbert
space into the tensor product of several Hilbert spaces.
The tensor product decomposition has been used to establish low-rank approximations
of operators and functions \cite{BeylkinMohlenkamp,HackbuschKhoromskij,JinMengLu,ReynoldsDoostanBeylkin}.
If we use the low-rank approximation to do the numerical integration, the computational
complexity can avoid the exponential dependence on the dimension in some cases \cite{BeylkinMohlenkampInt,LitsarevOseledets}.
Inspired by CP decomposition, this paper focuses on a special low-rank neural networks structure
and its numerical integration. It is worth mentioning that although CP decomposition
should be useful to obtain a low-rank approximation, there is no known general result
to give the relationship between the rank (hyperparameter $p$ in this paper) and error bounds.
For more details, please refer to \cite{HongKoldaDuersch,KoldaBader} and numerical
investigations \cite{BeylkinMohlenkamp}.

This paper aims to propose a type of tensor neural network (TNN) to build the trial
functions for solving high-dimensional PDEs.
The TNN is a function being designed by the tensor product operations on the neural networks
or by low-rank approximations of FNNs.
An important advantage is that we do not need to use Monte-Carlo method to
do the integration for the functions which is constructed by TNN.
This is the main motivation to design the TNN for high-dimensional PDEs in this paper.
We will show, the computational work for the integration of
the functions by TNN is only a polynomial scale of the dimension,
which means the TNN overcomes the ``curse of dimensionality'' in some sense for
solving high-dimensional PDEs.

An outline of the paper goes as follows. In Section \ref{Section_TNN}, we introduce the way to build TNN.
The numerical integration method for the functions constructed by TNN is designed
in Section \ref{Section_Integration}.
Section \ref{Section_Eigenvalue} is devoted to proposing the TNN-based machine learning method
for solving the high-dimensional eigenvalue
problem with the numerical integration method. Some numerical examples are provided
in Section \ref{Section_Numerical}
to show the validity and efficiency of the proposed numerical methods in this paper.
Some concluding remarks are given in the last section.

\section{Tensor neural network}\label{Section_TNN}
\subsection{The architecture of tensor neural network}
In this section, we introduce the TNN and its approximation property.
Without loss of generality, we first design the general TNN architecture with $K$-dimensional output to accommodate
more general computational aims. Then we consider the $1$-dimensional
output TNN architecture which is our primary
focus in this paper. Of course, it is easy to know that the $K$-dimensional output TNN can also be built
with $K$ $1$-dimensional output TNN.
The approximation property for the $1$-dimensional TNN, which will be given in this section,
can be directly extended to $K$-dimensional output TNN.

The architecture of TNN is similar to MIONet, just by setting the Banach spaces to Euclidean spaces,
more details about MIONet can be found in \cite{JinMengLu}.
MIONet mainly discusses the approximation of multiple-input continuous operators by low-rank
neural network structures and investigates the function approximation under $C$-norm.
The inputs of MIONet are the vectors that denote the coefficients of projections of
functions in infinite-dimensional Banach spaces onto the concerned finite-dimensional subspace.
While TNN considers solving high-dimensional PDEs and pays more attention
to the high-dimensional integration and the approximation to functions in Sobolev space
by low-rank neural network structures in $H^m$-norm. \revise{Different from MIONet,
the inputs of TNN are coordinates in the high-dimensional Euclidean space.
The tensor product structure of TNN can lead to high precision and high efficiency in calculating the numerical
integrations in the loss function derived from the variational principle.
The most important contribution of this paper is to reveal that we can do the highly accurate and efficient
numerical integrations of TNN for solving high-dimensional PDEs with TNN.  Furthermore,
this paper also shows that the Monte-Carlo or stochastic sampling process is not necessary for
machine learning}.
More specifically, we first construct $d$ subnetworks, where each subnetwork
is a continuous mapping from a bounded closed set $\Omega_i\subset\mathbb R$ to $\mathbb R^p$.
The $i$-th subnetwork can be expressed as:
\begin{eqnarray}\label{def_FNN}
\mathbf\Phi_i(x_i;\theta_i)=\big(\phi_{i,1}(x_i;\theta_i),\phi_{i,2}(x_i;\theta_i),
\cdots,\phi_{i,p}(x_i;\theta_i)\big)^T,\ \ \ i=1, \cdots, d,
\end{eqnarray}
where $x_i$ denotes the 1-dimensional input, $\theta_i$ denotes the parameters of the $i$-th
subnetwork, typically the weights and biases.
The number of layers and neurons in each layer, the selections of activation functions and
other hyperparameters can be different in different subnetworks. In this paper,
we simply use FNN architectures for each subnetwork. It is worth mentioning that,
in additon to FNN, other reasonable architecture can be used as long as it can approximate any mapping
from $\Omega_i\subset\mathbb R$ to $\mathbb R^p$ in some sense.
The only thing to be guaranteed is that the output dimensions of these subnetworks should be equal.
After building each subnetwork,
we combine the output layers of each subnetwork to obtain TNN architecture by the
following mapping from $\R^d$ to $\R^K$
\begin{equation}\label{def_TNN_K}
\mathbf\Psi(x;\theta)=W\cdot(\mathbf\Phi_{1}(x_1;\theta_1)\odot\mathbf\Phi_{2}(x_2;\theta_2)
\odot\cdots\odot\mathbf\Phi_{d}(x_d;\theta_d)),
\end{equation}
where $\odot$ is the Hadamard product (i.e., element-wise product), $\mathbf\Psi$
denotes a $K$-dimensional output function
which is defined as $\mathbf\Psi(x;\theta)=(\Psi_1(x;\theta),\cdots,\Psi_K(x;\theta))^T$,
the matrix $W\in\R^{K\times p}$ and $x=(x_1,\cdots,x_d)\in\Omega_1\times\cdots\times\Omega_d$.
In the following part of this paper, we set $\Omega=\Omega_1\times\cdots\times\Omega_d$.
Here $\theta=\{\theta_1,\cdots,\theta_d,W\}$ denote trainable parameters.
Figure \ref{TNNstructureK} shows the architecture of $K$-dimensional output TNN.
\begin{figure}[htb]
\centering
\includegraphics[width=16cm,height=12cm]{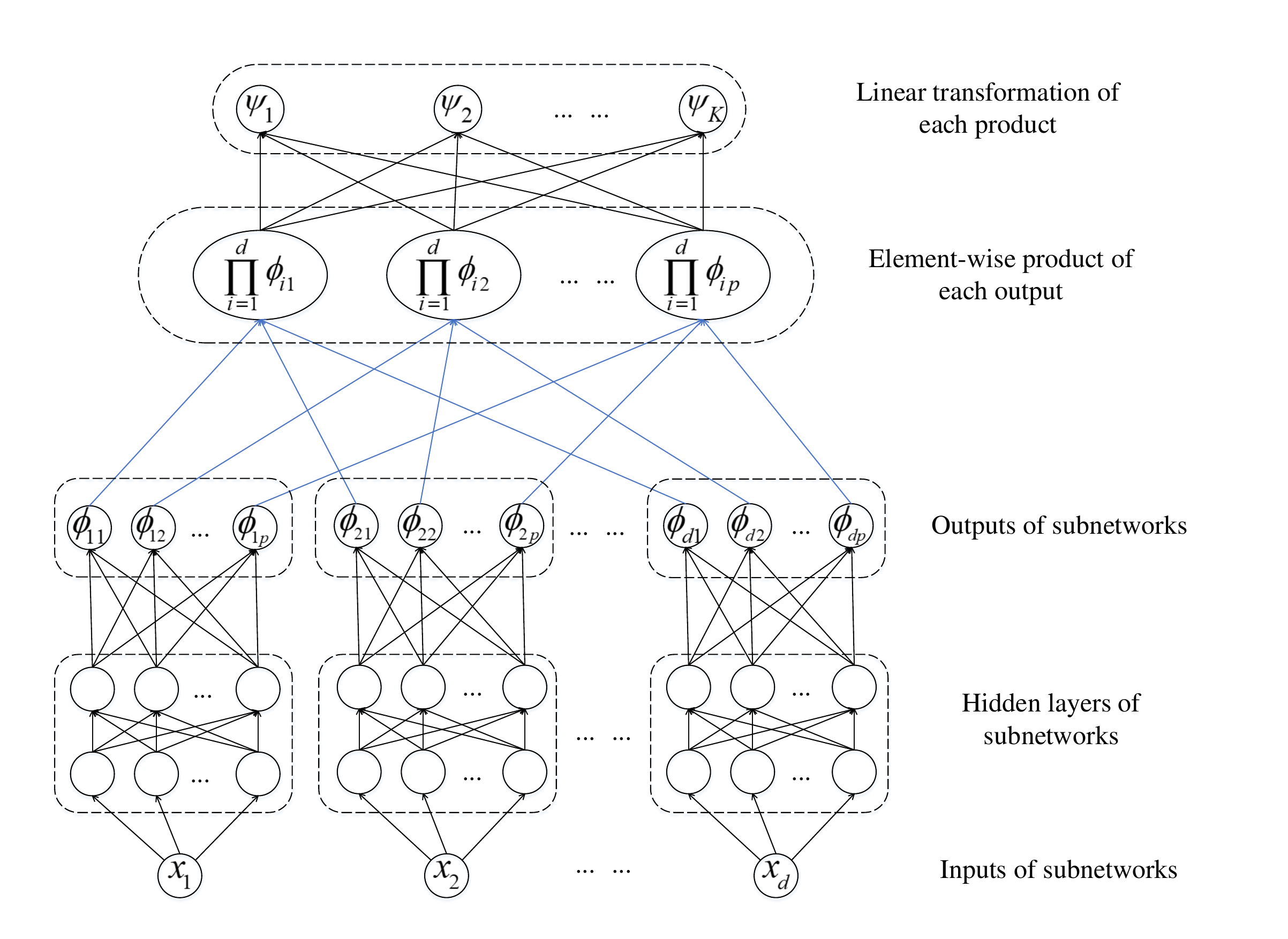}
\caption{Architecture of $K$-dimensional output TNN. Black arrows mean linear transformation
(or affine transformation). Each ending node of blue arrows is obtained by taking the scalar
multiplication of all starting nodes of blue arrows that end in this ending node.}\label{TNNstructureK}
\end{figure}

The $1$-dimensional output TNN (i.e. $K=1$) is always enough for solving normal high-dimensional PDEs.
When $K=1$, the matrix $W$ appears in (\ref{def_TNN_K}) degenerates to a row vector,
its members only play the role to scale the components of vectors obtained by the Hadamard product.
This effect can also be achieved by scaling the parameters of the output layers of the concerned subnetworks.
Therefore, in order to reduce the number of parameters,
we set the matrix $W$ to be unity and define the $1$-dimensional TNN as follows:
\begin{eqnarray}\label{def_TNN}
\Psi(x;\theta)=\sum_{j=1}^p\phi_{1,j}(x_1;\theta_1)\phi_{2,j}(x_2;\theta_2)\cdots\phi_{d,j}(x_d;\theta_d)
=\sum_{j=1}^p\prod_{i=1}^d\phi_{i,j}(x_i;\theta_i),
\end{eqnarray}
where $\theta=\{\theta_1,\cdots,\theta_d\}$ denotes all parameters of the whole architecture.
Figure \ref{TNNstructure} shows the corresponding architecture of $1$-dimensional output TNN.
For simplicity, TNN refers to the $1$-dimensional TNN hereafter in this paper.

\begin{figure}[htb]
\centering
\includegraphics[width=16cm,height=12cm]{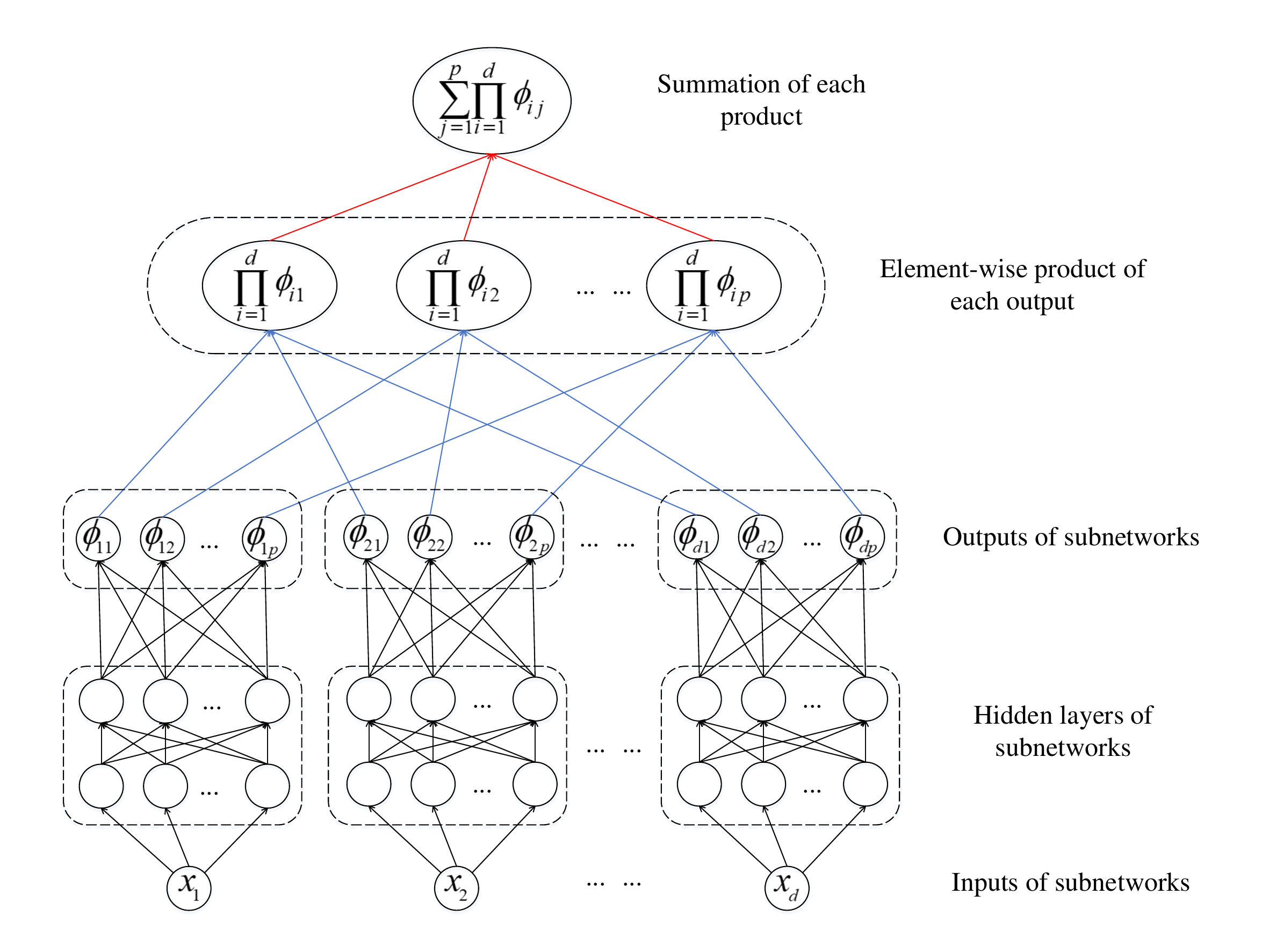}
\caption{Architecture of $1$-dimensional output TNN. Black arrows mean linear transformation
(or affine transformation). Each ending node of blue arrows is obtained by taking the
scalar multiplication of all starting nodes of blue arrows that end in this ending node.
The finall output of TNN is derived from summation of all starting nodes of red arrows.}\label{TNNstructure}
\end{figure}

Since there exists the isomorphism relation between $H^m(\Omega_1\times\cdots\times\Omega_d)$
and the tensor product space $H^m(\Omega_1)\otimes\cdots\otimes L^2(\Omega_d)$,
the process of approximating the function $f(x)\in H^m(\Omega_1\times\cdots\times\Omega_d)$
with the TNN defined by (\ref{def_TNN}) can be regarded as searching for a CP decomposition structure to
approximate $f(x)$ in the space $H^m(\Omega_1)\otimes\cdots\otimes H^m(\Omega_d)$
with the rank being not greater than $p$.
Due to the low-rank structure, we will find that the polynomial compound acting on the TNN and its derivatives
can be integrated with small scale computational work.

\subsection{Approximation of TNN in Sobolev space}
In order to show the validity for solving PDEs by TNN, we introduce the
following approximation result for the functions in the space $H^m(\Omega_1\times\cdots\times\Omega_d)$
under the sense of $H^m$-norm.

\begin{theorem}\label{theorem_approximation}
Assume that each $\Omega_i$ is a bounded \revise{open} interval in $\mathbb R$ for $i=1, \cdots, d$, $\Omega=\Omega_1\times\cdots\times\Omega_d$,
and the function $f(x)\in H^m(\Omega)$. Then for any tolerance $\varepsilon>0$, there exist a
positive integer $p$ and the corresponding TNN defined by (\ref{def_TNN})
such that the following approximation property holds
\begin{equation}\label{eq:L2_app}
\|f(x)-\Psi(x;\theta)\|_{H^m(\Omega)}<\varepsilon.
\end{equation}
\end{theorem}
\begin{proof}
Due to the isomorphism relation $H^m(\Omega)\cong H^m(\Omega_1)\otimes\cdots\otimes H^m(\Omega_d)$,
for any $\varepsilon>0$,
there exist a positive integer $p$, $h_{i,j}(x_i)\in H^m(\Omega_i),i=1,\cdots,d,j=1,\cdots,p$, and $h(x)\in H^m(\Omega)$
which is defined as follows
\begin{eqnarray*}
h(x)=\sum_{j=1}^ph_{1,j}(x_1)\cdots h_{d,j}(x_d)=\sum_{j=1}^p\prod_{i=1}^dh_{i,j}(x_i),
\end{eqnarray*}
such that the following estimate holds
\begin{equation}\label{Estimate_0}
\|f(x)-h(x)\|_{H^m(\Omega)}<\frac{\varepsilon}{2}.
\end{equation}
Denote $M_j=\max\limits_i\|h_{i,j}(x_i)\|_{H^m(\Omega_i)}, j=1,\cdots,p$, and
\begin{eqnarray*}
M=\sum_{j=1}^p\left(\tbinom{d}{1}M_j^{d-1}+\tbinom{d}{2}M_j^{d-2}+\cdots+\tbinom{d}{d-1}M_j^1+1\right).
\end{eqnarray*}
\revise{
From the density results in \cite[Chapter 5.3.3]{Evans} and $\Omega_i$ is a bounded open interval in $\mathbb R$,
there exists a $\bar h_{i,j}(x_i)\in C^\infty(\bar\Omega_i)\subset C(\bar\Omega_i)$ can approximate the one-dimensional function $h_{i,j}(x_i)$ with arbitrary accuracy under $H^m(\Omega)$-norm.
In \cite{Ellacott,Leshno}, it is shown that one-hidden layer FNN can approximate any continuous function on a compact set as long as the activation function is not a polynomial.
This conclusion can be naturally generalized from $\bar\Omega_i\rightarrow\mathbb R$ to $\bar\Omega_i\rightarrow\mathbb R^p$.
Then for $\delta=\min\left\{1,\frac{\varepsilon}{2M}\right\}$, there exist FNN structures
$\phi_i(x_i;\theta_i),i=1,\cdots,d$,
which are defined by (\ref{def_FNN}), such that
\begin{eqnarray}\label{Estimate_1}
\left\|h_{i,j}(x_i)-\phi_{i,j}(x_i;\theta_i)\right\|_{H^m(\Omega_i)}<\delta,\ \ \  i=1,\cdots,d, \ \ j=1,\cdots,p.
\end{eqnarray}
Denote $e_{i,j}(x_i)=\phi_{i,j}(x_i;\theta_i)-h_{i,j}(x_i)$, inequalities in  (\ref{Estimate_1}) imply that $\|e_{i,j}(x_i)\|_{H^m(\Omega_i)}<\delta$.
}

Since the property of multidimensional integrations on the tensor product domain $\Omega$,
for any $g_i(x_i)\in H^m(\Omega_i),i=1,\cdots,d$, the following inequality holds
\begin{eqnarray}\label{eq_mutiple_integrations}
\left\|\prod_{i=1}^dg_i(x_i)\right\|_{H^m(\Omega)}\leq\prod_{i=1}^d\left\|g_i(x_i)\right\|_{H^m(\Omega_i)}.
\end{eqnarray}
For the sake of clarity, we give a simple proof for (\ref{eq_mutiple_integrations}) as follows:
\begin{eqnarray*}
\left\|\prod_{i=1}^d g_i(x_i) \right\|_{H^m(\Omega)}^2&=&\sum_{|\alpha|\leq m}
\left\|D^\alpha\left(\prod_{i=1}^d g_i(x_i)\right)\right\|_{L^2(\Omega)}^2
=\sum_{\alpha_1+\cdots+\alpha_d\leq m}\left\|\prod_{i=1}^d\frac{\partial^{\alpha_i}g_i(x_i)}{\partial x_i^{\alpha_i}}\right\|_{L^2(\Omega_i)}^2\\
&=&\sum_{\alpha_1+\cdots+\alpha_d\leq m}\prod_{i=1}^d\left\|\frac{\partial^{\alpha_i}g_i(x_i)}{\partial x_i^{\alpha_i}}\right\|_{L^2(\Omega_i)}^2
\leq\prod_{i=1}^d\left(\sum_{\alpha_i\leq m}\left\|\frac{\partial^{\alpha_i}g_i(x_i)}{\partial x_i^{\alpha_i}}\right\|_{L^2(\Omega_i)}^2\right)\\
&=&\prod_{i=1}^d\left\|g_i(x_i)\right\|_{H^m(\Omega_i)}^2.
\end{eqnarray*}

Then from the property of binomial multiplication and inequality (\ref{eq_mutiple_integrations}),
we can build a TNN $\Psi(x;\theta)$ by (\ref{def_TNN}) such that the following inequalities hold
\begin{eqnarray}\label{Estimate_2}
&&\left\|h(x)-\Psi(x;\theta)\right\|_{H^m(\Omega)}
=\left\|\sum_{j=1}^p\prod_{i=1}^dh_{i,j}(x_i)-\sum_{j=1}^p\prod_{i=1}^d
\big(h_{i,j}(x_i)+e_{i,j}(x_i)\big)\right\|_{H^m(\Omega)}\nonumber\\
&&\leq \sum_{j=1}^p\left\|\prod_{i=1}^dh_{i,j}(x_i)-\prod_{i=1}^d
\big(h_{i,j}(x_i)+e_{i,j}(x_i)\big)\right\|_{H^m(\Omega)}\nonumber\\
&&\leq\sum_{j=1}^p\left(\tbinom{d}{1}M_j^{d-1}\delta^1+\tbinom{d}{2}M_j^{d-2}\delta^2
+\cdots\tbinom{d}{d}M_j^0\delta^d\right)\nonumber\\
&&\leq\sum_{j=1}^p\left(\tbinom{d}{1}M_j^{d-1}+\tbinom{d}{2}M_j^{d-2}
+\cdots+\tbinom{d}{d-1}M_j^1+1\right)\delta\nonumber\\
&&\leq M\delta<\frac{\varepsilon}{2}.
\end{eqnarray}
Therefore, from (\ref{Estimate_0}), (\ref{Estimate_2}) and triangle inequality, we have following estimates
\begin{eqnarray*}
&&\|f(x)-\Psi(x;\theta)\|_{H^m(\Omega)} \nonumber\\
&&\leq \|f(x)-h(x)\|_{H^m(\Omega)} +\|h(x)- \Psi(x;\theta)\|_{H^m(\Omega)} <\frac{\varepsilon}{2}+\frac{\varepsilon}{2}=\varepsilon.
\end{eqnarray*}
This is the desired result (\ref{eq:L2_app}) and the proof is complete.
\end{proof}

\revise{
Theorem \ref{theorem_approximation} gives the approximation property of TNN, it shows that TNN can approximate any $H^m(\Omega)$ function under $H^m(\Omega)$-norm.
It has to be pointed out that, since TNN has a tensor structure, each sub-network has only one-dimensional input.
In the proof of Theorem \ref{theorem_approximation}, we only need the approximation property of FNNs with  one-dimensional input.
Compared with that of the FNNs with the $d$-dimensional input,
the analysis of  that with one-dimensional input is always easier.
Here we cite a few instructive conclusions.
In \cite{HeLiXuZheng}, it is shown that linear finite element basis can be represented by the FNN with
one-dimensioanl input and the activation function being defined by the rectified linear unit (ReLU).
In \cite{LiTangYu}, it is proved that the monomail $x^n,n\in\mathbb N$ can be represented exactly by
the FNN with the rectified power unit (ReQU) acting as the activation function.
}

\revise{
\subsection{Approximation in $H_{\rm mix}^{t,\ell}(\Omega)$ by TNN}
Although there is no general result to give the relationship between the hyperparameter $p$ and error bounds,
there are still some estimations of traditional methods that can be used.}
For example, the sparse grid method and hyperbolic cross approximation method have become widely-used
numerical tools for high-dimensional problems \cite{ShenYu1,ShenYu2}.
\revise{These two methods also
assume that the approximation function has the similar tensor-product form, while each one dimensional
function is defined on the linear space with fixed basis. }
The conclusions about the cardinal of subspaces for sparse grid method and hyperbolic cross approximation method can be extended to the analysis of the hyperparameter $p$ of TNN.

For clarity, we focus on the periodic setting with $I^d=I\times I\times\cdots\times I=[0,2\pi]^d$ and
the approximations property of TNN to the functions in the linear space which is defined with Fourier basis.
Note that similar approximation results of TNN can be extended to the non-periodic functions.

Fist, for each variable $x_i\in[0,2\pi]$, let us define the one-dimensional Fourier basis $\{\varphi_{k_i}(x_i):= \frac{1}{\sqrt{2\pi}}e^{{\rm i}k_ix_i},k_i\in\mathbb Z\}$ and calssify functions via the decay of their Fourier coefficients.
For example, the isotropic Sobolev spaces \cite{Adams} on $I$ can be defined as follow
\begin{eqnarray}
H^s(I)=\left\{u(x_i)=\sum_{k_i\in\mathbb Z}c_{k_i}\varphi_{k_i}(x_i):\|u\|_{H^s(I)}=\left(\sum_{k_i\in\mathbb Z}(1+|k_i|)^{2s}\cdot|c_{k_i}|^2\right)^{1/2}<\infty \right\}.
\end{eqnarray}
Further denote multi-index $k=(k_1,\cdots,k_d)\in\mathbb Z^d$ and $x=(x_1,\cdots,x_d)\in I^d$.
Then the $d$-dimensional Fourier basis can be built with the tensor product way
\begin{eqnarray}
\varphi_k(x)\coloneqq\prod_{i=1}^d\varphi_{k_1}(x_i)=(2\pi)^{-d/2}e^{{\rm i}k\cdot x}.
\end{eqnarray}
We denote
\begin{eqnarray}
\lambda_{\rm mix}(k)\coloneqq\prod_{i=1}^d(1+|k_i|)\ \ \ {\rm and}\ \ \ \lambda_{\rm iso}(k)\coloneqq 1+\sum_{i=1}^d|k_i|.
\end{eqnarray}
Now, for $-\infty<t,\ell<\infty$, we  define the space $H_{\rm mix}^{t,\ell}(I^d)$ as follows (cf. \cite{})
\begin{eqnarray}
H_{\rm mix}^{t,\ell}(I^d)=\left\{u(x)=\sum_{k\in\mathbb Z^d}c_k\varphi_k(x):\|u\|_{H_{\rm mix}^{t,\ell}(I^d)}=\left(\sum_{k\in\mathbb Z^d}\lambda_{\rm mix}(k)^{2t}\cdot\lambda_{\rm iso}(k)^{2\ell}\cdot|c_k|^2\right)^{1/2}<\infty \right\}.
\end{eqnarray}
Then the standard isotropic Sobolev spaces \cite{Adams} and the Sobolev space of dominating mixed smoothness \cite{SchmeisserTriebel} can be written as
\begin{eqnarray}
H^s(I^d)=H_{\rm mix}^{0,s}(I^d)
=\left\{u(x)=\sum_{k\in\mathbb Z^d}c_k\varphi_k(x):\|u\|_{H_{\rm mix}^s(I^d)}=\Big(\lambda_{\rm iso}(k)^{2s}\cdot|c_k|^2\Big)^{1/2}<\infty\right\},
\end{eqnarray}
and
\begin{eqnarray}
H_{\rm mix}^t(I^d)=H_{\rm mix}^{t,0}(I^d)
=\left\{u(x)=\sum_{k\in\mathbb Z^d}c_k\varphi_k(x):\|u\|_{H_{\rm mix}^{t,0}(I^d)}=\left(\sum_{k\in\mathbb Z^d}\lambda_{\rm mix}(k)^{2t}\cdot|c_k|^2\right)^{1/2}<\infty\right\},
\end{eqnarray}
respectively.
Note that the parameter $\ell$ governs the isotropic smoothness, whereas $t$ governs the mixed smoothness.
The spaces $H_{\rm mix}^{t,\ell}(I^d)$ gives a quite flexible framework for the study of problems in Sobolev spaces.
See \cite{GriebelHamaekers,GriebelKnapek,Knapek} for more information on the space $H_{\rm mix}^{t,\ell}(I^d)$.

Second, for $K\in\mathbb N$ and $T\in(-\infty,1]$, define the following general sparse grid space to approximate functions in space $H_{\rm mix}^{t,\ell}(I^d)$ according to the frequency $k$
\begin{eqnarray}
V_{K,T}\coloneqq{\rm span}\left\{\varphi_k(x):k\in\mathbb Z^d,\lambda_{\rm mix}(k)\cdot\lambda_{\rm iso}(k)^{-T}\leq K^{1-T}\right\}.
\end{eqnarray}
The corresponding multi-index set of frequency $k$ is
\begin{eqnarray}
D_{K,T}\coloneqq\left\{k=(k_1,\cdots,k_d):\lambda_{\rm mix}(k)\cdot\lambda_{\rm iso}(k)^{-T}\leq K^{1-T}\right\}.
\end{eqnarray}
Obviously, the degree of freedom of space $|V_{K,T}|$  and the cardinal of the set $|D_{K,T}|$ are equivalent.
By \cite{Knapek2000,Zung}, the degree of freedom of space $V_{K,T}$ with respect to the parameter $K$ and $T$ is
\begin{eqnarray}
|V_{T,K}|=
\left\{
\begin{aligned}
&\mathcal O(K+1),&\ \ \ &{\rm for}\ 0<T<1,\\
&\mathcal O((K+1)\cdot\log(K+1)^{d-1}),&\ \ \ &{\rm for}\ T=0,\\
&\mathcal O((K+1)^{\frac{T-1}{T/d-1}}),&\ \ \ &{\rm for}\ T<0,\\
&\mathcal O((K+1)^d),&\ \ \ &{\rm for}\ T=-\infty.
\end{aligned}
\right.
\end{eqnarray}
In this case of $0<T<1$, the degree of freedom of the space $V_{K,T}$ as well as the cardinal of
the set $D_{K,T}$ are independent of dimension $d$.
The following lemma gives the approximation property of the space $V_{K,T}$.
\begin{lemma}\label{lemma_sparse_grid}
Let $f\in H_{\rm mix}^{t,\ell}(I^d)$,
$f_{K,T}$ be the best approximation in $V_{K,T}$ with respect to $H^m$-norm.
Futhermore denote by $p$ the actual number of degrees of freedom of $V_{K,T}$ as well as the cardinal of set $D_{K,T}$. Consider the case $T\in(0,(m-\ell)/t]$. Then, there holds
\begin{eqnarray}
\|f-f_{K,T}\|_{H^m(I^d)}\leq C(d)\cdot p^{-(\ell-m+t)}\cdot\|u\|_{H_{\rm mix}^{t,\ell}(I^d)},
\end{eqnarray}
where $C(d)\leq c\cdot d^2\cdot0.97515^d$ and the constant $c$ is independent of $d$.
\end{lemma}
Lemma \ref{lemma_sparse_grid} is introduced in \cite{GriebelHamaekers}, where more general cases are considered.
Note that in Lemma \ref{lemma_sparse_grid}, the best approximation $u_{K,T}$ has the following form
\begin{eqnarray}
f_{K,T}=\sum_{k\in D_{K,T}}c_k\varphi_k(x)=\sum_{k\in D_{K,T}}c_k\prod_{i=1}^d\varphi_{k_i}(x_i),
\end{eqnarray}
which is similar to the structure of TNN (\ref{def_TNN}).
Analogically, the hyperparameter $p$ in TNN is equivalent to the cardinal of set $D_{K,T}$ and each one-dimensional Fourier basis  $\varphi_{k_i}(x_i)$ is equivalent to $\phi_{i,j}(x_i)$ in (\ref{def_TNN}).
That's why we denote $p$ in Lemma \ref{lemma_sparse_grid} as the cardinal of set $D_{K,T}$.

In order to obtain a comprehensive error estimate for TNN, the problem we left behind is whether an one-dimensional Fourier basis function can be approximated by an FNN with one-dimensional input.
Fortunately, the Fourier basis function $\varphi_{k_i}(x_i)=\frac{1}{\sqrt{2\pi}}e^{-{\rm i}k_ix_i}$ can be represented by the FNN with one-dimensional input, one hidden layer and activation function $\sigma(x)=\sin(x)$.
The reason is based on the following property
\begin{eqnarray*}
e^{-{\rm i}k_ix_i} = \cos(k_ix_i)-{\rm i}\sin(k_ix_i) = \sin\left(\frac{\pi}{2}-k_ix_i\right)-{\rm i}\sin(k_ix_i).
\end{eqnarray*}
Thus, we can immediately obtain the following comprehensive error estimate for TNN.
\begin{theorem}\label{theorem_aprrox_rate}
Assume function $f(x)\in H_{\rm mix}^{t,\ell}(I^d)$, $t>0$ and $m>\ell$. Then there exists a TNN $\Psi(x;\theta)$ defined by (\ref{def_TNN}) such that the following approximation property holds
\begin{eqnarray}
\|f(x)-\Psi(x;\theta)\|_{H^m(I^d)}\leq C(d)\cdot p^{-(\ell-m+t)}\cdot\|u\|_{H_{\rm mix}^{t,\ell}(I^d)},
\end{eqnarray}
where $C(d)\leq c\cdot d^2\cdot0.97515^d$ and $c$ is independent of $d$. And each subnetwork of TNN is
a FNN which is built by using $\sin(x)$ as the action function and one hidden layer with $2p$ neurons,
see Figure \ref{TNNstructure}.
\end{theorem}

TNN-based machine learning method in this paper will adaptively select $p$ rank-one functions by training process.
From the approximation result in Theorem \ref{theorem_aprrox_rate}, when the target function belongs to $H_{\rm mix}^{t,\ell}(\Omega)$,
there exists a TNN with $p\sim\mathcal O(\varepsilon^{-(m-\ell-t)})$ such that the accuracy is $\varepsilon$.

Note that our analysis is based on the special activation function $\sin(x)$ and space $H_{\rm mix}^{t,\ell}(\Omega)$.
General approximation results in the $H^m(\Omega)$-norm for the functions in Barron space by the FNNs with $d$-dimensional input and general activation functions are discussed in \cite{SiegelXu}.

\section{Quadrature scheme for TNN}\label{Section_Integration}
In this section, we focus on the numerical integration of polynomial composite function of TNN and its derivatives.
Our main theorem shows that the application of TNN can
bring a significant reduction of the computational complexity for the related numerical integration.
For convenience, we first introduce the following sets of multiple indices
\begin{eqnarray*}
\mathcal B&:=&\left\{\beta=(\beta_1,\cdots,\beta_d)\in\mathbb N_0^d\ \Big|\ |\beta|\coloneqq\sum_{i=1}^d\beta_i\leq m \right\},\\
\mathcal A&:=&\left\{\alpha=(\alpha_\beta)_{\beta\in\mathcal B}\in\mathbb N_0^{|\mathcal B|}\
\Big|\ |\alpha|\coloneqq\sum_{\beta\in\mathcal B}\alpha_\beta\leq k \right\},
\end{eqnarray*}
where $\mathbb N_0$ denotes the set of all non-negative integers,  $m$ and $k$ are two positive integers,
$|\mathcal B|$ and $|\mathcal A|$ denote the cardinal numbers of $\mathcal B$ and $\mathcal A$, respectively.

For example, if $d=2$ and $m=1$, the set $\mathcal B$ is follows
\begin{eqnarray}\label{Definition_B}
\mathcal B = \Big\{(0,0), (1,0), (0,1)\Big\},
\end{eqnarray}
which has $3$ members.  Then $\mathcal A$ is a triple-index set. If $k=2$, the set $\mathcal A$ can be described
as follows
\begin{eqnarray}\label{Definition_A}
\mathcal A =\Big\{(0,0,0), (1,0,0), (0,1,0), (0,0,1), (2,0,0), (0,2,0), (0,0,2), (1,1,0), (1,0,1), (0,1,1)\Big\}.
\end{eqnarray}
Each index in $\mathcal A$ corresponds to the member in $\mathcal B$.
For example, we can simply take the order for the members in the set $\mathcal B$
as that of (\ref{Definition_B}). Then the member $\alpha=(1,1,0)\in \mathcal A$ indicates
that $\alpha_{(0,0)}=1$, $\alpha_{(1,0)}=1$ and $\alpha_{(0,1)}=0$, the member
$\alpha = (2,0,0)\in \mathcal A$ indicates $\alpha_{(0,0)}=2$, $\alpha_{(1,0)}=0$ and $\alpha_{(0,1)}=0$.

In this paper, we focus on the high-dimensional cases where $m\ll d$ and $k\ll d$.
Simple calculation leads to the following equations
$$|\mathcal B|=\sum_{j=0}^m\binom{j+d-1}{j}, \ \ \ \ |\mathcal A|=\sum_{j=0}^k\binom{j+|\mathcal B|-1}{j}.$$
By further estimation, we know that the scales of magnitudes of $|\mathcal B|$ and $|\mathcal A|$ are
$\mathcal O\big((d+m)^m\big)$
and $\mathcal O\big(((d+m)^m+k)^k\big)$, respectively.

Here and after, the parameter $\theta$ in (\ref{def_TNN}) will be omitted for brevity
without confusion. The activation function of TNN needs
to be smooth enough such  that $\Psi(x)$ has partial derivatives up to order $m$.
Here, we assume $F(x)$ includes the $k$-degree complete polynomial of $d$-dimensional TNN and
its partial derivatives up to order $m$
that can be expressed as follows
\begin{eqnarray}\label{def_F(x)}
F(x)=\sum_{\alpha\in\mathcal A}A_{\alpha}(x)\prod_{\beta\in\mathcal B}
\left(\frac{\partial^{|\beta|}\Psi(x)}{\partial x_1^{\beta_1}
\cdots\partial x_d^{\beta_d}}\right)^{\alpha_\beta},
\end{eqnarray}
where the coefficient $A_\alpha(x)$ is given by the following expansion such that the rank of $A_\alpha(x)$
is not greater
than $q$ in the tensor product space $L^2(\Omega_1)\otimes\cdots\otimes L^2(\Omega_d)$
\begin{eqnarray}\label{def_A_alpha}
A_\alpha(x)=\sum_{\ell=1}^q B_{1,\ell,\alpha}(x_1)B_{2,\ell,\alpha}(x_2)\cdots B_{d,\ell,\alpha}(x_d).
\end{eqnarray}
Here $B_{i,\ell,\alpha}(x_i)$ denotes the one-dimensional function in $L^2(\Omega_i)$ for $i=1, \cdots, d$
and  $\ell=1,\cdots, q$.
When using neural networks to solve PDEs, we always need to do
the high-dimensional integration $\int_\Omega F(x)dx$. If $\Psi(x)$ is a FNN,
$\int_\Omega F(x)dx$ can only be treated as a direct $d$-dimensional numerical integration,
which requires exponential scale of computational work according to the dimension $d$.
In practical applications, it is well known that the high-dimensional FNN functions can only be integrated
by the Monte-Carlo method with a low accuracy.
Different from FNN, we will show that the high-dimensional integration $\int_\Omega F(x)dx$ for the TNN
can be implemented by the normal numerical quadrature with the polynomial scale
of computational work with respect to the dimension $d$.  This means that the TNN
can cope with the curse of dimensionality in some sense.
The key idea to reduce the computational complexity of the numerical integration $\int_\Omega F(x)dx$ is
that we can decompose the TNN function $F(x)$ into a tensor product structure.

In order to implement  the  decomposition, for each
$\alpha = (\alpha_1,\cdots,\alpha_{|\mathcal B|})\in\mathcal A$, we  give the following definition
\begin{eqnarray*}
\mathcal B_\alpha:=\Big\{\beta=\left(\beta_1,\cdots,\beta_d\right)\in\mathcal B
\ \big|\ \alpha_\beta\geq 1\Big\}.
\end{eqnarray*}
For example, when the sets $\mathcal B$ and $\mathcal A$ are defined by (\ref{Definition_B})
and (\ref{Definition_A}), respectively,
the set $\mathcal B_\alpha$ corresponding to the member $\alpha = (1,1,0) \in \mathcal A$
in (\ref{Definition_A})
can be described as follows
\begin{eqnarray*}
\mathcal B_\alpha =\Big\{(0,0),  (1,0)\Big\}.
\end{eqnarray*}
By the definition of the index set $\mathcal A$, we can deduce that $|\mathcal B_\alpha|\leq k$ for any $\alpha\in\mathcal A$.

Since  $\Psi(x)$ has the TNN structure (\ref{def_TNN}), the cumprod can be further decomposed as
\begin{eqnarray}\label{eq_decomposition_prod}
&&\prod_{\beta\in\mathcal B_\alpha}
\left(\frac{\partial^{|\beta|}\Psi(x)}{\partial x_1^{\beta_1}
\cdots\partial x_d^{\beta_d}}\right)^{\alpha_\beta}
=\prod_{\beta\in\mathcal B_\alpha}\left(\frac{\partial^{|\beta|}\sum\limits_{j=1}^p\phi_{1,j}(x_1)
\cdots\phi_{d,j}(x_d)}{\partial x_1^{\beta_1}
\cdots\partial x_d^{\beta_d}}\right)^{\alpha_\beta}\nonumber\\
&&=\prod_{\beta\in\mathcal B_\alpha}\left(\sum\limits_{j=1}^p
\frac{\partial^{\beta_1}\phi_{1,j}(x_1)}{\partial x_1^{\beta_1}}
\cdots\frac{\partial^{\beta_d}\phi_{d,j}(x_d)}{\partial x_d^{\beta_d}} \right)^{\alpha_\beta}\nonumber\\
&&=\prod_{\beta\in\mathcal B_\alpha}\sum_{1\leq j_1,\cdots,j_{\alpha_\beta}\leq p}\left(\frac{\partial^{\beta_1}\phi_{1,j_1}(x_1)}{\partial x_1^{\beta_1}}
\cdots\frac{\partial^{\beta_1}\phi_{1,j_{\alpha_{\beta}}}(x_1)}{\partial x_1^{\beta_1}}\right) \cdots\left(\frac{\partial^{\beta_d}\phi_{d,j_1}(x_d)}{\partial x_d^{\beta_d}}
\cdots\frac{\partial^{\beta_d}\phi_{d, j_{\alpha_{\beta}}}(x_d)}{\partial x_d^{\beta_d}}\right)
\ \ \ \ \ \nonumber\\
&&=\prod_{\beta\in\mathcal B_\alpha}\sum_{1\leq j_1,\cdots,j_{\alpha_\beta}\leq p}\left(\prod_{\ell=1}^{\alpha_\beta}\frac{\partial^{\beta_1}\phi_{1,j_\ell}(x_1)}{\partial x_1^{\beta_1}}\right)\cdots\left(\prod_{\ell=1}^{\alpha_\beta}
\frac{\partial^{\beta_d}\phi_{d,j_\ell}(x_d)}{\partial x_d^{\beta_d}}\right) \nonumber\\
&&=\sum_{\substack{\beta\in\mathcal B_\alpha, \ell=1,\cdots,\alpha_\beta, \\ 1\leq j_{\beta,\ell}\leq p}}
\left(\prod_{\beta\in\mathcal B_\alpha}\prod_{\ell=1}^{\alpha_\beta}
\frac{\partial^{\beta_1}\phi_{1,j_{\beta,\ell}}(x_1)}{\partial x_1^{\beta_1}}\right)\cdots
\left(\prod_{\beta\in\mathcal B_\alpha}\prod_{\ell=1}^{\alpha_\beta}
\frac{\partial^{\beta_d}\phi_{d,j_{\beta,\ell}}(x_d)}{\partial x_d^{\beta_d}}\right).
\end{eqnarray}
With the help of expansion (\ref{eq_decomposition_prod}),  we can give the following expansion for $F(x)$
\begin{eqnarray}\label{Expansion_Fx}
F(x) &=& \sum_{\alpha\in\mathcal A}
\left(\sum_{\ell=1}^qB_{1,\ell,\alpha}(x_1)\cdots B_{d,\ell,\alpha}(x_d)\right)\nonumber\\
&&\ \ \ \cdot\sum_{\substack{\beta\in\mathcal B_\alpha, \ell=1,\cdots,\alpha_\beta,\\
1\leq j_{\beta,\ell}\leq p}}
\left(\prod_{\beta\in\mathcal B_\alpha}\prod_{\ell=1}^{\alpha_\beta}
\frac{\partial^{\beta_1}\phi_{1,j_{\beta,\ell}}(x_1)}{\partial x_1^{\beta_1}}\right)
\cdots\left(\prod_{\beta\in\mathcal B_\alpha}\prod_{\ell=1}^{\alpha_\beta}\frac{\partial^{\beta_d}\phi_{d,j_{\beta,\ell}}(x_d)}{\partial x_d^{\beta_d}}\right)\nonumber\\
&=&\sum_{\alpha\in\mathcal A}\sum_{\ell=1}^q\sum_{\substack{\beta\in\mathcal B_\alpha,
\ell=1,\cdots,\alpha_\beta,\\ 1\leq j_{\beta,\ell}\leq p}}
\left(B_{1,\ell,\alpha}(x_1)\prod_{\beta\in\mathcal B_\alpha}\prod_{\ell=1}^{\alpha_\beta}
\frac{\partial^{\beta_1}\phi_{1,j_{\beta,\ell}}(x_1)}{\partial x_1^{\beta_1}}\right)\nonumber\\
&&\ \quad\quad \ \cdots\left(B_{d,\ell,\alpha}(x_d)
\prod_{\beta\in\mathcal B_\alpha}\prod_{\ell=1}^{\alpha_\beta}
\frac{\partial^{\beta_d}\phi_{d,j_{\beta,\ell}}(x_d)}{\partial x_d^{\beta_d}}\right).
\end{eqnarray}
Based on the decomposition (\ref{Expansion_Fx}), we have the following splitting scheme for the integration
$\int_\Omega F(x)dx$
\begin{eqnarray}\label{Integration_Expansions}
\int_\Omega F(x)dx &=&\sum_{\alpha\in\mathcal A}\sum_{\ell=1}^q\sum_{\substack{\beta\in\mathcal B_\alpha, \ell=1,\cdots,\alpha_\beta,\\ 1\leq j_{\beta,\ell}\leq p}}
\int_{\Omega_1}\left(B_{1,\ell,\alpha}(x_1)\prod_{\beta\in\mathcal B_\alpha}\prod_{\ell=1}^{\alpha_\beta}
\frac{\partial^{\beta_1}\phi_{1,j_{\beta,\ell}}(x_1)}{\partial x_1^{\beta_1}}\right)dx_1\nonumber\\
&&\ \quad\quad \ \cdots\int_{\Omega_d}\left(B_{d,\ell,\alpha}(x_d)
\prod_{\beta\in\mathcal B_\alpha}\prod_{\ell=1}^{\alpha_\beta}
\frac{\partial^{\beta_d}\phi_{d,j_{\beta,\ell}}(x_d)}{\partial x_d^{\beta_d}}\right)dx_n.
\end{eqnarray}
Now, we come to introduce the detailed numerical integration method for the TNN function $F(x)$.
Without loss of generality, for $i=1, \cdots, d$, we choose $N_i$
Gauss points $\{x_i^{(n_i)}\}_{n_i=1}^{N_i}$ and the corresponding
weights $\{w_i^{(n_i)}\}_{n_i=1}^{N_i}$
for the $i$-th dimensional domain $\Omega_i$, and denote $N=\max\{N_1,\cdots,N_d\}$ and
$\underline N = \min\{N_1,\cdots,N_d\}$. Introducing the index
$n=(n_1,\cdots,n_d)\in\mathcal N\coloneqq\{1,\cdots,N_1\}\times\cdots\times\{1,\cdots,N_d\}$,
then the Gauss points and their corresponding weights for the integration (\ref{Integration_Expansions})
can be expressed as follows
\begin{eqnarray}\label{def_Gauss}
\left.
\begin{array}{rcl}
\Big\{x^{(n)}\Big\}_{n\in\mathcal N}&=&\left\{\left\{x_1^{(n_1)}\right\}_{n_1=1}^{N_1},\ \left\{x_2^{(n_2)}\right\}_{n_2=1}^{N_2},\ \cdots,\ \left\{x_d^{(n_d)}\right\}_{n_d=1}^{N_d}\right\},\\
\Big\{w^{(n)}\Big\}_{n\in\mathcal N}&=&\left\{\left\{w_1^{(n_1)}\right\}_{n_1=1}^{N_1}\times \left\{w_2^{(n_2)}\right\}_{n_2=1}^{N_2}\times \cdots, \times  \left\{w_d^{(n_d)}\right\}_{n_d=1}^{N_d}\right\}.
\end{array}
\right.
\end{eqnarray}
Then from    (\ref{def_F(x)}) and (\ref{def_A_alpha}), the numerical integration $\int_\Omega F(x)dx$
can be computed as follows:
\begin{eqnarray}\label{eq_I_simple_form}
\int_\Omega F(x)dx \approx\sum_{n\in\mathcal N}w^{(n)}\sum_{\alpha\in\mathcal A}
\left(\sum_{\ell=1}^qB_{1,\ell,\alpha}(x_1^{(n_1)})\cdots B_{d,\ell,\alpha}(x_d^{(n_d)})\right)
\prod_{\beta\in\mathcal B_\alpha}
\left(\frac{\partial^{|\beta|}\Psi(x^{(n)})}{\partial x_1^{\beta_1}
\cdots\partial x_d^{\beta_d}}\right)^{\alpha_\beta}.
\end{eqnarray}
Fortunately, with the help of expansions (\ref{eq_decomposition_prod}) and (\ref{Integration_Expansions}),
we can give the following splitting numerical integration scheme for $\int_\Omega F(x)dx$:
\begin{eqnarray}\label{eq_I_tensor_form}
\int_\Omega F(x)dx&\approx&
\sum_{\alpha\in\mathcal A}\sum_{\ell=1}^q\sum_{\substack{\beta\in\mathcal B_\alpha,
\ell=1,\cdots,\alpha_\beta,\\ 1\leq j_{\beta,\ell}\leq p}}
\left(\sum_{n_1=1}^{N_1}w_1^{(n_1)}B_{1,\ell,\alpha}(x_1^{(n_1)})\prod_{\beta\in\mathcal B_\alpha}\prod_{\ell=1}^{\alpha_\beta}
\frac{\partial^{\beta_1}\phi_{1,j_{\beta,\ell}}(x_1^{(n_1)})}{\partial x_1^{\beta_1}}\right)\nonumber\\
&&\ \ \quad \ \cdots\left(\sum_{n_d=1}^{N_d}w_d^{(n_d)}B_{d,\ell,\alpha}(x_d^{(n_d)})
\prod_{\beta\in\mathcal B_\alpha}\prod_{\ell=1}^{\alpha_\beta}
\frac{\partial^{\beta_d}\phi_{d,j_{\beta,\ell}}(x_d^{(n_d)})}{\partial x_d^{\beta_d}}\right).
\end{eqnarray}
The quadrature scheme (\ref{eq_I_tensor_form}) decomposes the
high-dimensional integration $\int_\Omega F(x)dx$ into to a series of 1-dimensional
integration, which is the main contribution of this paper.
Due to the simplicity of the 1-dimensional integration, the scheme  (\ref{eq_I_tensor_form})
can reduce the computational work of the high-dimensional integration for the $d$-dimensional
function $F(x)$ to the
polynomial scale of dimension $d$.
Theorem \ref{theorem_Gauss} gives the corresponding result.

\begin{theorem}\label{theorem_Gauss}
Assume that the function $F(x)$ is defined as (\ref{def_F(x)}), where the coefficient $A_\alpha(x)$
has the expansion (\ref{def_A_alpha}). Employ Gauss quadrature points and corresponding weights
(\ref{def_Gauss}) to $F(x)$ on the $d$-dimensional tensor product domain $\Omega$.
If the function $\Psi(x)$ involved in the function $F(x)$ has TNN form (\ref{def_TNN}),
the efficient quadrature scheme (\ref{eq_I_tensor_form}) is equivalent to (\ref{eq_I_simple_form})
and has $2\underline N$-th order accuracy. Let $T_1$ denote the computational complexity
for the $1$-dimensional function evaluation operations. The computational complexity can be bounded by
$\mathcal O\big(dqT_1k^2p^k\big((d+m)^m+k\big)^kN\big)$, which is the polynomial scale of the dimension $d$.
\end{theorem}
\begin{proof}
First, we point out that the number of $j_{\beta,\ell}$ in the last summation of
(\ref{eq_decomposition_prod}) is no more than $k$.
This result can be easily proved by the following inequality
\begin{eqnarray*}
\sum_{\beta\in\mathcal B_\alpha}\alpha_\beta=|\alpha|\leq k.
\end{eqnarray*}
Then, by direct calculation, the computational complexity of (\ref{eq_I_tensor_form})
can be bounded by $\mathcal O\big(dqT_1k^2p^k((d+m)^m+k)^kN\big)$.
Since the 1-dimensional integration with $N_i$ Gauss points has $2N_i$-th order accuracy
and the equivalence of (\ref{eq_I_tensor_form}) and (\ref{eq_I_simple_form}),
both quadrature schemes (\ref{eq_I_tensor_form}) and (\ref{eq_I_simple_form})
have the $2\underline N$-th order accuracy. Then the proof is complete.
\end{proof}
\begin{remark}
Theorem \ref{theorem_Gauss} considers using 1-dimensional Gauss points to compute $d$-dimensional integrations.
Other types of 1-dimensional quadrature schemes can also be employed to do the $d$-dimensional integration and have similar results. In numerical examples, we decompose each $\Omega_i$ into subintervals with mesh size $h$ and choose $N_i$ 1-dimensional Gauss points in each subinterval. Then the deduced $d$-dimensional quadrature scheme has accuracy $\mathcal O(h^{2\underline N}/(2\underline N)!)$, where the included constant depends on the smoothness of $F(x)$.
\end{remark}
If $\Psi(x)$ does not have the tensor form, for example, $\Psi(x)$ is a $d$-dimensional FNN
and use the same quadrature scheme (\ref{eq_I_tensor_form}),  the computational complexity
is $\mathcal O\big((dqT_1+kT_d)((d+m)^m+k)^kN^d\big)$, where $T_d$ denotes the complexity
of the $d$-dimensional function evaluation operations.

\section{Solving high-dimensional eigenvalue problem by TNN}\label{Section_Eigenvalue}
This section is devoted to discussing the applications of TNNs to the numerical
solution of the high-dimensional second order elliptic eigenvalue problems.
For simplicity, we are concerned with the following model problem:
\begin{equation}\label{def_EVP_strong}
\left\{\begin{aligned}
-\Delta u +vu&=\lambda u, & & \text {in}\ \ \Omega, \\
u &=0, & & \text {on}\ \partial \Omega,
\end{aligned}\right.
\end{equation}
where $\Omega=\Omega_1\times\cdots\times\Omega_d$, each $\Omega_i=(a_i,b_i),i=1,\cdots,d$
is a bounded interval in $\mathbb R$, $v\in L^2(\Omega)$ is a potential function.
We assume that the rank of $v$ is finite in the tensor product space
$L^2(\Omega_1)\otimes\cdots\otimes L^2(\Omega_d)$.
The potential function $v$ often occurs in quantum mechanics problems. In this paper, we consider the
following cases: \\
zero function
\begin{equation}\label{potential_well}
v(x)= 0, \ \ \ \ {\rm in}\ \ \Omega,
\end{equation}
harmonic oscillator
\begin{equation}\label{harmonic oscillator}
v(x)=\sum_{i=1}^dx_i^2,\ \ \ {\rm in}\ \ \Omega, 
\end{equation}
and coupled oscillator
\begin{eqnarray}\label{coupled harmonic}
v(x)=\sum_{i=1}^dx_i^2-\sum_{i=1}^{d-1}x_ix_{i+1},
\end{eqnarray}
and the Column potential for Schr\"{o}dinger equation.

In quantum mechanics, the eigenvalue problem (\ref{def_EVP_strong}) with the potential
function (\ref{potential_well})
is the Schr\"{o}dinger equation with infinite potential well.
The eigenvalue problem with the potential (\ref{harmonic oscillator}) comes from the truncation of the
Schr\"{o}dinger equation with the harmonic oscillator potential which is defined in the whole space.
The more complicated eigenvalue problem with the potential (\ref{coupled harmonic})
describes the system of chains of $d$ coupled harmonic oscillators which is described
in detail in \cite{CHO}.

There is a well-known variational principle or minimum theorem of the eigenvalue
problem (\ref{def_EVP_strong}) for the smallest eigenpair $(\lambda,u)$:
\begin{eqnarray}\label{minimum theorem}
\lambda=\min_{w\in H_0^1(\Omega)}\mathcal R(w)=\min_{w\in H_0^1(\Omega)}\frac{\int_\Omega |\nabla w|^2dx
+\int_\Omega vw^2dx}{\int_\Omega w^2dx},
\end{eqnarray}
where $\mathcal R(w)$ denotes the Rayleigh quotient for the function $w\in H_0^1(\Omega)$.

In order to solve the eigenvalue  problem (\ref{def_EVP_strong}), we build a TNN structure $\Psi(x;\theta)$
which is defined by (\ref{def_TNN}), and denote the set of all possible values of $\theta$ as $\Theta$.
In order to avoid the penalty on boundary condition,
we simply use the method in \cite{GuWangYang} to treat the Dirichlet boundary condition.
This method is firstly proposed in \cite{LagarisLikasFotiadis,LagarisLikasPapageorgiou}.
For $i=1,\cdots, d$, the $i$-th subnetwork $\phi_i(x_i;\theta_i)$ is defined as follows:
\begin{eqnarray*}\label{bd_subnetwork}
\phi_i(x_i;\theta_i)&\coloneqq&(x_i-a_i)(b_i-x_i)\widehat\phi_i(x_i;\theta_i)\nonumber\\
&=&\big((x_i-a_i)(b_i-x_i)\widehat\phi_{i,1}(x_i;\theta_i),\cdots,(x_i-a_i)(b_i-x_i)
\widehat\phi_{i,p}(x_i;\theta_i)\big)^T,
\end{eqnarray*}
where $\widehat\phi_i(x_i;\theta_i)$ is a FNN from $\mathbb R$ to $\mathbb R^p$ with sufficient
smooth activation functions, such that $\Psi(x;\theta)\in H_0^1(\Omega)$.

The trial function set $V$ is modeled by the TNN structure $\Psi(x;\theta)$
where parameters $\theta$ take all the possible values and it is obvious that $V\subset H_0^1(\Omega)$.
The solution and the parameters $(\lambda^*,\Psi(x;\theta^*))$ of the following
optimization problem are approximations to the smallest eigenpair:
\begin{eqnarray}\label{approx_opt}
\lambda^*=\min_{\Psi(x;\theta)\in V}\mathcal R(\Psi(x;\theta))
=\min_{\theta\in\Theta}\frac{\int_\Omega |\nabla \Psi(x;\theta)|^2dx
+\int_\Omega v(x)\Psi^2(x;\theta)dx}{\int_\Omega \Psi^2(x;\theta)dx}=\mathcal R(\Psi(x;\theta^*)).
\end{eqnarray}
Note that all integrands of the numerator and the denominator of (\ref{approx_opt})
have the form (\ref{def_F(x)}). With the help of Theorem \ref{theorem_Gauss},
we can implement these numerical integrations by scheme (\ref{eq_I_tensor_form}) with the computational
work being bounded by the polynomial  scale of dimension $d$.
We choose Gauss points and their corresponding weights
which are defined by (\ref{def_Gauss}) to compute these integrations,
and define the loss function as follows:
\begin{eqnarray}\label{loss}
L(\theta)\coloneqq\frac{\sum_{n\in\mathcal N}w^{(n)}|\nabla\Psi(x^{(n)},\theta)|^2
+\sum_{n\in\mathcal N}w^{(n)}v(x^{(n)})\Psi^2(x^{(n)};\theta)}{\sum_{n\in\mathcal N}\Psi^2(x^{(n)};\theta)}.
\end{eqnarray}
In this paper, the gradient descent (GD) method is adopted to minimize the loss function $L(\theta)$.
The GD scheme can be described as follows:
\begin{eqnarray}\label{gd}
\theta^{(k+1)}=\theta^{(k)}-\eta\nabla L(\theta^{(k)}),
\end{eqnarray}
where $\theta^{(k)}$ denotes the parameters after the $k$-th GD step, $\eta$ is the learning rate (step size).

Different from the general FNN-based machine learning method, we use the fixed quadrature points
$\{x^{(n)}\}_{n\in\mathcal N}$ to do the numerical integration in this paper.
Using the fixed quadrature points for FNN, the computational
work for the numerical integration will depend exponentially on the dimension $d$.
In order to avoid the ``curse of dimensionality'', in the numerical implementation for
solving high-dimensional PDEs by FNN-based method,
the stochastic gradient descent (SGD) method \cite{KingmaAdam} with Monte-Carlo
integration are always used \cite{EYu}. The application of random sampling quadrature
points always leads to the low accuracy and instability convergence for the FNN method.

Fortunately, based on TNN structure in the loss function (\ref{loss}),  Theorem \ref{theorem_Gauss}
shows that the numerical integration does not encounter ``curse of dimensionality''
since the computational work can be bounded by the polynomial scale of dimension $d$.
This is the reason we can use GD  method  to solve the optimization problem (\ref{approx_opt})
instead of SGD in this paper. That is to say, using all quadrature points to
implement the integration and the GD step (\ref{gd}) in TNN-based machine learning are reasonable.
With the help of the high accuracy of the tensor product with Gauss points and Theorem \ref{theorem_approximation},
the high accuracy of the TNN-based method can be guaranteed.


Although in this paper, we simply choose a fixed rank $p$ in our numerical examples, it is worth mentioning that, by adding columns to the weight matrices of the output layer in each subnetwork, we can transfer weights from the old TNN to the new one to improve the rank $p$. We can stop this process when the accuracy improvement is small enough.
Choosing the rank $p$ by the computable posterior error estimation and the corresponding transfer learning framework will be presented in our future work.

\begin{remark}
In this section, we are mainly talking about solving high-dimensional eigenvalue problems by TNN.
It is worth mentioning that TNN structure can be applied naturally to solve PDEs
with different types of loss functions.
We will do a preliminary test in our numerical examples.
\end{remark}

\section{Numerical examples}\label{Section_Numerical}
In this section, we provide several examples to validate the efficiency and accuracy of the  TNN-based
machine learning method proposed in this paper. The first two examples are used to
demonstrate the high accuracy of the TNN method for high-dimensional problems.
We will explore the effect of the vital hyperparameter $p$ on the accuracy in the third example
where the ground state energy may not be exactly represented by a finite-rank CP decomposition.
The fourth example for the ground state of a helium atom which comes from the real physical problem
is used to give an illuminating way to deal with the problem that does not satisfy
the assumption (\ref{def_A_alpha}). Note that in the fourth example, the potential cannot
be exactly expressed as a CP decomposition of
finite rank, this makes the loss function no longer satisfy the assumption (\ref{def_A_alpha}).
In the last example, we solve a boundary value problem with Neumann boundary condition to
show the efficiency of TNN for solving high-dimensional PDEs.

In order to show the convergence behavior and accuracy of eigenfunction approximations by TNN,
we define the $L^2(\Omega)$ projection operator $\mathcal P:H_0^1(\Omega)
\rightarrow {\rm span}\{\Psi(x;\theta^*)\}$ as follows:
\begin{eqnarray*}
\left\langle\mathcal Pu,v\right\rangle_{L^2}=\left\langle u,v\right\rangle_{L^2}\coloneqq\int_\Omega uvdx,
\ \ \ \forall v\in {\rm span}\{\Psi(x;\theta^*)\}\ \ {\rm for}\ u\in H_0^1(\Omega).
\end{eqnarray*}
And we define the $H^1(\Omega)$ projection operator $\mathcal Q:H_0^1(\Omega)\rightarrow {\rm span}\{\Psi(x;\theta^*)\}$
as follows:
\begin{eqnarray*}
\left\langle\mathcal Qu,v\right\rangle_{H^1}=\left\langle u,v\right\rangle_{H^1}\coloneqq\int_\Omega\nabla u\cdot\nabla vdx,
\ \ \ \forall v\in {\rm span}\{\Psi(x;\theta^*)\}\ \ {\rm for}\ u\in H_0^1(\Omega).
\end{eqnarray*}
Then we define the following errors for the approximated eigenvalue $\lambda^*$ and eigenfunction $\Psi(x;\theta^*)$
\begin{eqnarray*}\label{relative_errors}
e_\lambda\coloneqq\frac{|\lambda^*-\lambda|}{|\lambda|},\ \ \ e_{L^2}\coloneqq\frac{\|u-\mathcal Pu\|_{L^2(\Omega)}}{\|u\|_{L^2(\Omega)}},
\ \ \ e_{H^1}\coloneqq\frac{\left|u-\mathcal Qu\right|_{H^1(\Omega)}}{\left|u\right|_{H^1(\Omega)}},
\end{eqnarray*}
in all eigenvalue examples.
As for Neumann boundary value problem, we define the following errors for the approximated solution $\Psi(x;\theta^*)$
\begin{eqnarray*}
\widehat e_{L^2}\coloneqq\frac{\|u-\Psi(x;\theta^*)\|_{L^2(\Omega)}}{\|f\|_{L^2(\Omega)}},
\ \ \ \widehat e_{H^1}\coloneqq\frac{\left|u-\Psi(x;\theta^*)\right|_{H^1(\Omega)}}{\left|f\right|_{H^1(\Omega)}}.
\end{eqnarray*}
Here $\|\cdot\|_{L^2}$ and $|\cdot|_{H^1}$ denote $L^2(\Omega)$ norm and $H^1(\Omega)$ seminorm, respectively.
These relative errors are often used to test numerical methods for eigenvalue problems and PDEs.
We use the quadrature scheme (\ref{eq_I_tensor_form}) to compute $e_{L^2}$ and $e_{H^1}$
with the same Gauss points and weights as computing the loss functions
if the rank of the exact solution $u(x)$ is finite in the tensor product space
$L^2(\Omega_1)\otimes\cdots\otimes L^2(\Omega_d)$, otherwise we only report $e_\lambda$.
With the help of Theorem \ref{theorem_Gauss} and Gauss quadrature points,
the high efficiency and accuracy
for computing $e_{L^2}$ and $e_{H^1}$ can be guaranteed.

In implementation, we train the networks by Adam optimizer \cite{KingmaAdam} and use automatic differentiation for derivatives in PyTorch.
\revise{In this chapter, all examples is using sine function $\sin(x)$ as activation function.}

\subsection{Laplace eigenvalue problem}\label{Section_laplace}
In the first example, the potential function is defined as (\ref{potential_well}) with the computational domain $\Omega=[0,1]^d$.
Then the exact smallest eigenvalue and eigenfunction are
\begin{eqnarray*}
\lambda = d \pi^2, \ \ \ \ \ u(x)=\prod_{i=1}^d\sin(\pi x_i).
\end{eqnarray*}
First, we test high-dimensional cases with $d=5,10,20$, respectively.
Quadrature scheme for TNN is obtained by decomposing each $\Omega_i$ ($i=1, \cdots, d$) into $10$ equal subintervals and choosing $16$ Gauss points on each subinterval.
The Adam optimizer is employed with a learning rate of 0.003 to train \revise{a TNN with $p=10$. Each subnetwork of TNN is a FNN with two hidden layers and each hidden layer has 50 hidden neurons, see Figure \ref{TNNstructure}.} Figure \ref{fig_laplace_high} shows the relative errors $e_\lambda$, $e_{L^2}$ and $e_{H^1}$ versus the number of epochs. The final relative errors after 100000 epochs are reported in Table \ref{table_laplace_high} for different dimensional cases. Then we can find that the TNN method has almost the same convergence behaviors for different dimensions.

\begin{figure}[htb]
\centering
\includegraphics[width=4cm,height=4cm]{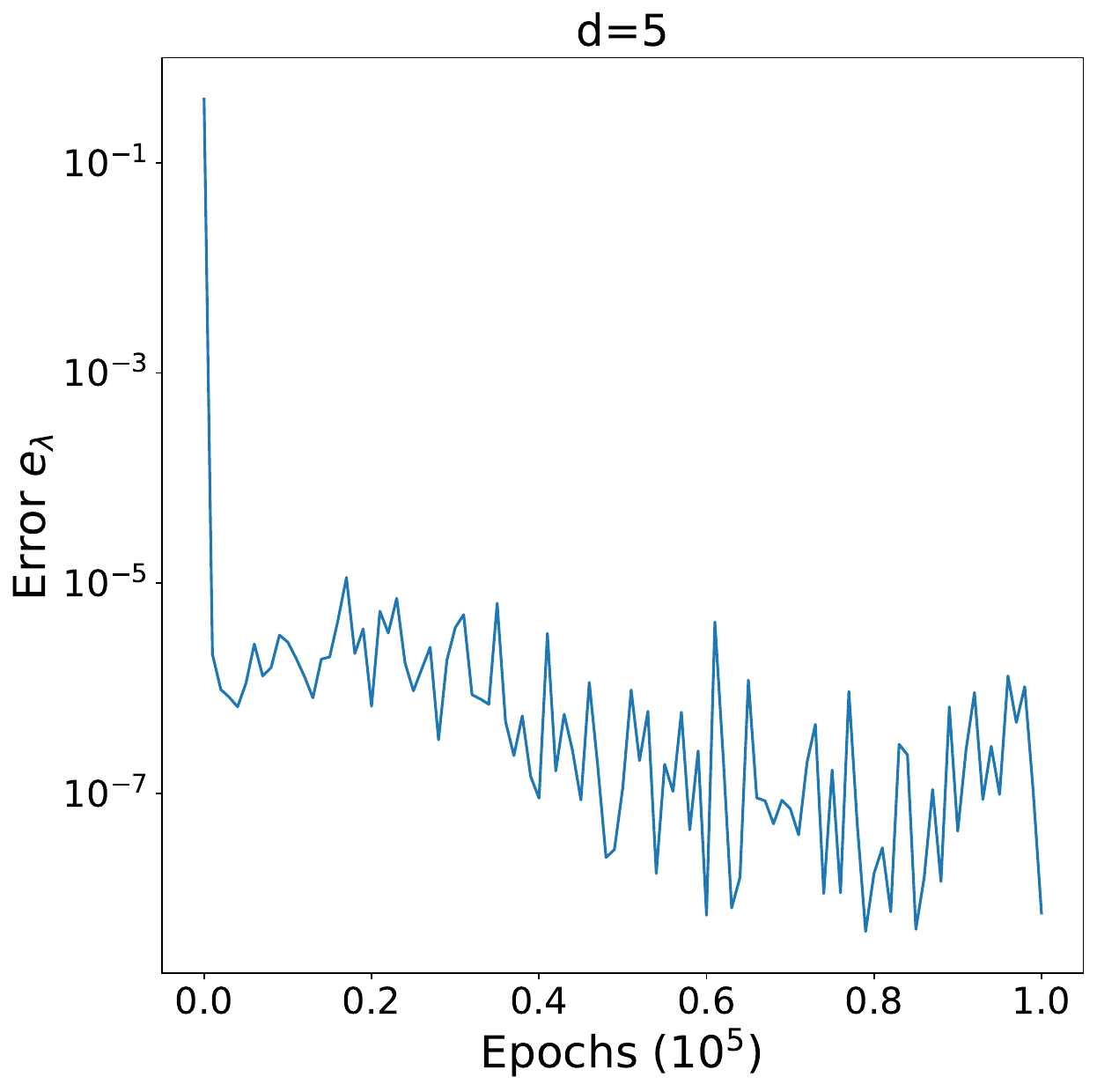}
\includegraphics[width=4cm,height=4cm]{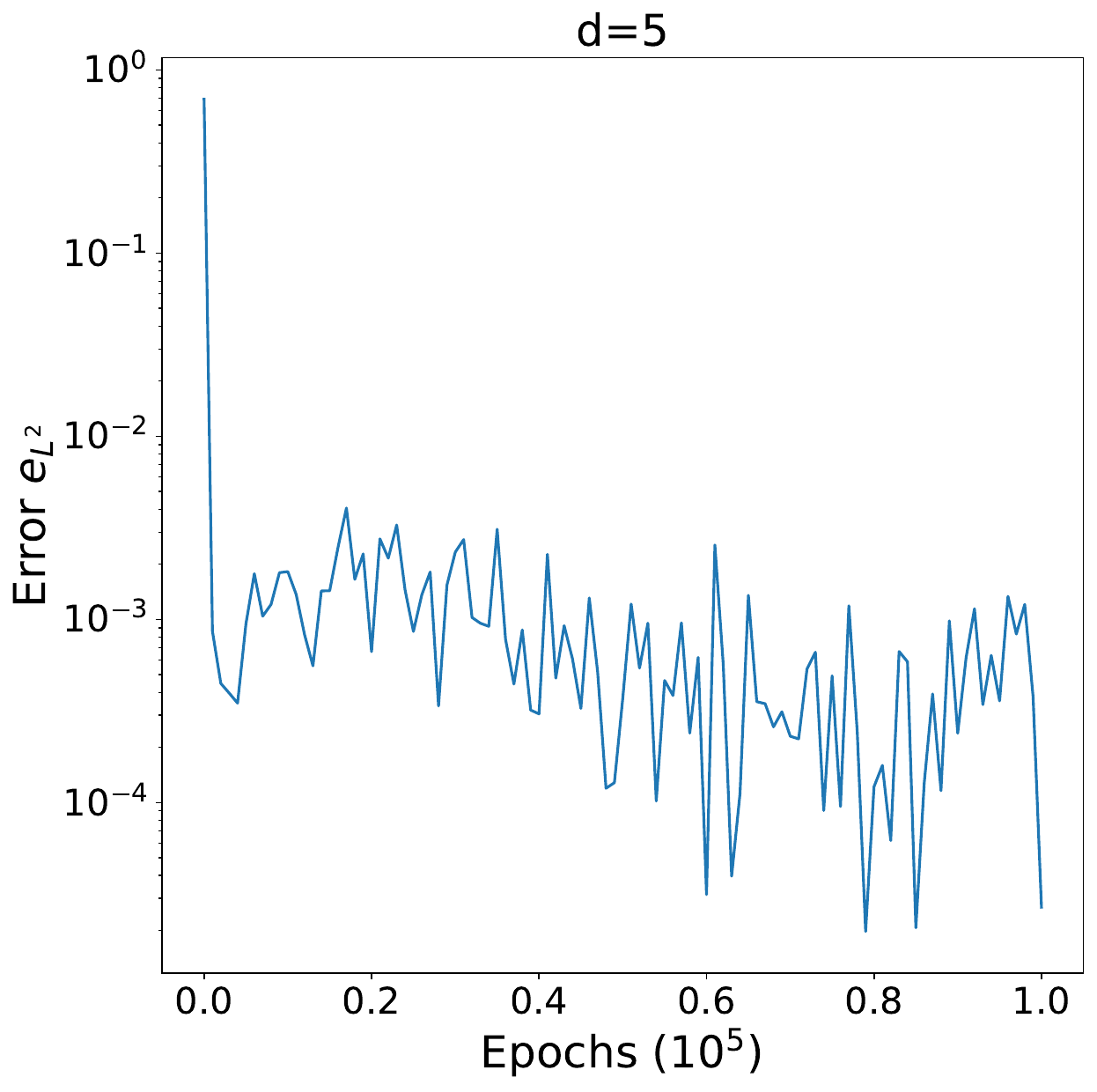}
\includegraphics[width=4cm,height=4cm]{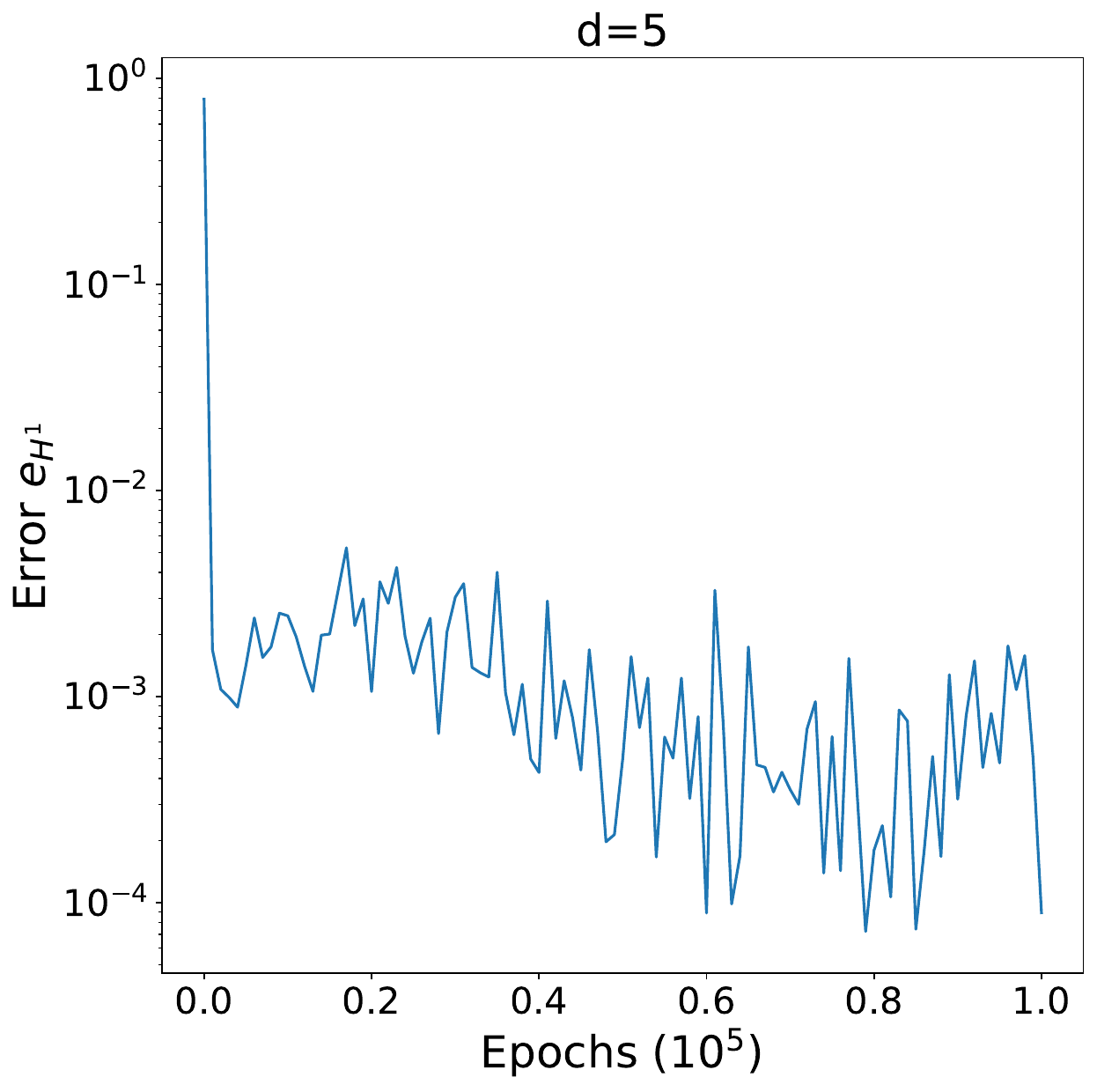}\\
\includegraphics[width=4cm,height=4cm]{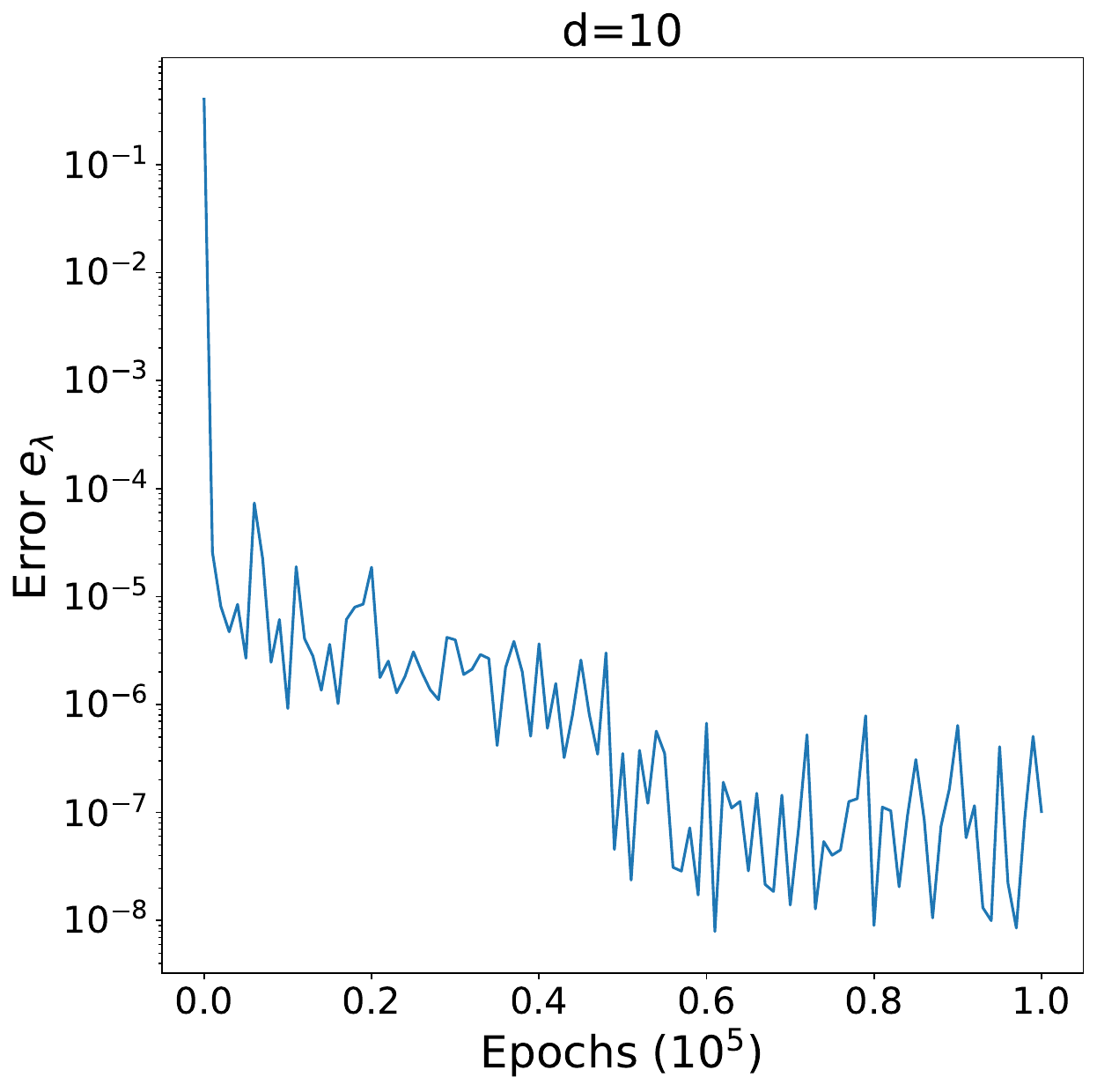}
\includegraphics[width=4cm,height=4cm]{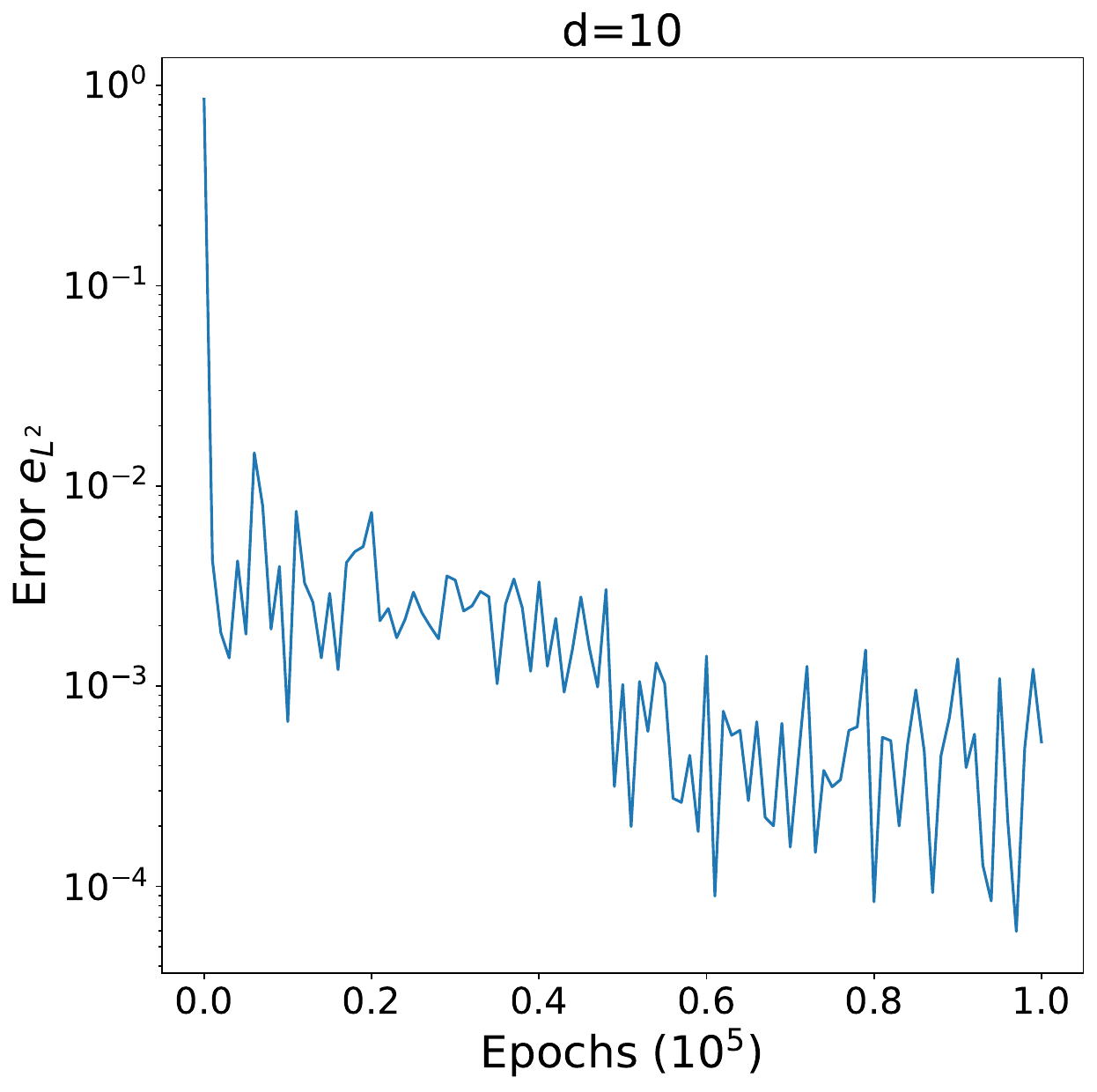}
\includegraphics[width=4cm,height=4cm]{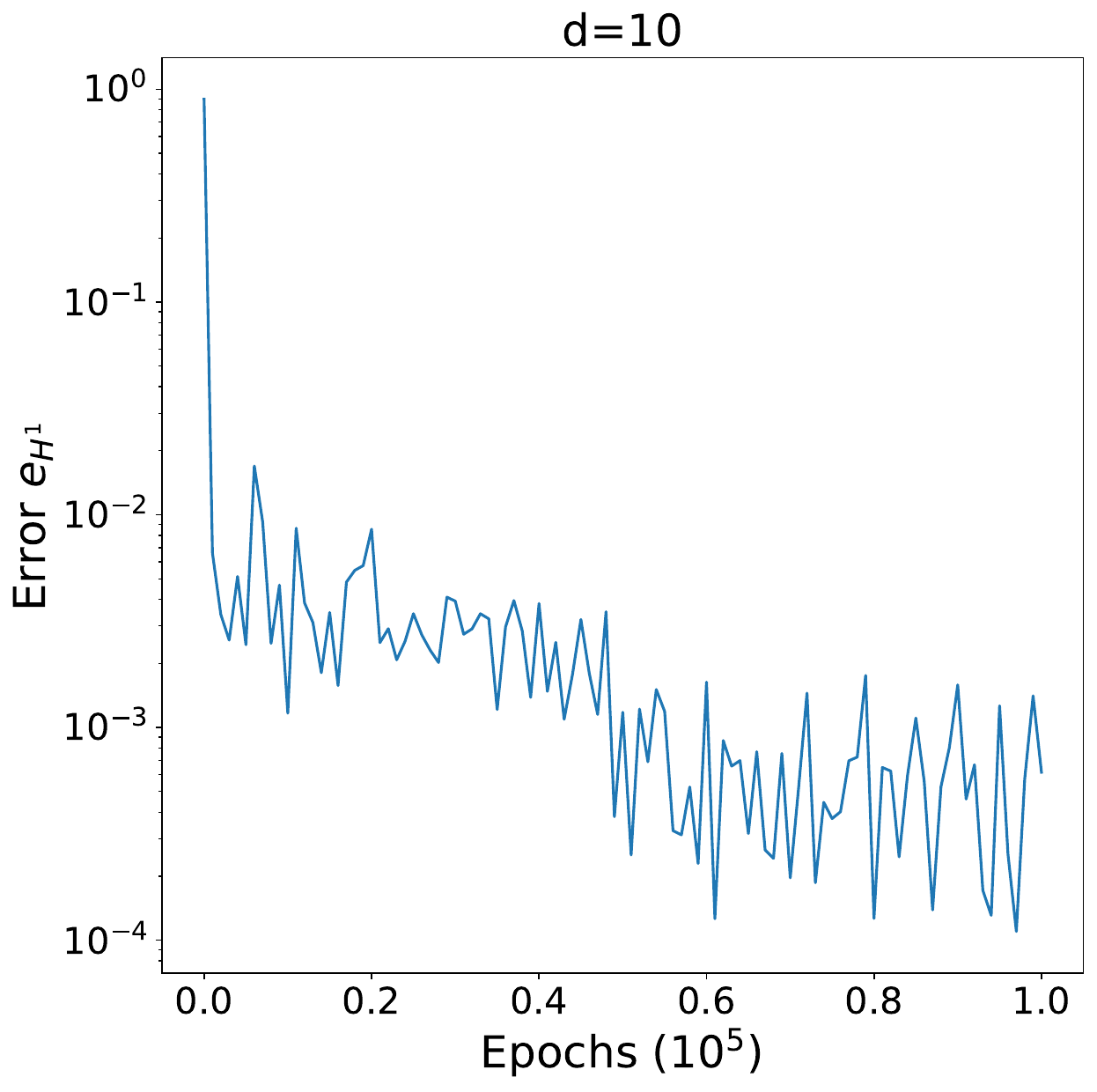}\\
\includegraphics[width=4cm,height=4cm]{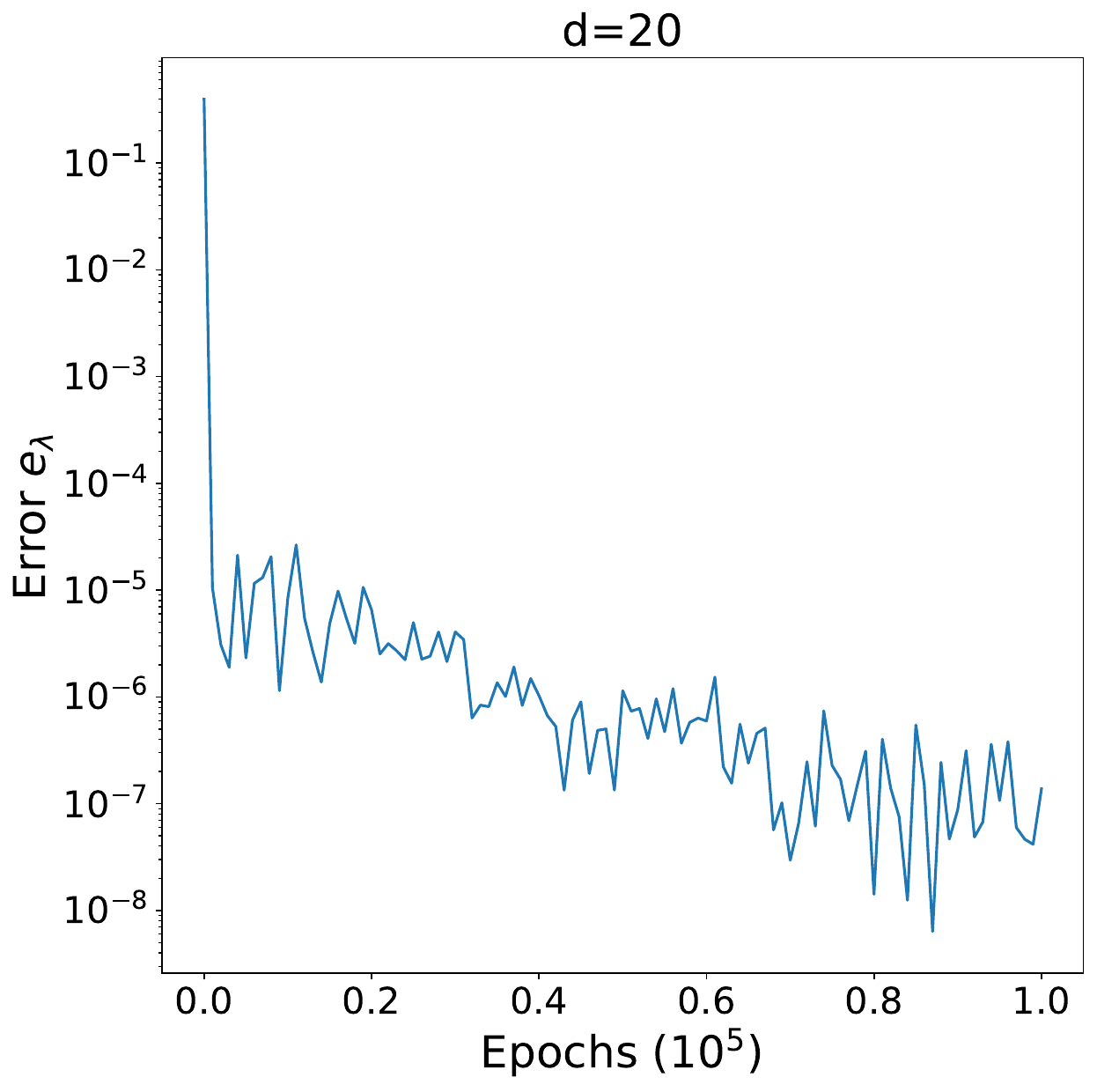}
\includegraphics[width=4cm,height=4cm]{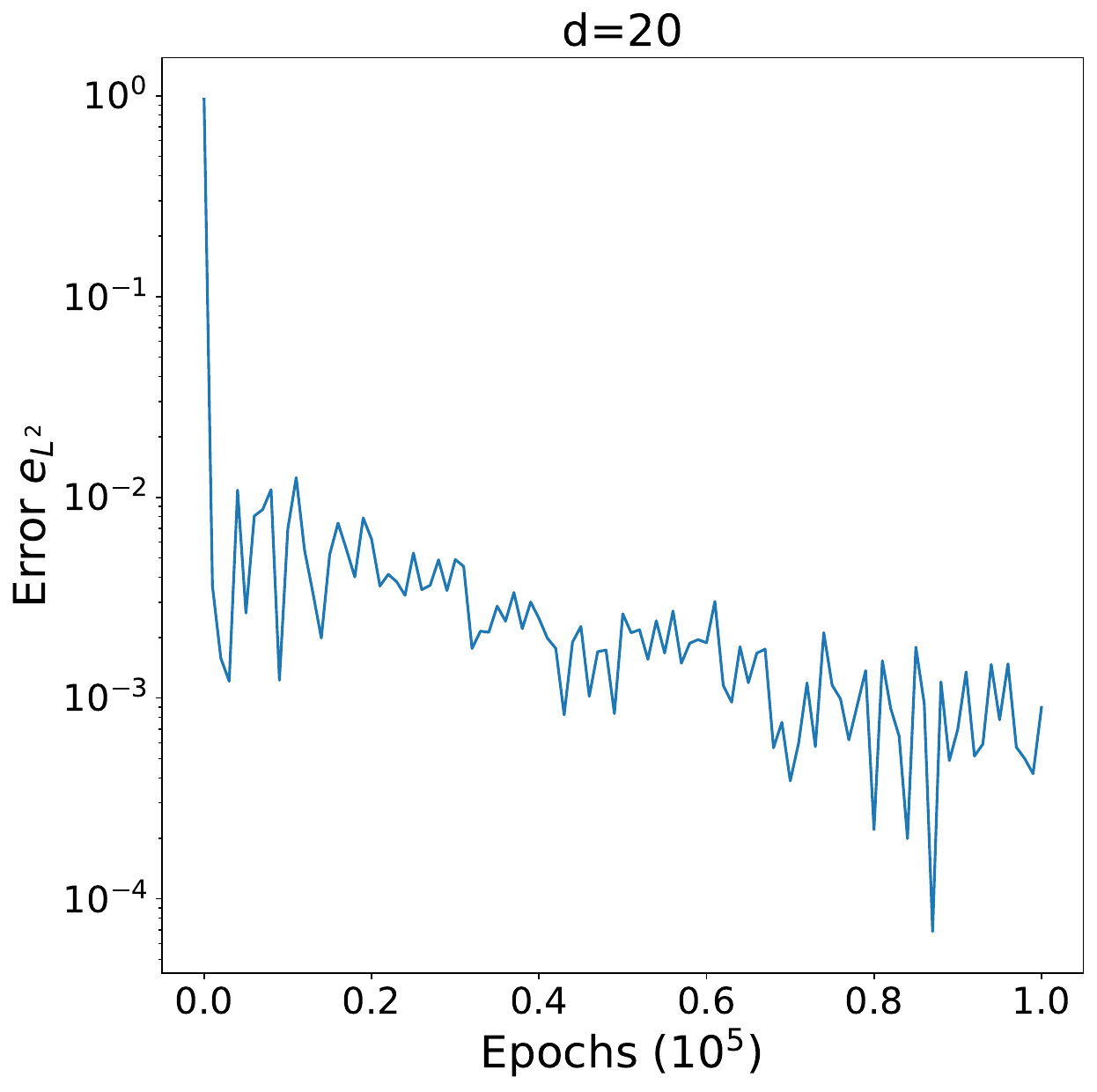}
\includegraphics[width=4cm,height=4cm]{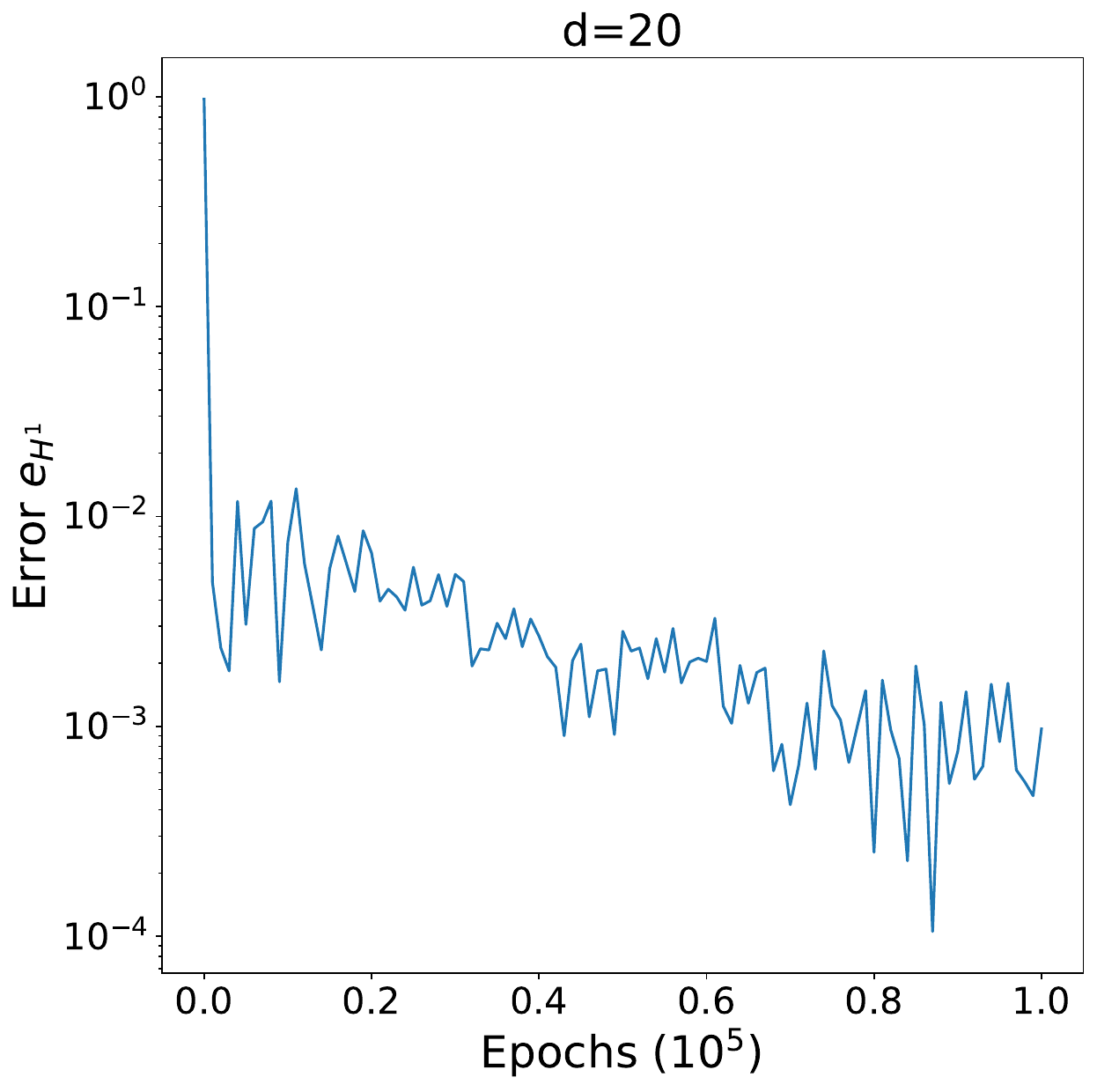}\\
\caption{Relative errors during the training process for Laplace eigenvalue problem: for $d=5$, $10$, and $20$. The left column shows the relative errors of eigenvalue approximations, the middle column shows the relative $L^2(\Omega)$ errors and the right column shows the relative $H^1(\Omega)$ errors of eigenfunction approximations.}\label{fig_laplace_high}
\end{figure}

\begin{table}[htb!]
\caption{Errors of Laplace eigenvalue problem for $d=5,10,20$.}\label{table_laplace_high}
\begin{center}
\begin{tabular}{ccccc}
\hline
$d$&   $e_{\lambda}$&   $e_{L^2}$&   $e_{H^1}$\\
\hline
5&   4.838e-09&   1.977e-05&   7.231e-05\\
10&   7.916e-09&   8.941e-05&   1.261e-04\\
20&   6.354e-09&   6.872e-05&   1.052e-04\\
\hline
\end{tabular}
\end{center}
\end{table}
Then we test ultra-high-dimensional cases with $d=128,256,512$, respectively.
Although the usual problem doesn't have such a high dimension,
it is important to point out that since the TNN structure (\ref{def_TNN})
includes the compound of $d$ terms, in ultra-high-dimensional cases,
numerical instability may occur. To improve the numerical stability,
we do a suitable scale for each dimension of TNN at the initialization step.
In implement, we decompose each $\Omega_i$ ($i=1, \cdots, d$) into $10$ equal subintervals and choose $16$ Gauss points on each subinterval. The Adam optimizer is employed with learning rate 0.003 to train a smaller scale TNN with $p=10$. \revise{Each subnetwork in of the TNN is an FNN with two hidden layers and
each hidden layer has 20 hidden neurons.}
\revise{We use the Adam optimizer in the first
100000 steps and then the L-BFGS in the subsequent 10000 steps to produce the final results.}
The final results are shown in Figure \ref{fig_laplace_ultra_high} and Table \ref{table_laplace_ultra_high}. In ultra-high-dimensional cases, the TNN method still has almost the same convergence behaviors for different dimensions and the final results are not much worse than that in high-dimensional cases. All relative errors are on a convincing order of magnitude.

\begin{figure}[htb]
\centering
\includegraphics[width=4cm,height=4cm]{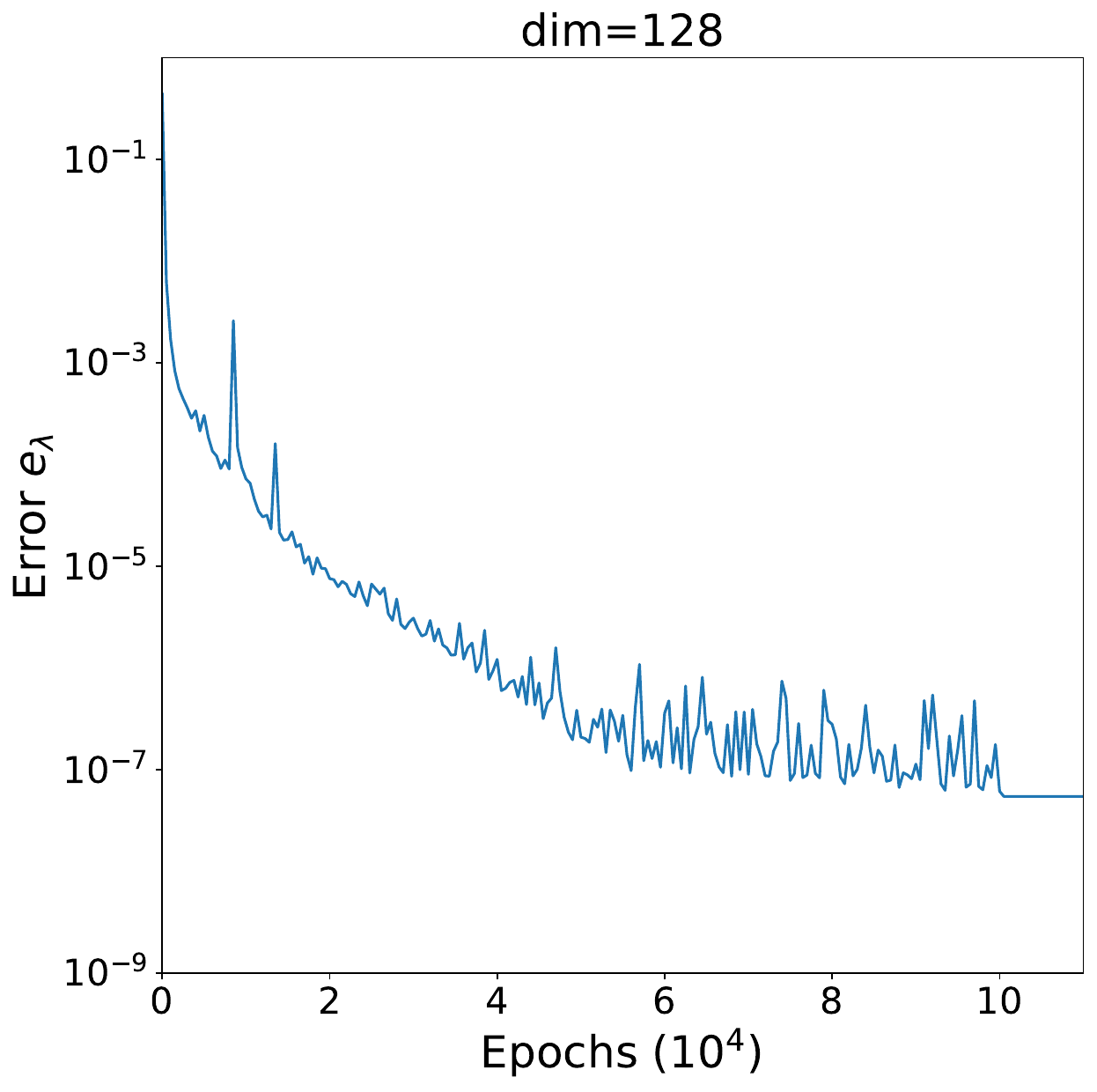}
\includegraphics[width=4cm,height=4cm]{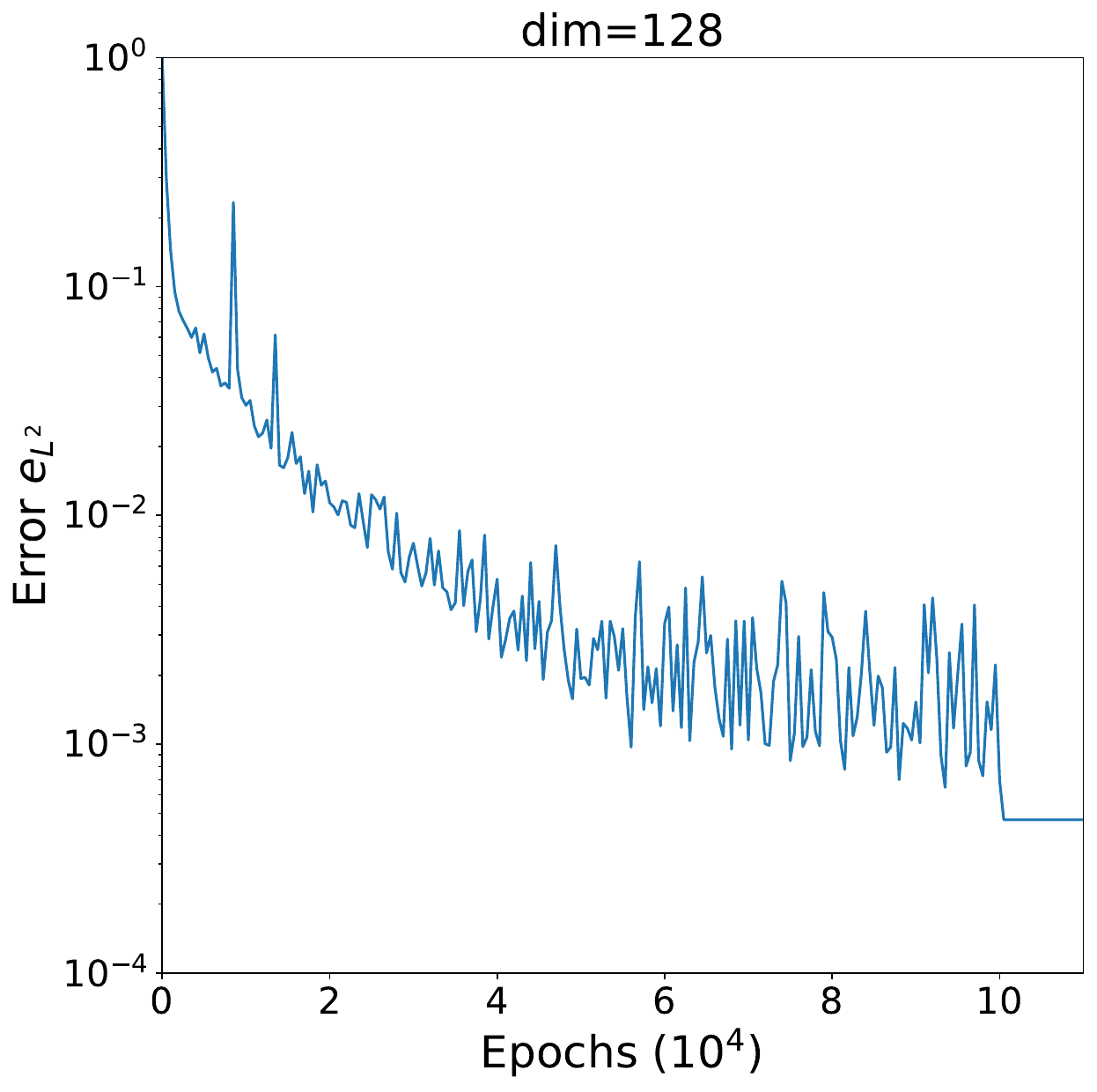}
\includegraphics[width=4cm,height=4cm]{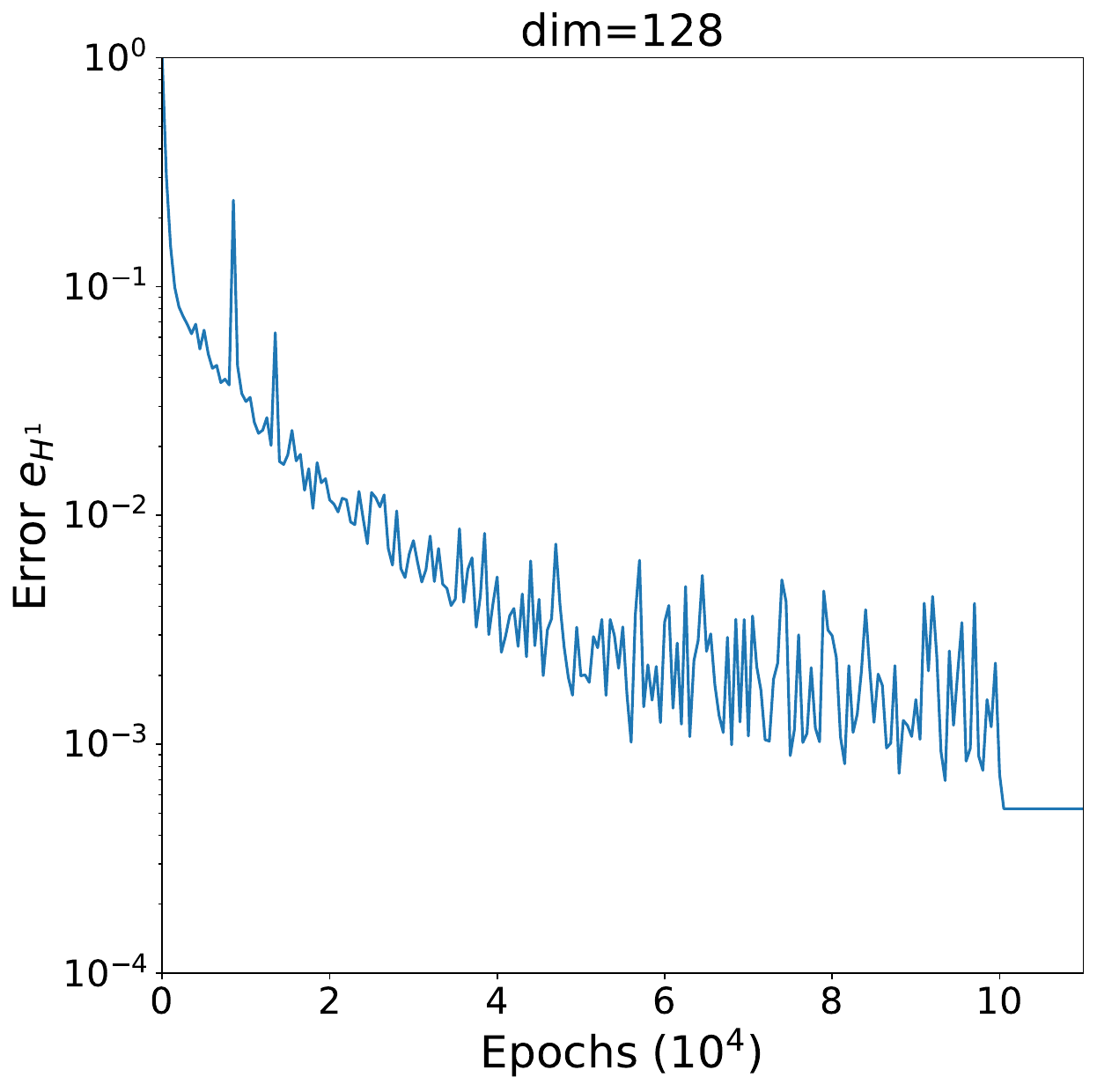}\\
\includegraphics[width=4cm,height=4cm]{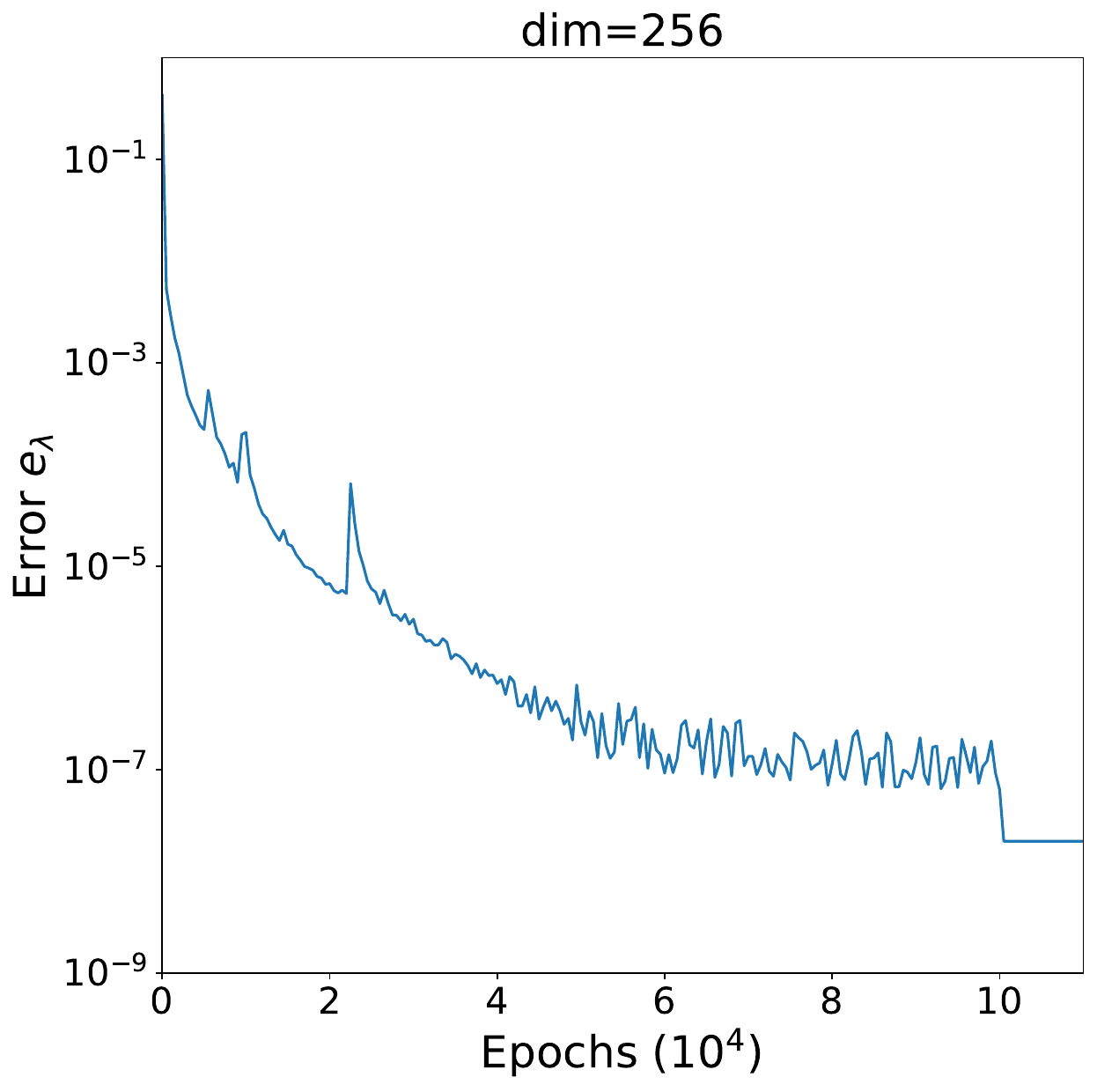}
\includegraphics[width=4cm,height=4cm]{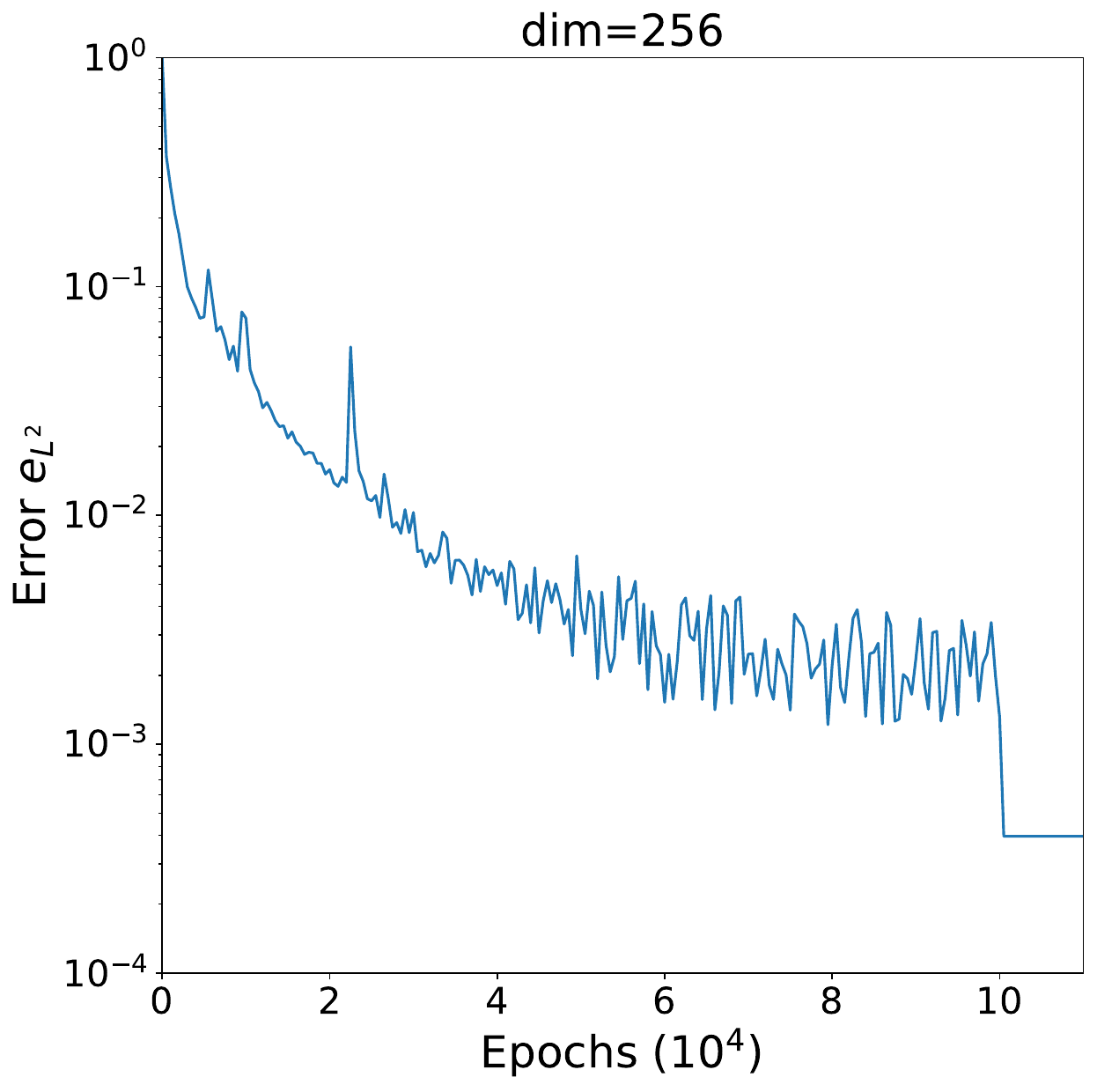}
\includegraphics[width=4cm,height=4cm]{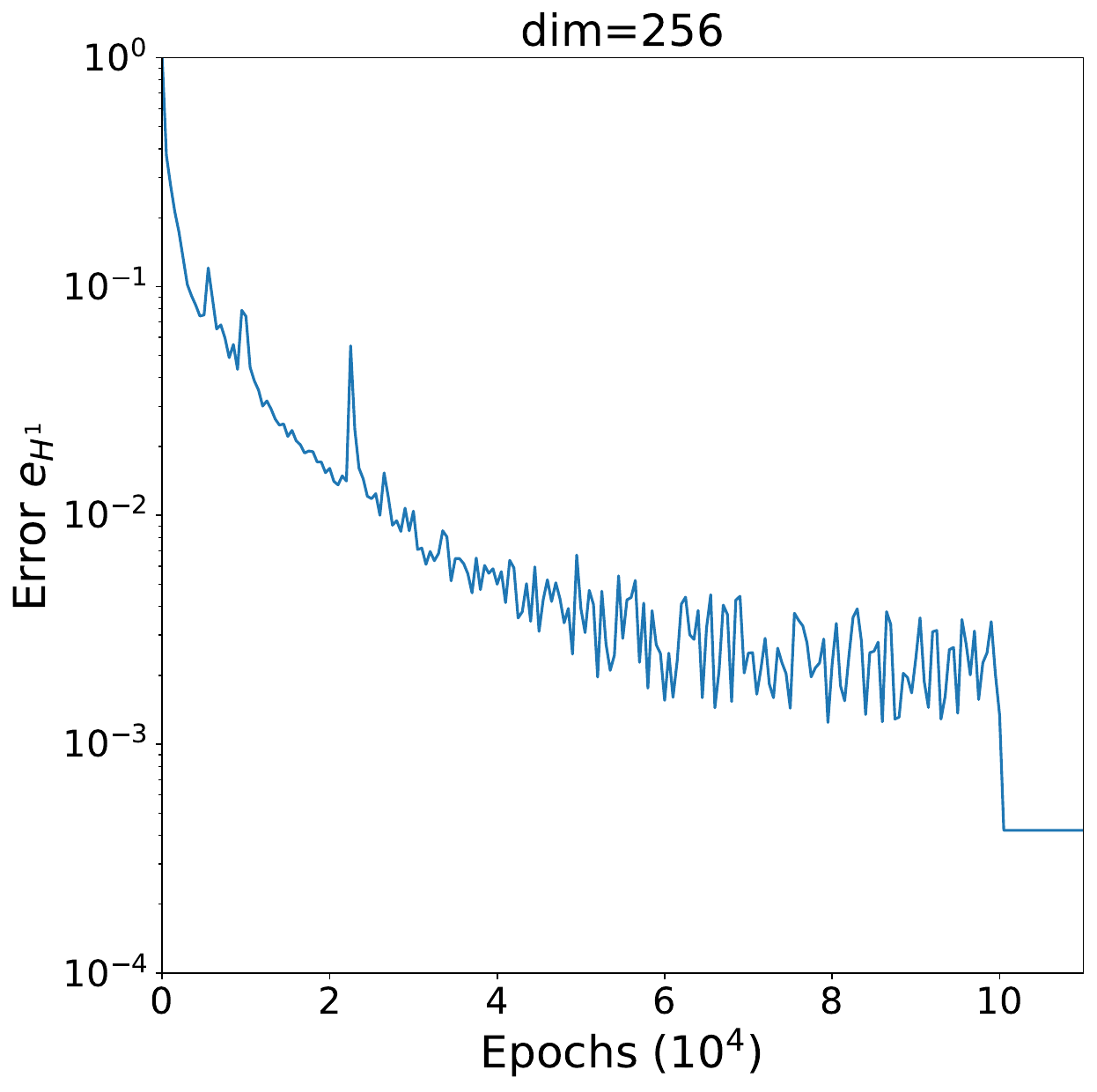}\\
\includegraphics[width=4cm,height=4cm]{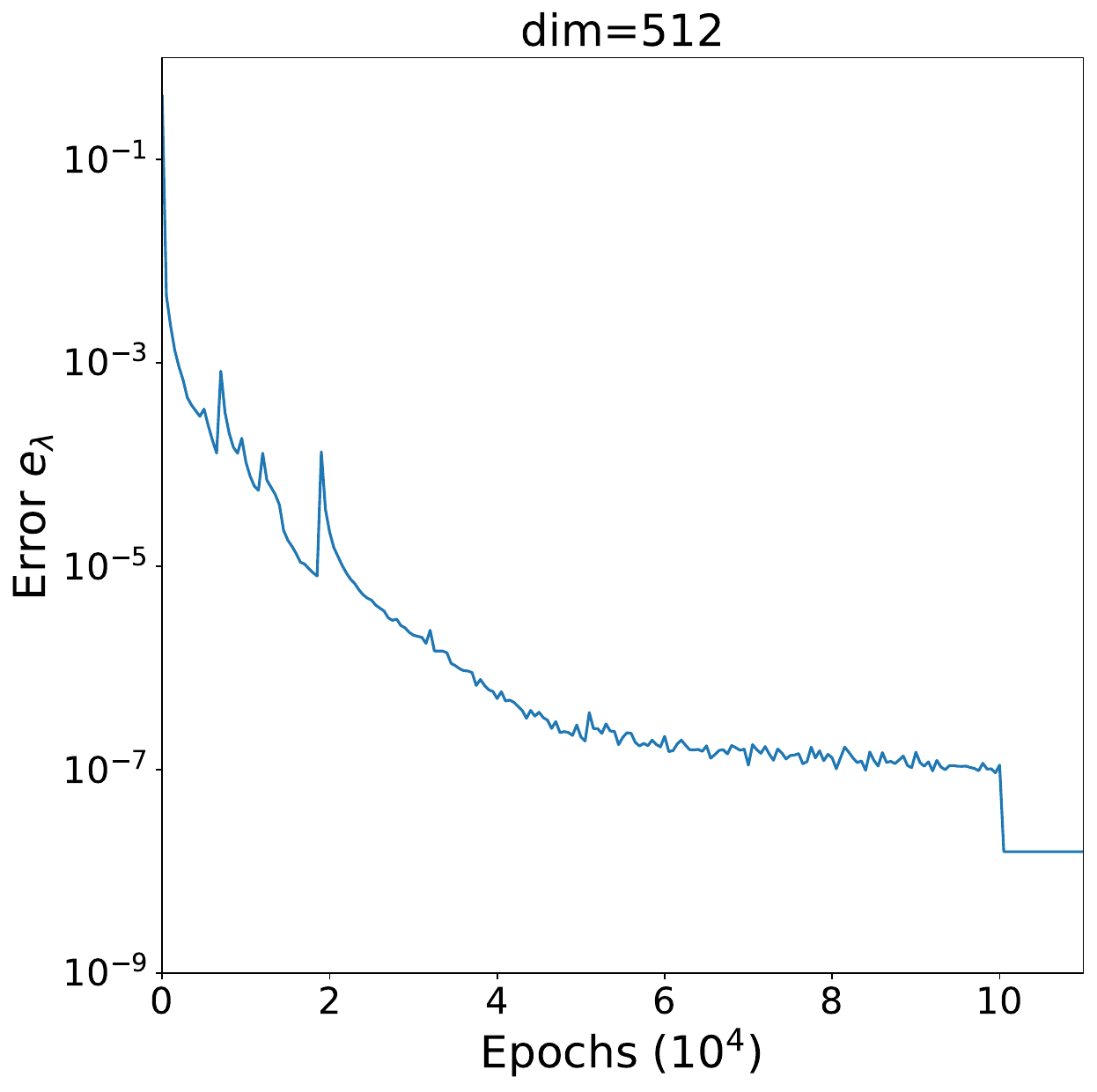}
\includegraphics[width=4cm,height=4cm]{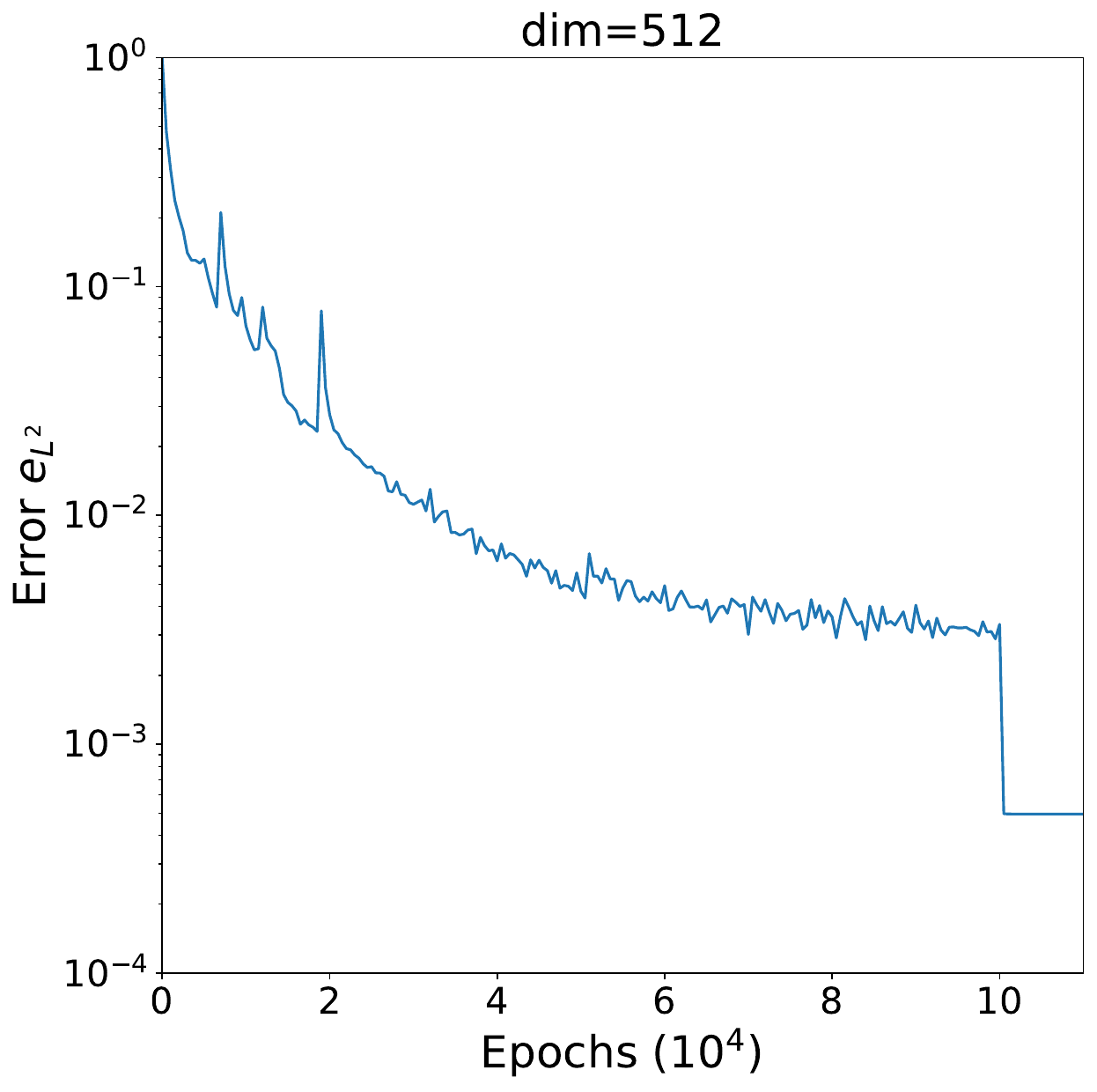}
\includegraphics[width=4cm,height=4cm]{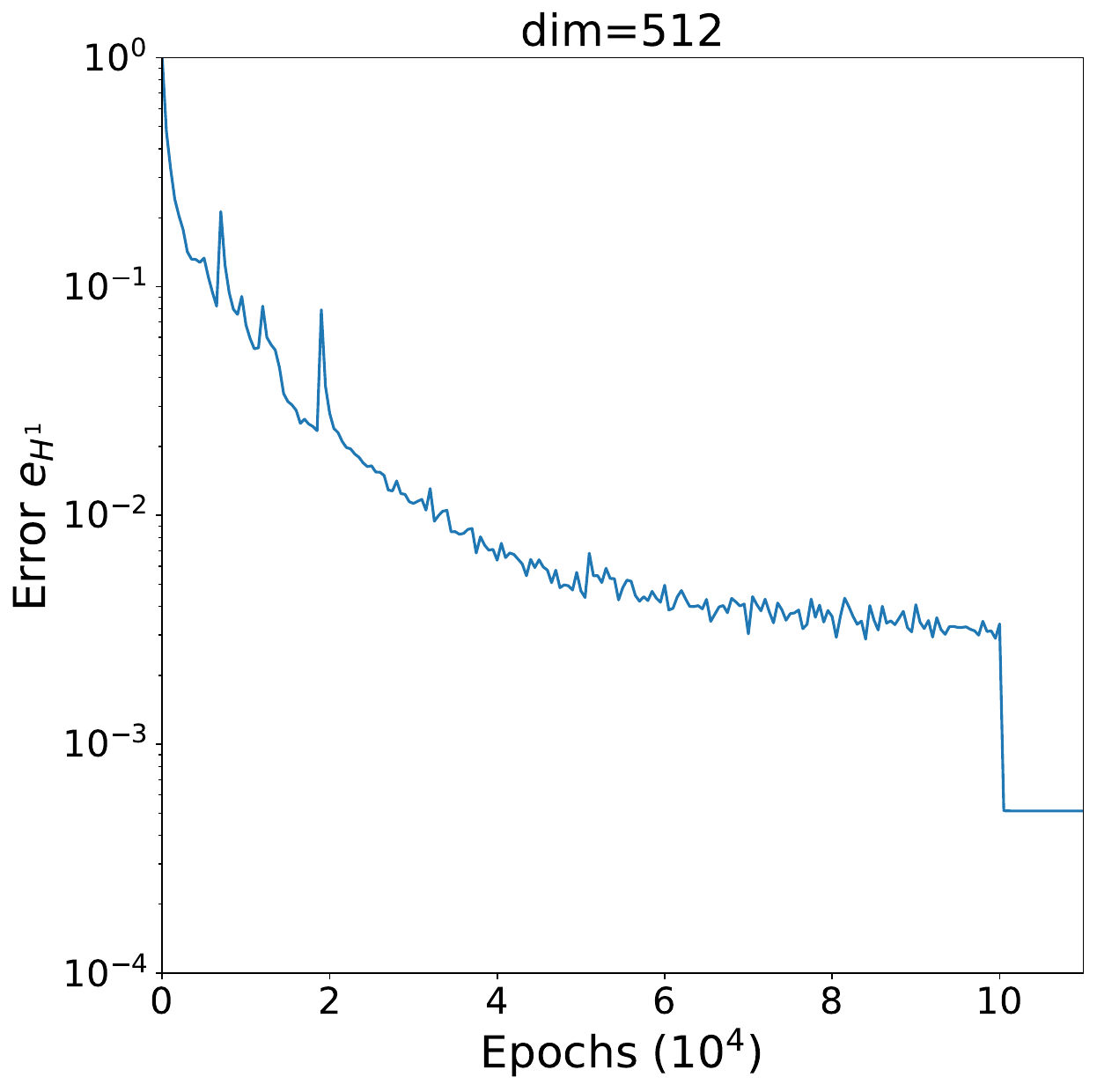}\\
\caption{\revise{Relative errors during the training process for Laplace eigenvalue problem: for $d=128$,
$256$ and $512$. The left column shows the relative errors of eigenvalue approximations,
the middle column shows the relative $L^2(\Omega)$ errors and the right column shows
the relative $H^1(\Omega)$ errors of eigenfunction approximations.}}\label{fig_laplace_ultra_high}
\end{figure}

\begin{table}[htb!]
\caption{\revise{Errors of Laplace eigenvalue problem for $d=128,256,512$.}}\label{table_laplace_ultra_high}
\begin{center}
\begin{tabular}{ccccc}
\hline
$d$&   $e_{\lambda}$&   $e_{L^2}$&   $e_{H^1}$\\
\hline
128&   5.462e-08&   4.676e-04&   5.223e-04\\
256&   1.980e-08&   3.962e-04&   4.205e-04\\
512&   1.560e-08&   4.956e-04&   5.111e-04\\
\hline
\end{tabular}
\end{center}
\end{table}

\subsection{Eigenvalue problem with harmonic oscillator}\label{Section_harmonic}
In the second example, the potential function is defined as (\ref{harmonic oscillator}).
Then the exact smallest eigenvalue and eigenfunction are
\begin{eqnarray*}
\lambda=d,\ \ \ u(x)=\prod_{i=1}^de^{-x_i^2/2}.
\end{eqnarray*}

As the first example in Section \ref{Section_laplace}, high-dimensional cases
with $d=5,10,20$ and ultra-high-dimensional cases with $d=128, 256, 512$
are also tested, respectively. We truncate the computational domain from $\mathbb R^d$ to $[-5,5]^d$,
use 100 equal subintervals and 16 Gauss points quadrature scheme for all cases.
The Adam optimizer is employed to train
TNN of the same size as the first example but with learning rate of $0.01$ and $0.003$
for high-dimensional cases and ultra-high-dimensional cases, respectively.
\revise{For high-dimensional cases, the Adam optimizer is used with $100000$ steps}.
\revise{For ultra-high-dimensional cases, we use the Adam optimizer in the first 50000 steps
and then the L-BFGS in the subsequent 10000 steps to produce the final results.}
Numerical results for $d=5, 10, 20$ are shown in Figure \ref{fig_harmonic_high}
and Table \ref{table_harmonic_high} and that for $d=128, 256, 512$ are shown
in Figure \ref{fig_harmonic_ultra_high} and Table \ref{table_harmonic_ultra_high}.
Since we truncate the computational domain, it is reasonable that the final relative
errors are a little worse than the examples of the Laplace eigenvalue problem.
There should exist some room for improving the accuracy.

\begin{figure}[htb!]
\centering
\includegraphics[width=4cm,height=4cm]{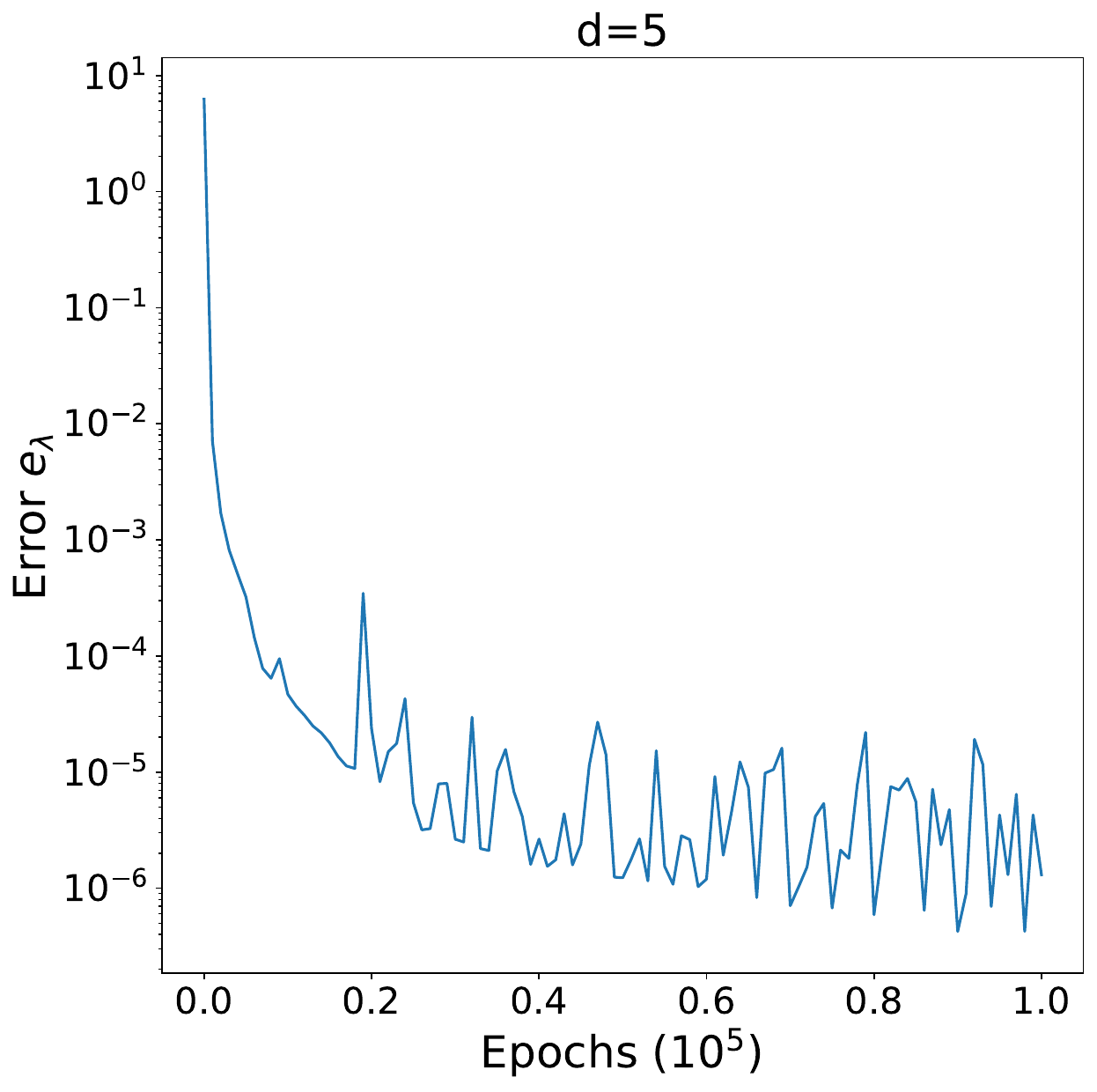}
\includegraphics[width=4cm,height=4cm]{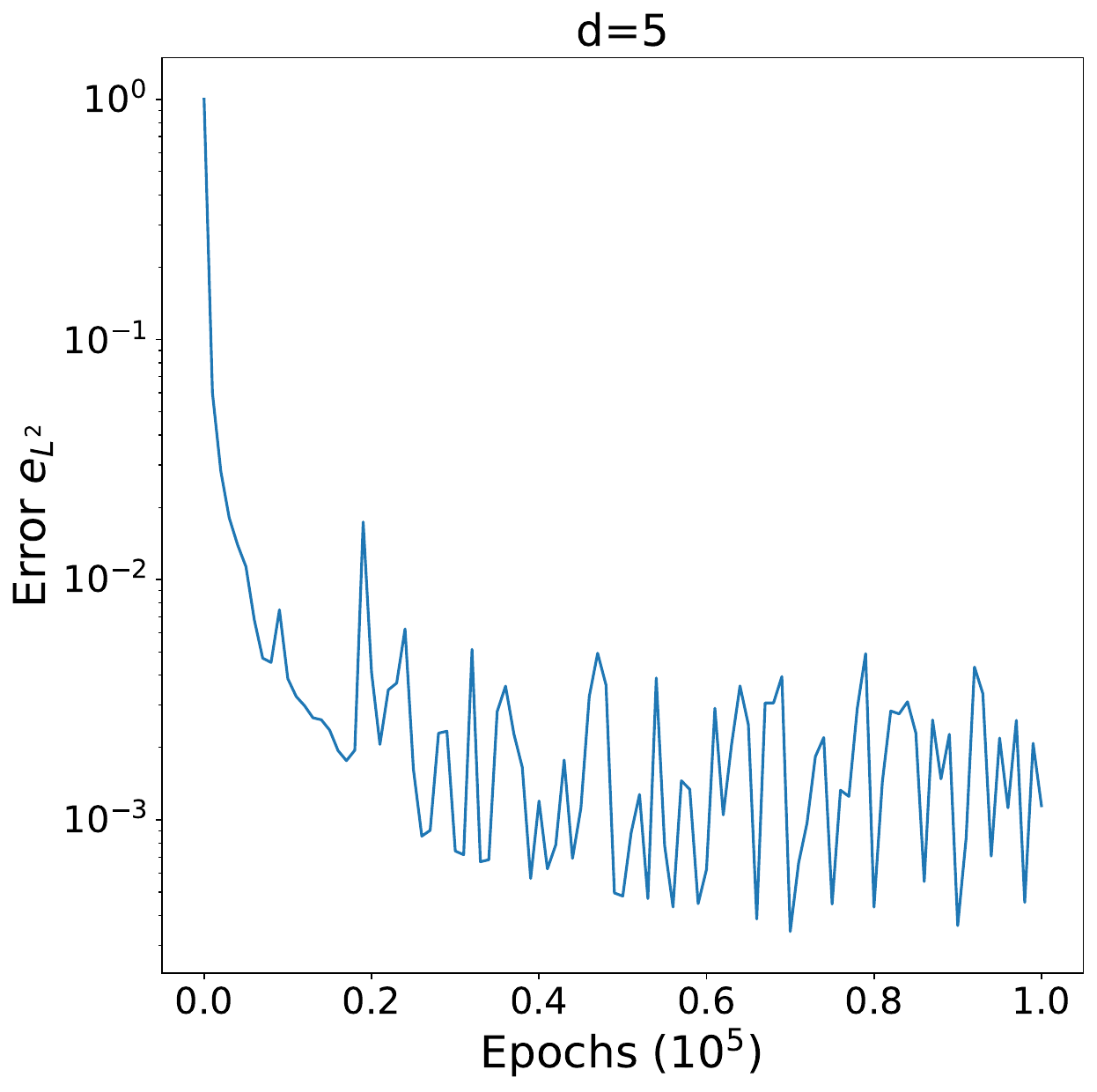}
\includegraphics[width=4cm,height=4cm]{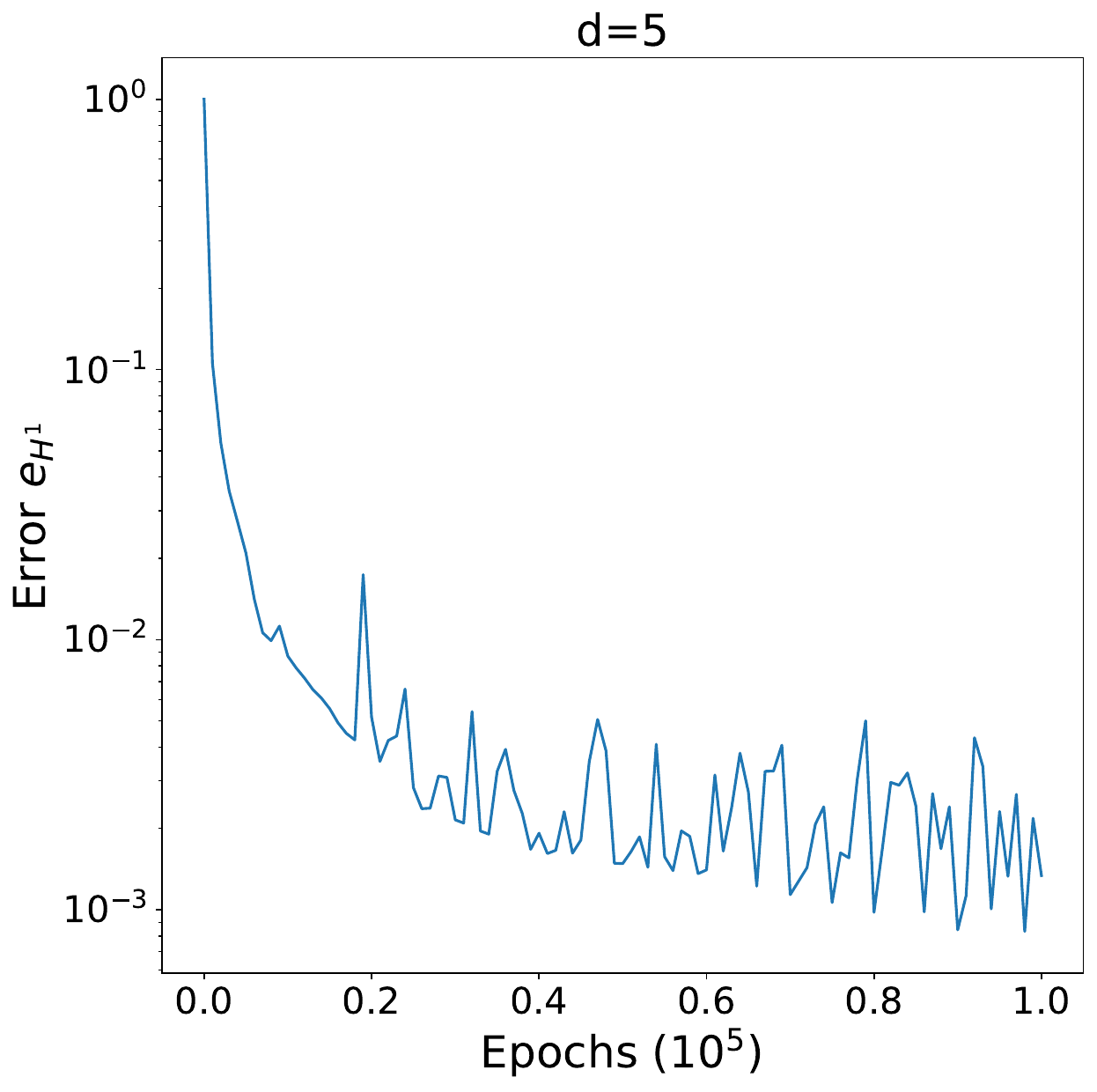}\\
\includegraphics[width=4cm,height=4cm]{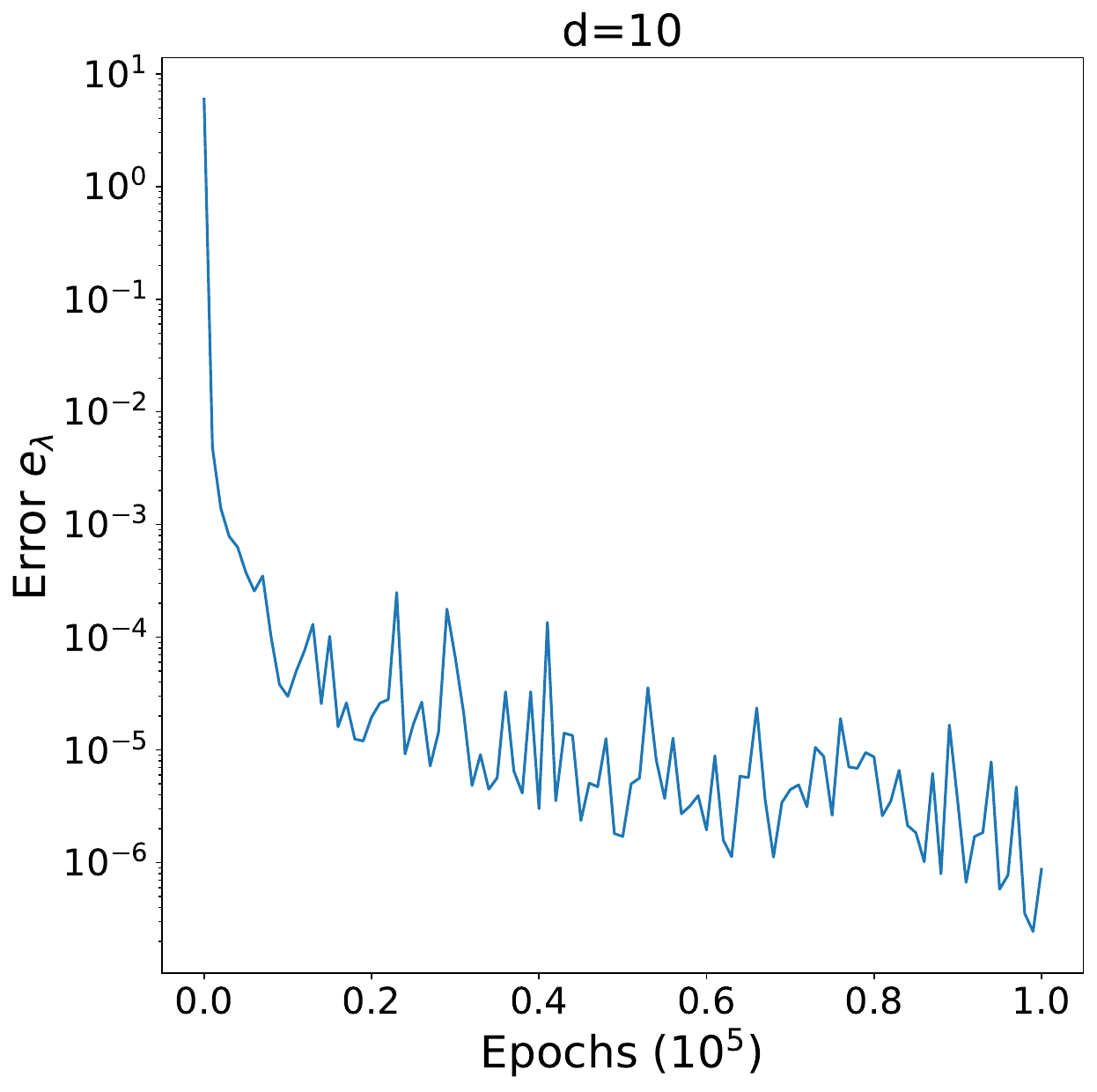}
\includegraphics[width=4cm,height=4cm]{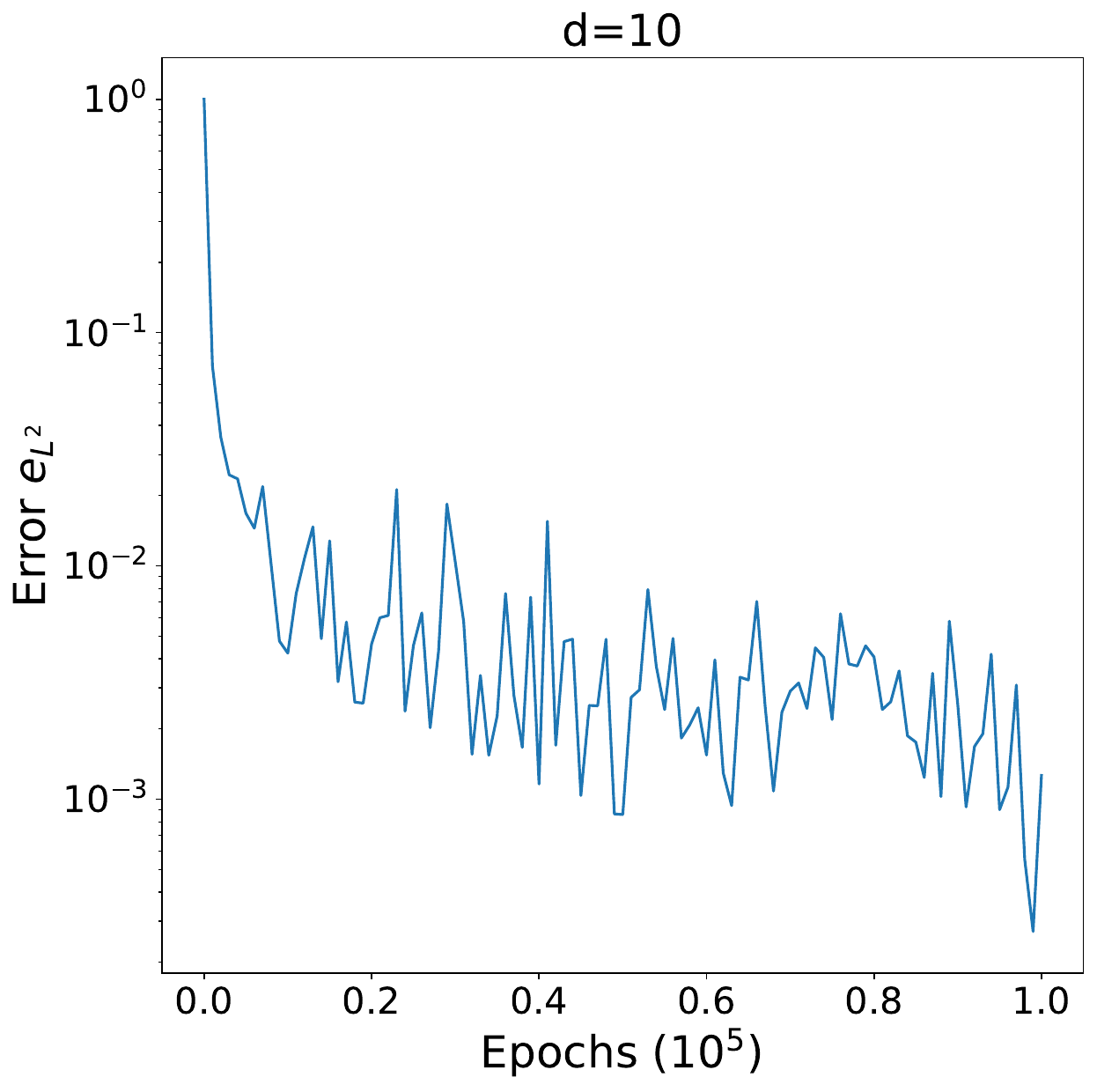}
\includegraphics[width=4cm,height=4cm]{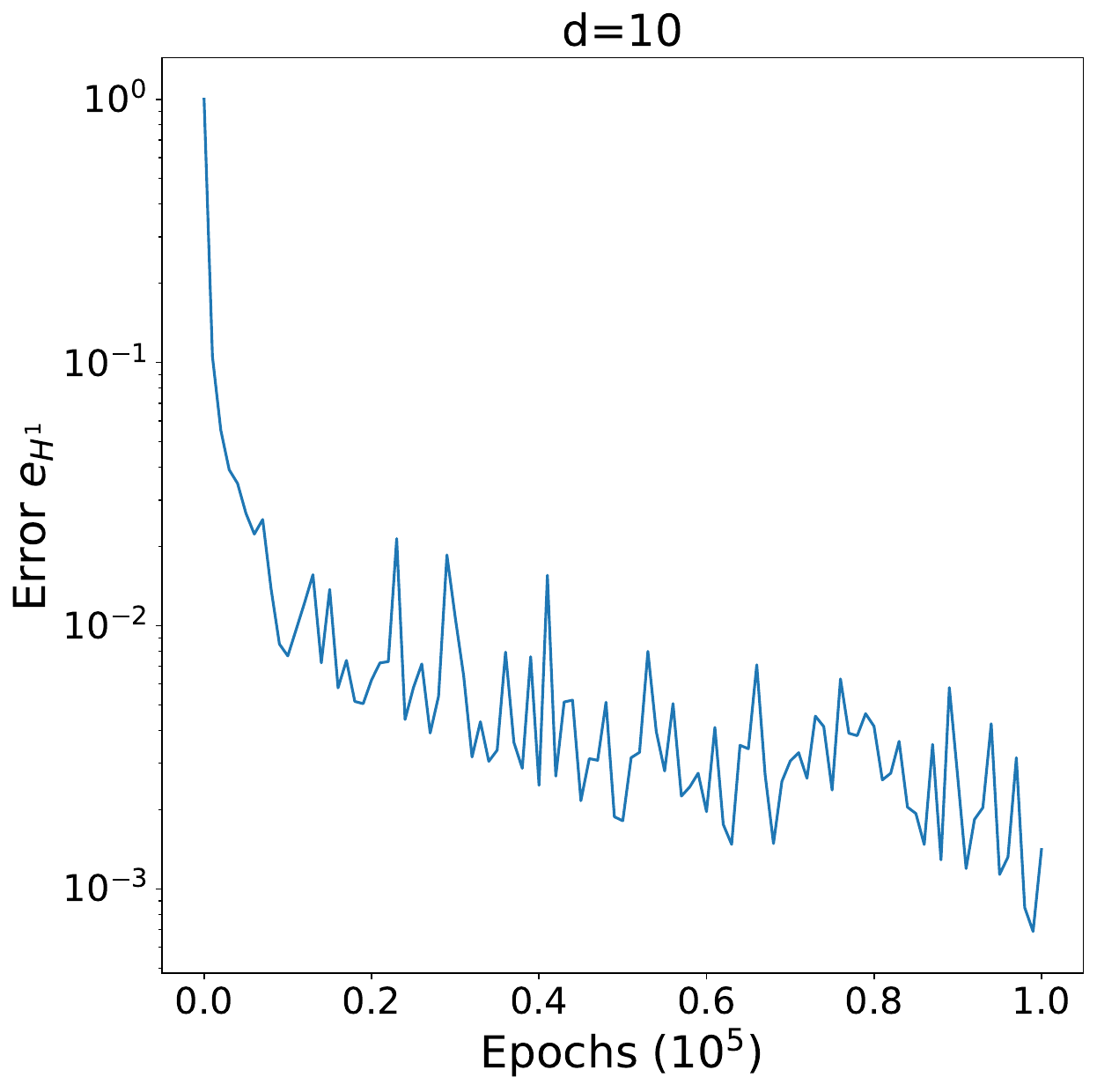}\\
\includegraphics[width=4cm,height=4cm]{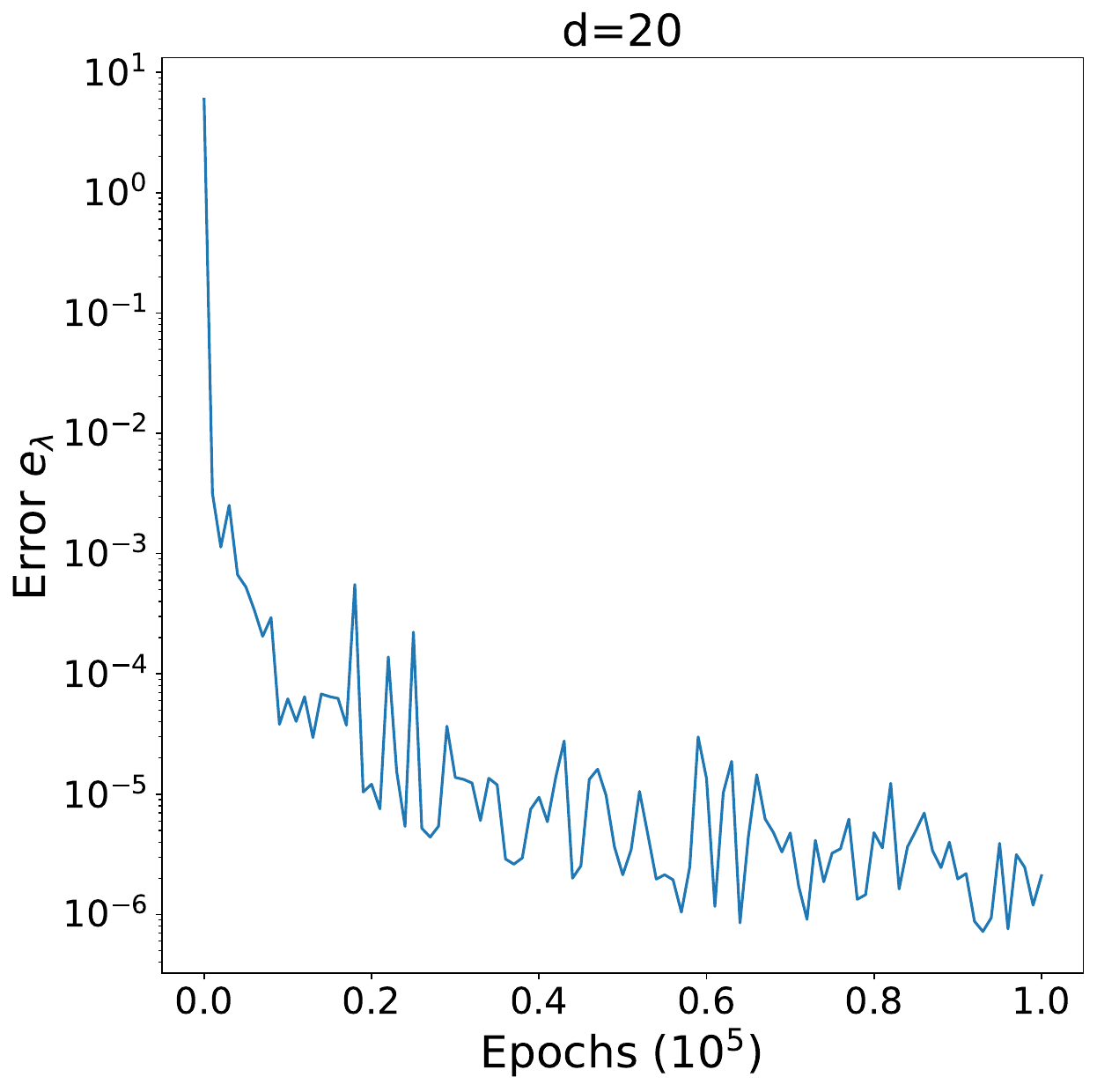}
\includegraphics[width=4cm,height=4cm]{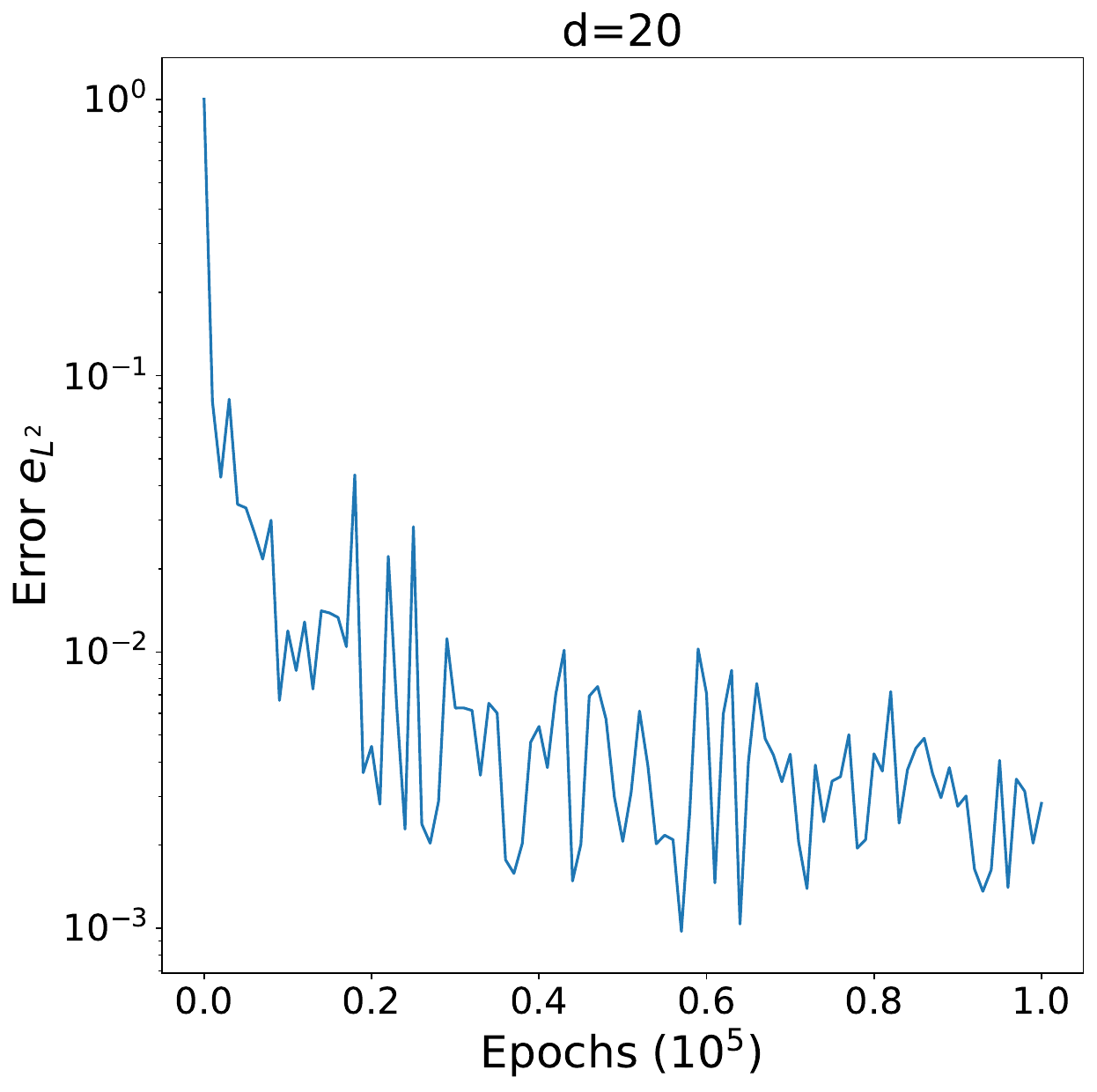}
\includegraphics[width=4cm,height=4cm]{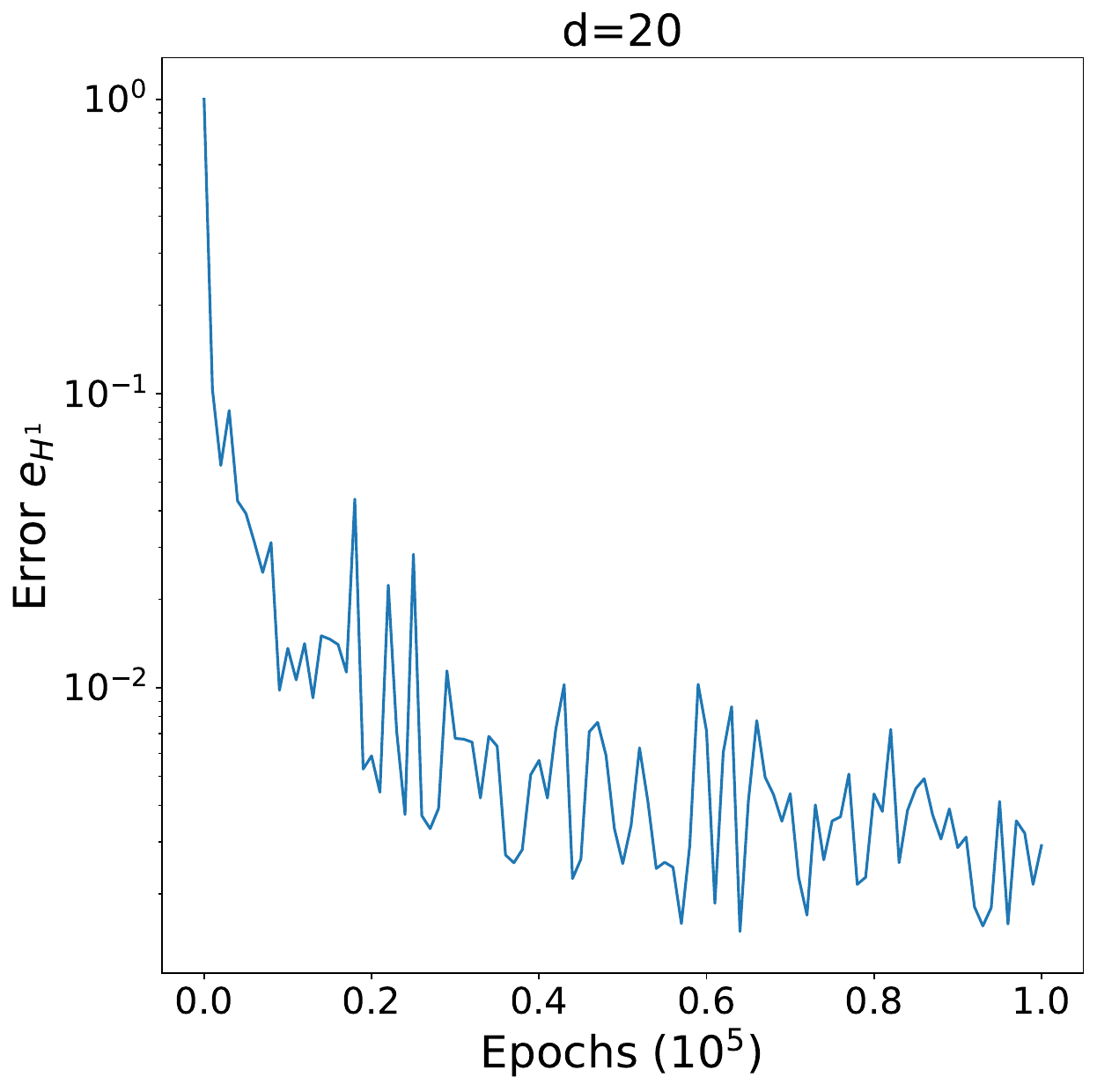}\\
\caption{Relative errors during the training process for the harmonic oscillator problem: for $d=5,10,20$.
The left column shows the relative errors of eigenvalue, the middle column shows the relative $L^2$ errors
and the right column shows the relative $H^1$ errors of eigenfunction approximations.}\label{fig_harmonic_high}
\end{figure}

\begin{table}[!htb]
\caption{Errors of the harmonic oscillator problem for $d=5,10,20$.}\label{table_harmonic_high}
\begin{center}
\begin{tabular}{ccccc}
\hline
$d$&   $e_{\lambda}$&   $e_{L^2}$&   $e_{H^1}$\\
\hline
5&   4.241e-07&   3.626e-04&   8.431e-04\\
10&   2.446e-07&   2.709e-04&   6.889e-04\\
20&   7.225e-07&   1.361e-03&   1.555e-03\\
\hline
\end{tabular}
\end{center}
\end{table}

\begin{figure}[htb!]
\centering
\includegraphics[width=4cm,height=4cm]{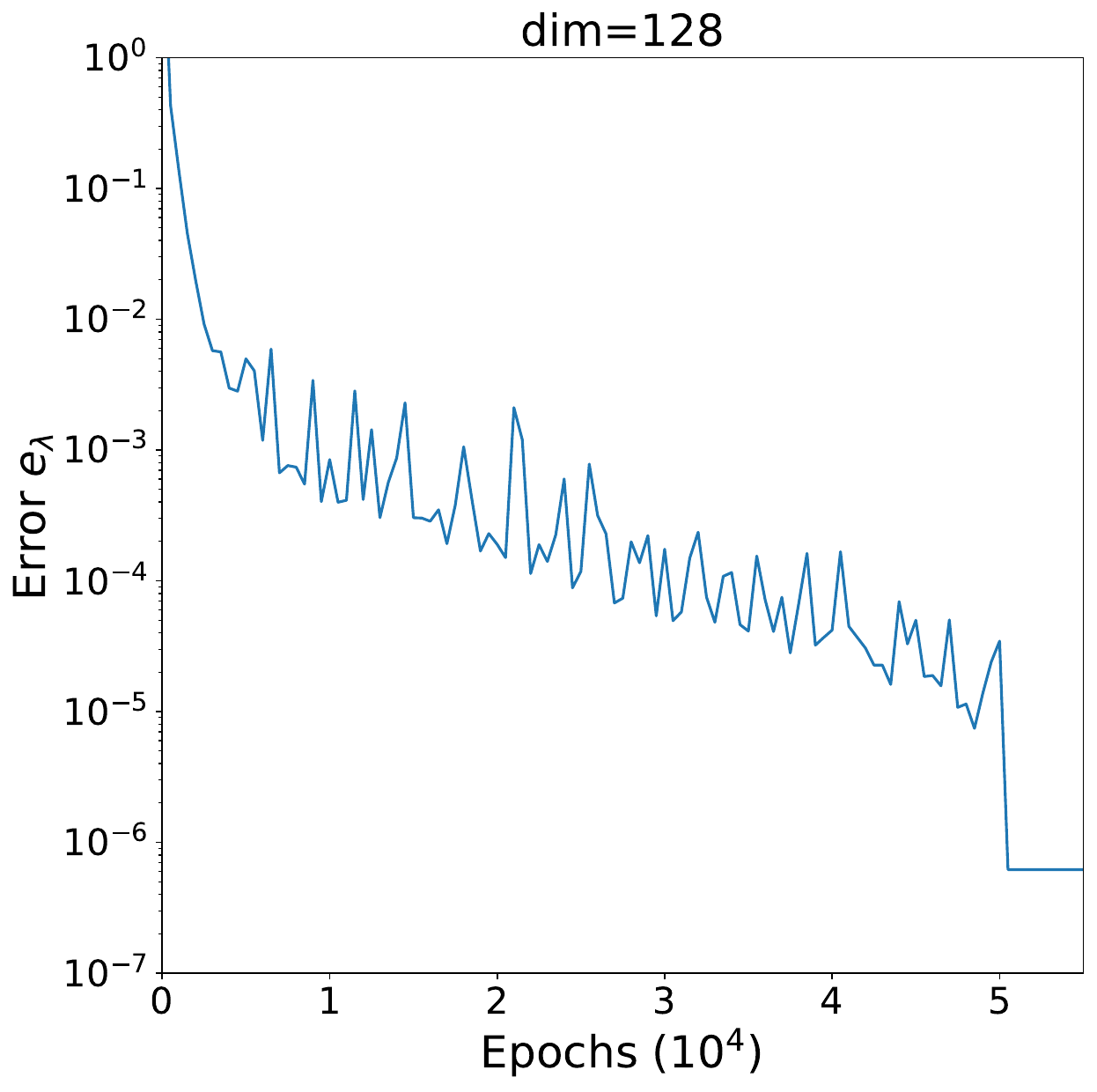}
\includegraphics[width=4cm,height=4cm]{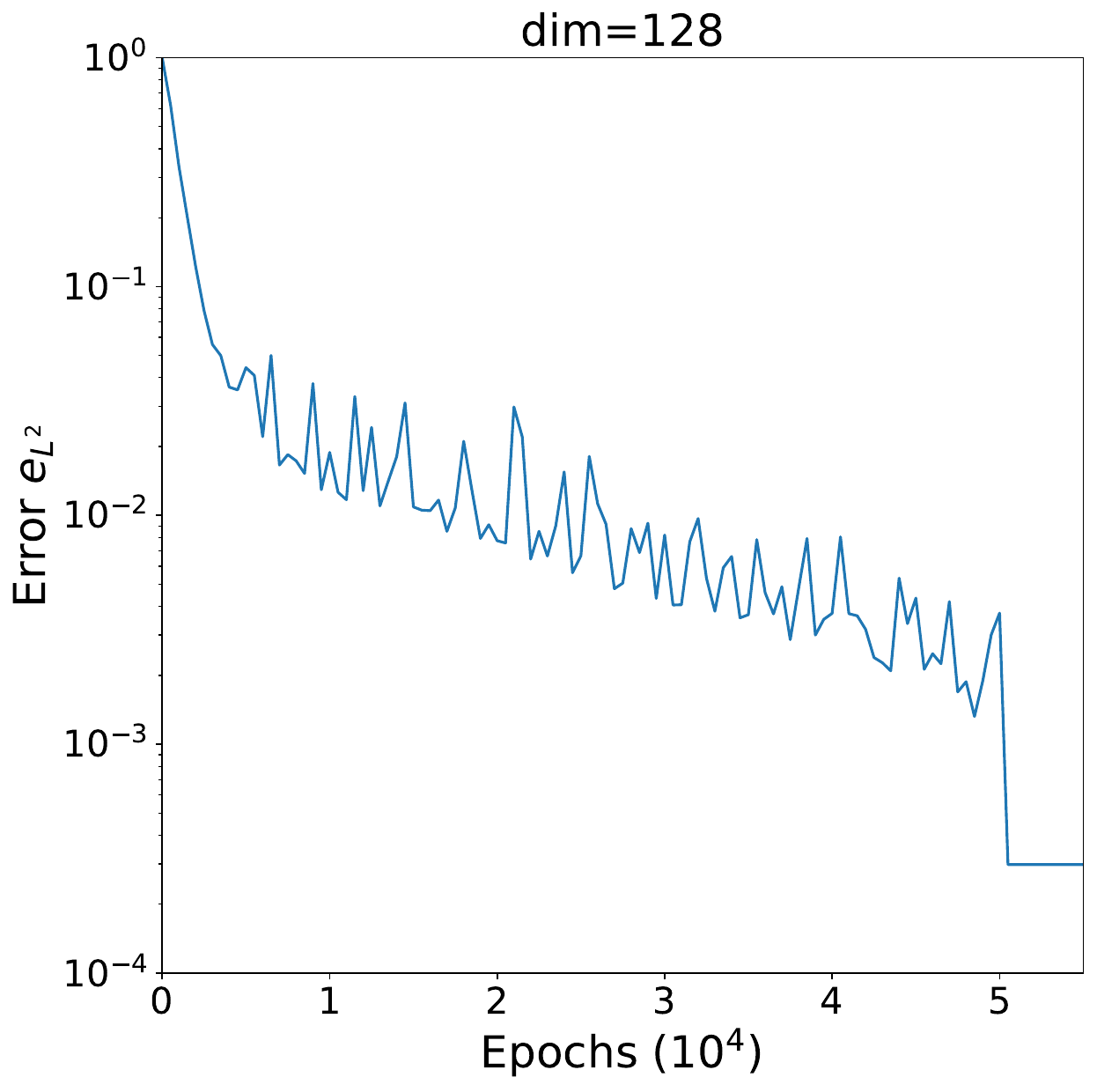}
\includegraphics[width=4cm,height=4cm]{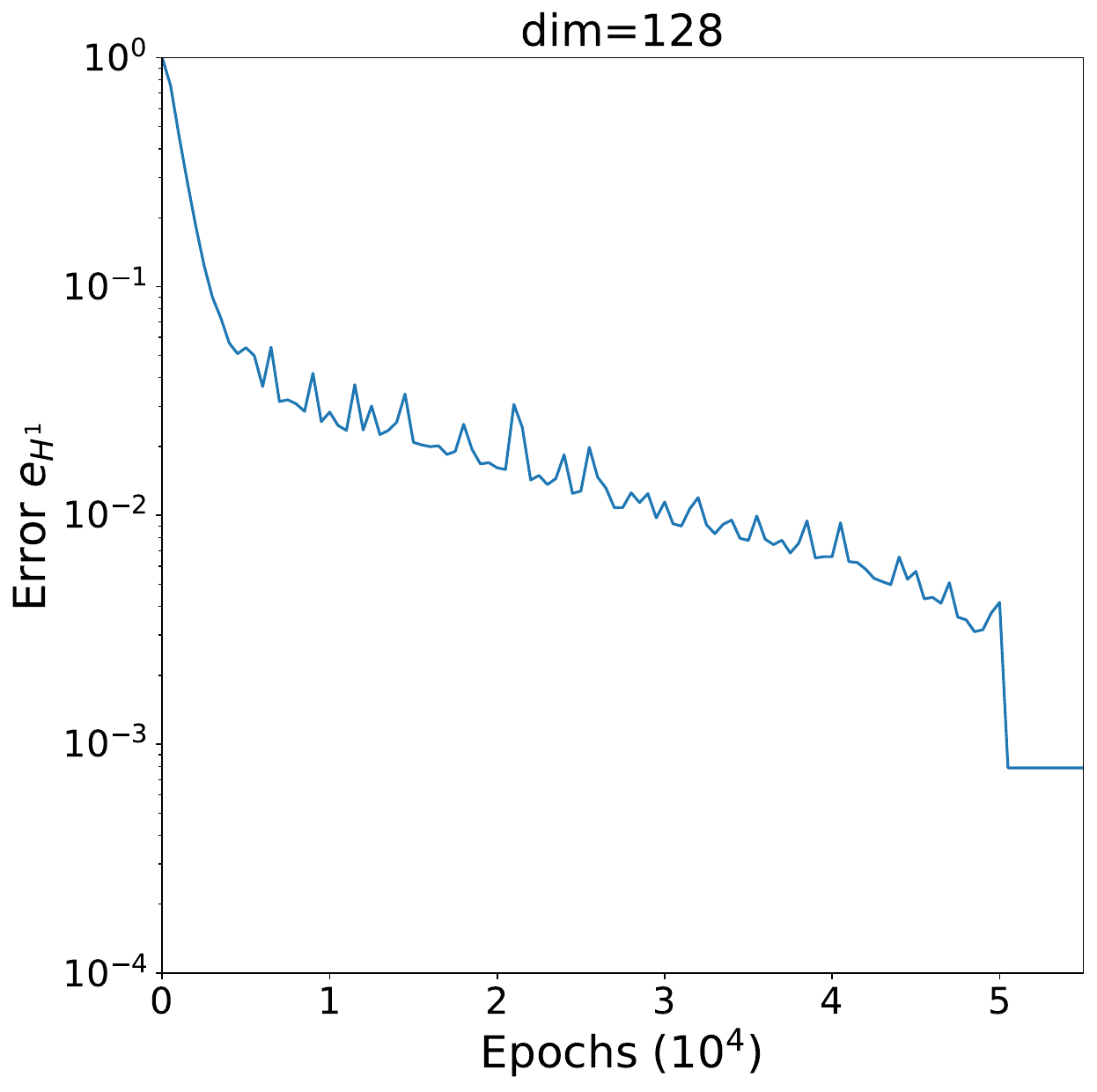}\\
\includegraphics[width=4cm,height=4cm]{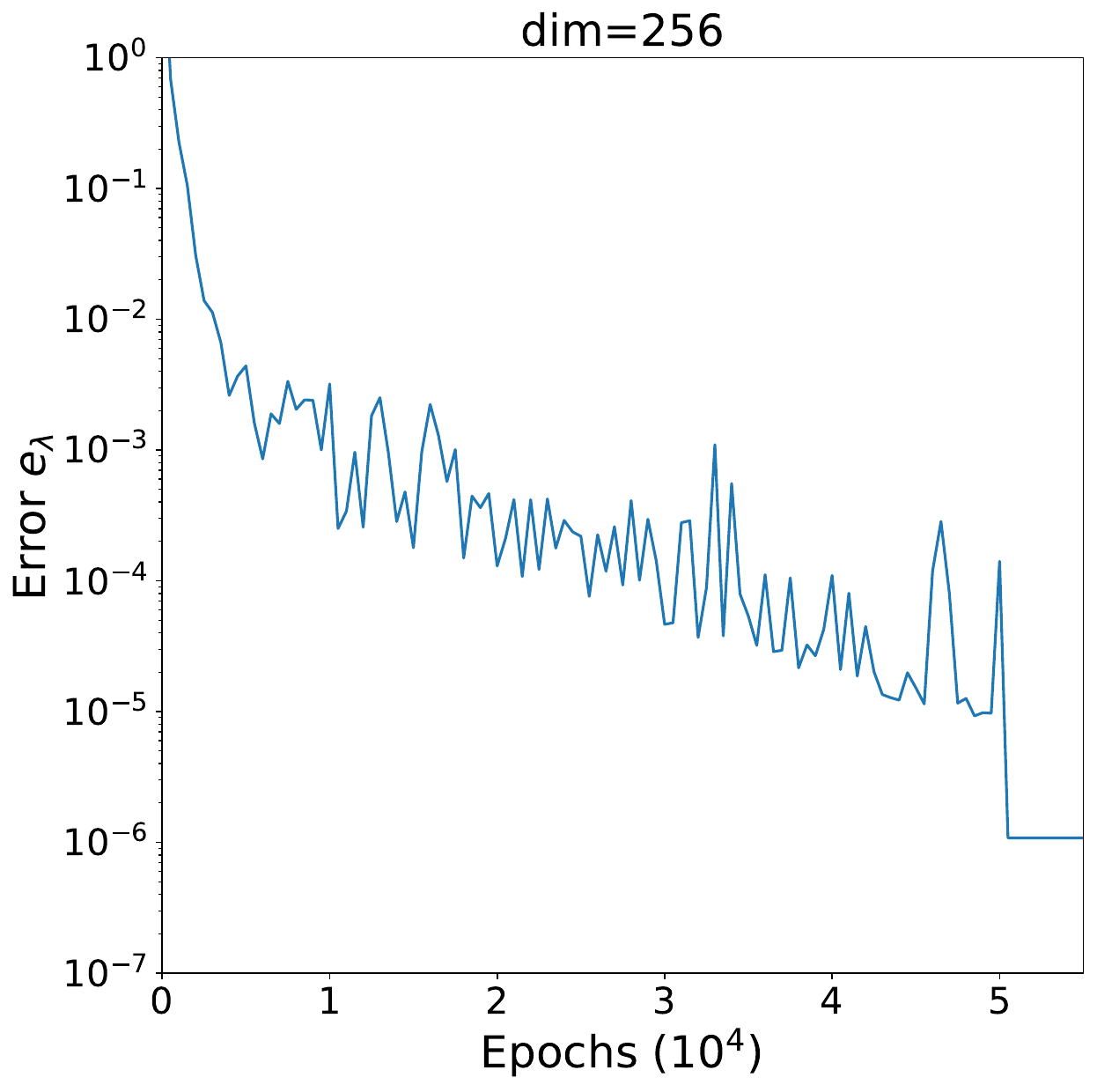}
\includegraphics[width=4cm,height=4cm]{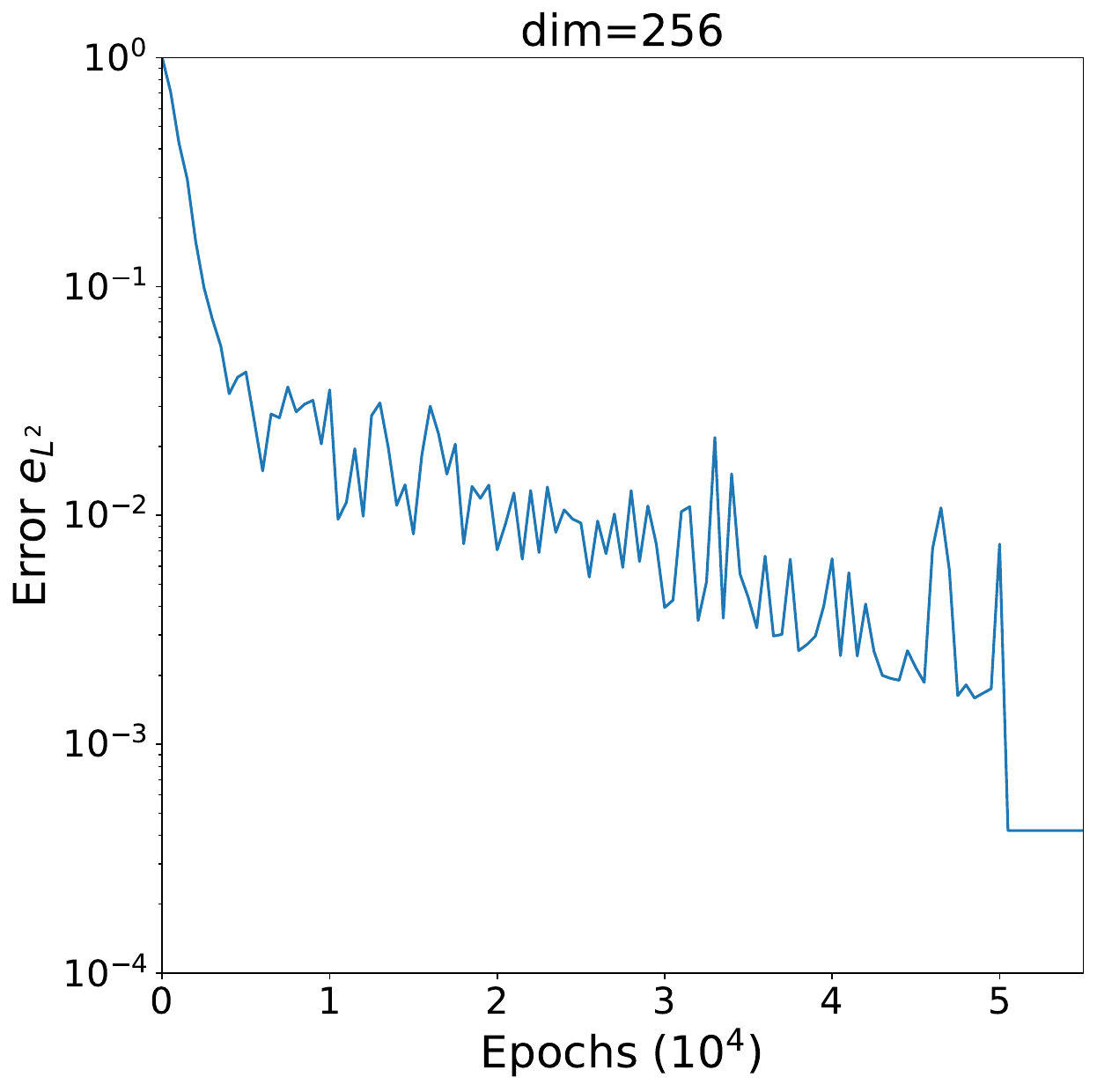}
\includegraphics[width=4cm,height=4cm]{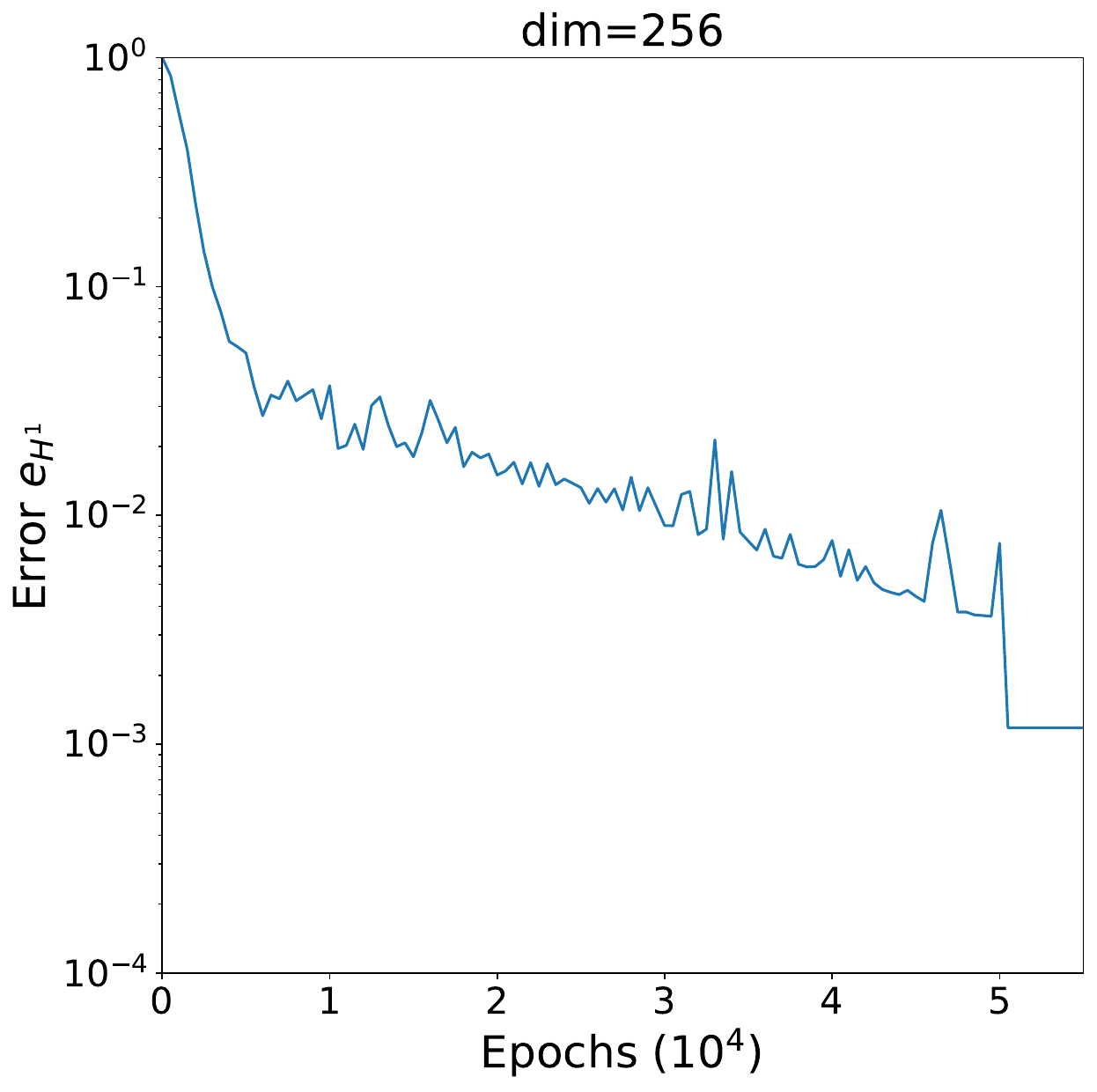}\\
\includegraphics[width=4cm,height=4cm]{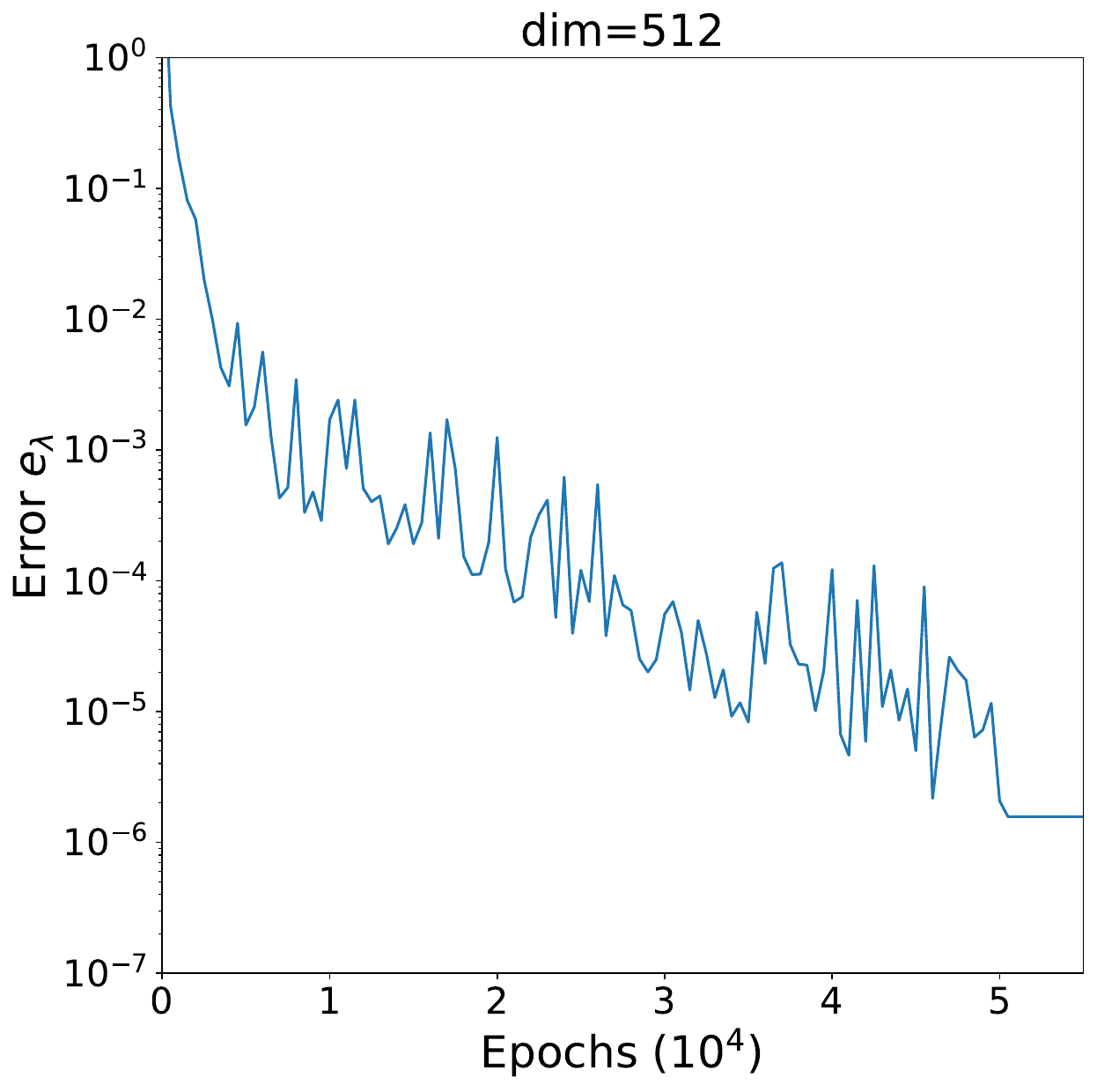}
\includegraphics[width=4cm,height=4cm]{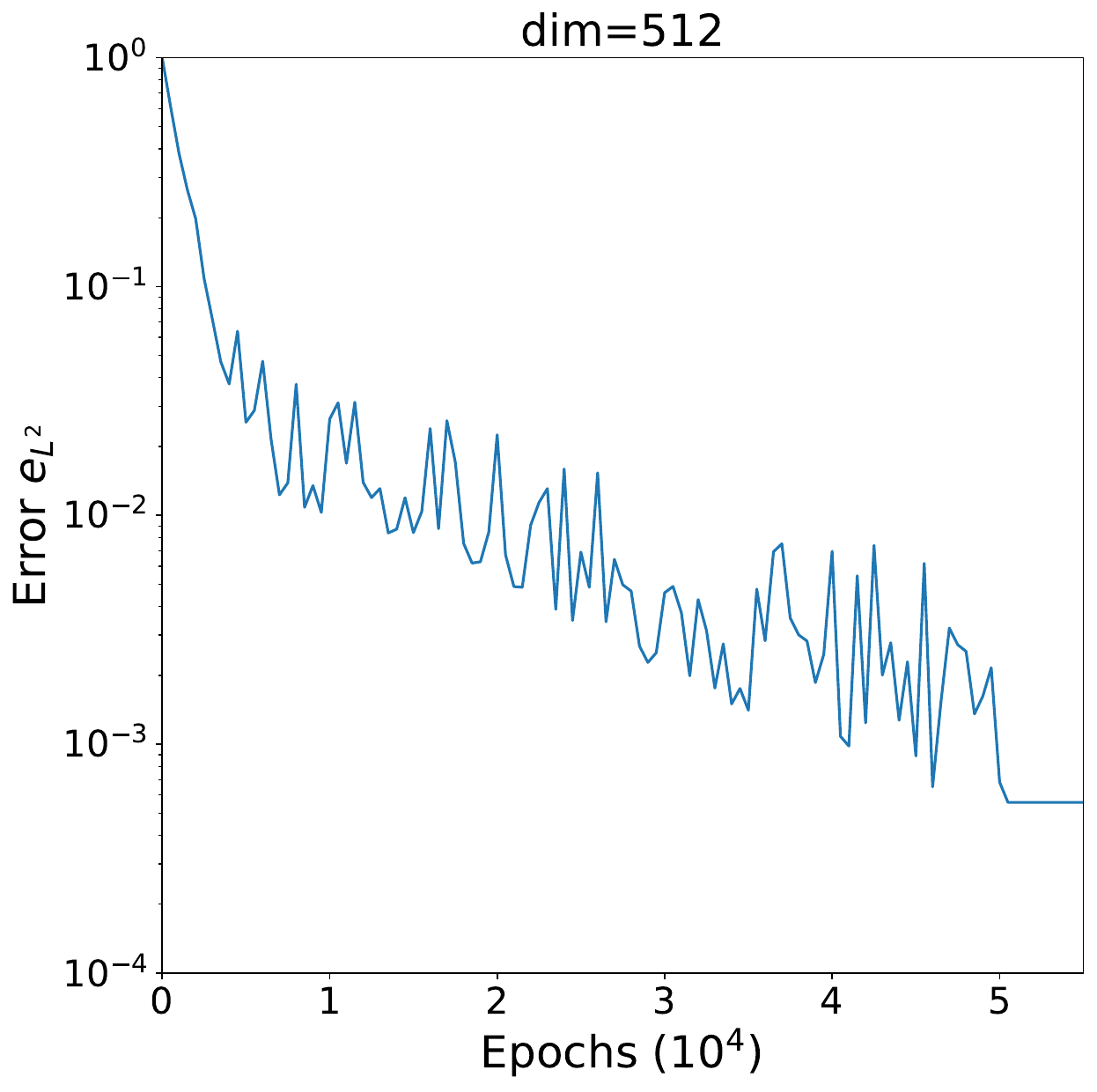}
\includegraphics[width=4cm,height=4cm]{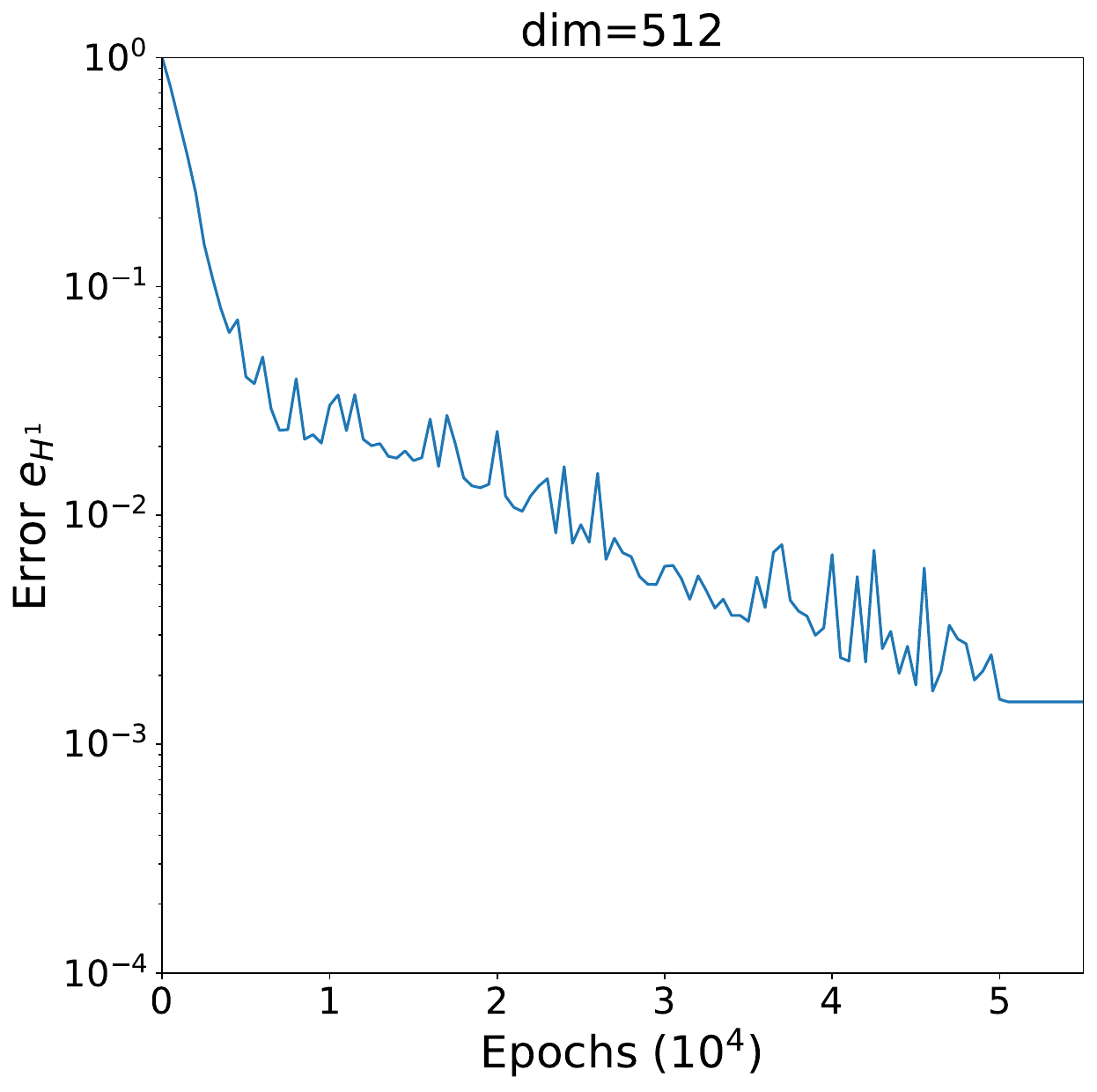}\\
\caption{\revise{Relative errors during the training process for the harmonic oscillator problem: for $d=128,256,512$.
The left column shows the relative errors of eigenvalue approximations, the middle column shows the relative $L^2$
errors and the right column shows the relative $H^1$ errors of eigenfunction approximations.}}\label{fig_harmonic_ultra_high}
\end{figure}

\begin{table}[!htb]
\caption{\revise{Errors of the harmonic oscillator problem for $d=128,256,512$.}}\label{table_harmonic_ultra_high}
\begin{center}
\begin{tabular}{ccccc}
\hline
$d$&   $e_{\lambda}$&   $e_{L^2}$&   $e_{H^1}$\\
\hline
128&   6.188e-07&   2.979e-04&   7.884e-04\\
256&   1.080e-06&   4.195e-04&   1.181e-03\\
512&   1.570e-06&   5.570e-04&   1.532e-03\\
\hline
\end{tabular}
\end{center}
\end{table}

\subsection{Eigenvalue problem with coupled harmonic oscillator}
In the third example, the potential function is defined as (\ref{coupled harmonic}).
Similar to the derivation in \cite{CHO}, the exact smallest eigenvalue is
\begin{eqnarray*}
\lambda_0=\sum_{i=1}^d\sqrt{1-\cos\Big(\frac{i\pi}{d+1}\Big)},
\end{eqnarray*}
and the exact eigenfunction has Gaussian form.
In this example, we test the case of $d=4$. Then the exact eigenfunction is
\begin{eqnarray*}\label{exact_coupled}
&&u(x_1,x_2,x_3,x_4)\nonumber\\
&=&\exp\Big[-\frac{1}{2}(\omega_1a^2+\omega_2b^2+\omega_3a^2+\omega_4b^2)(x_1^2+x_3^2)
-\frac{1}{2}(\omega_1b^2+\omega_2a^2+\omega_3b^2+\omega_4a^2)(x_2^2+x_4^2)\nonumber\\
&&-ab(-\omega_1-\omega_2+\omega_3+\omega_4)(x_1x_2+x_3x_4)
-ab(\omega_1-\omega_2+\omega_3-\omega_4)(x_1x_4+x_2x_3)\nonumber\\
&&-(-\omega_1a^2+\omega_2b^2+\omega_3a^2-\omega_4b^2)x_1x_3
-(-\omega_1b^2+\omega_2a^2+\omega_3b^2-\omega_4a^2)x_2x_4 \Big],
\end{eqnarray*}
where $a=\sqrt{5-\sqrt{5}}/(2\sqrt{5})$, $b=\sqrt{5+\sqrt{5}}/(2\sqrt{5})$, $\omega_1=\sqrt{5+\sqrt{5}}/2$, $\omega_2=\sqrt{3+\sqrt{5}}/2$, $\omega_3=\sqrt{5-\sqrt{5}}/2$, $\omega_4=\sqrt{3-\sqrt{5}}/2$.
To demonstrate the efficiency of the proposed method, we take hyperparameter $p$ from 1 to 30.
\revise{Each dimension of the TNN is an FNN with two hidden layers and  each hidden layer has 50 hidden neurons.
We train the TNN to investigate the dependence of
the convergence behavior on the hyperparameter $p$.
All the cases are trained by the Adam optimizer with
the same learning rate of 0.001 and epochs of 500000.}
The computational domain is truncated to $[-5,5]^d$. And we decompose the interval $[-5,5]$
in each dimensional into 100 equal subintervals
and choose 16 Gauss points on each subinterval. The relative errors $e_\lambda$ during
the training process for different $p$ are shown in
Figure \ref{fig_coupled} and Figure \ref{fig_coupled_error_p} demonstrates how the final error
changes as $p$ increases. From Figures \ref{fig_coupled} and \ref{fig_coupled_error_p}, we can
find that the proposed method converges at an impressive accuracy.

\begin{figure}[htb!]
\centering
\includegraphics[width=2.75cm,height=2.5cm]{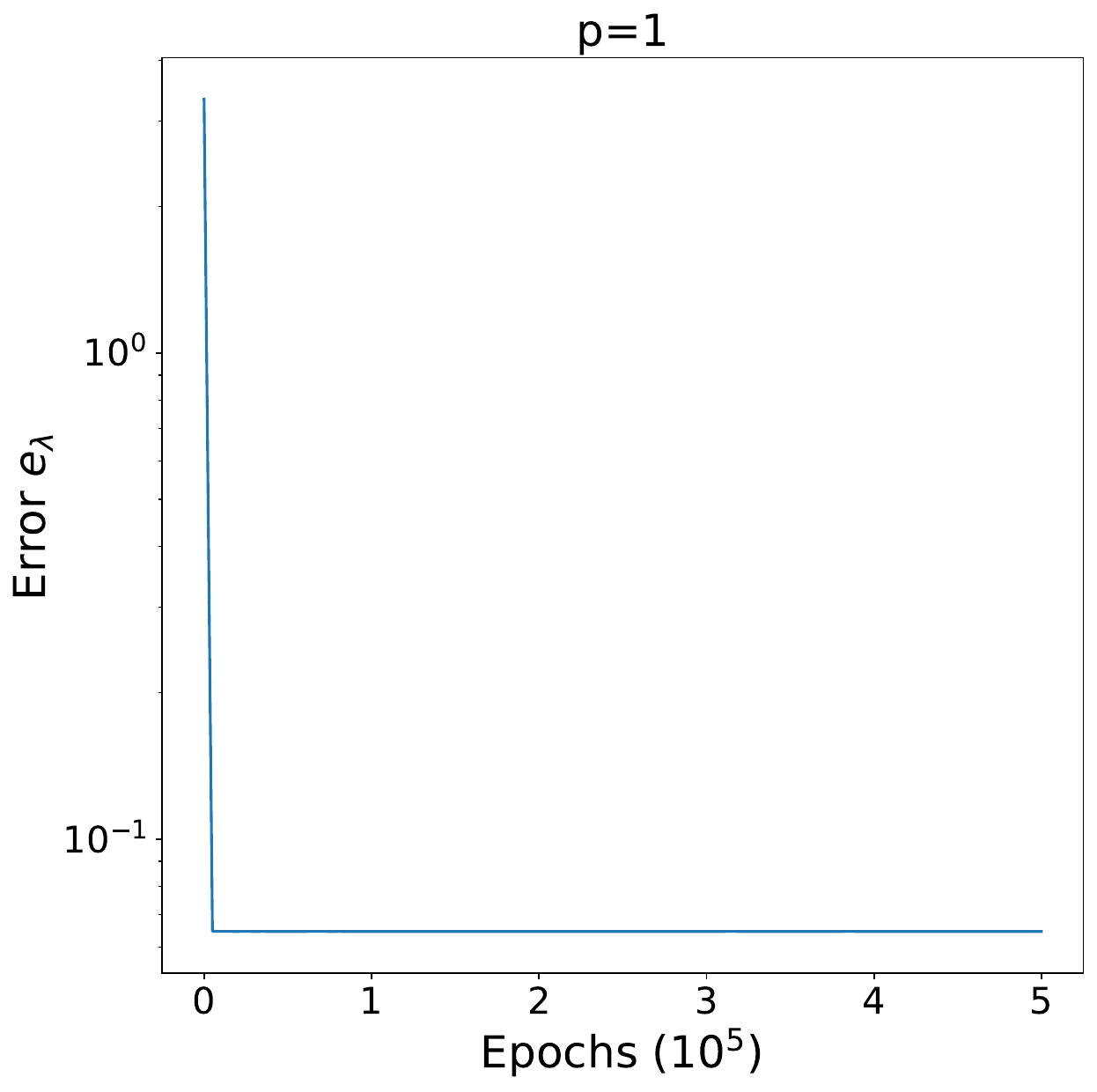}
\includegraphics[width=2.75cm,height=2.5cm]{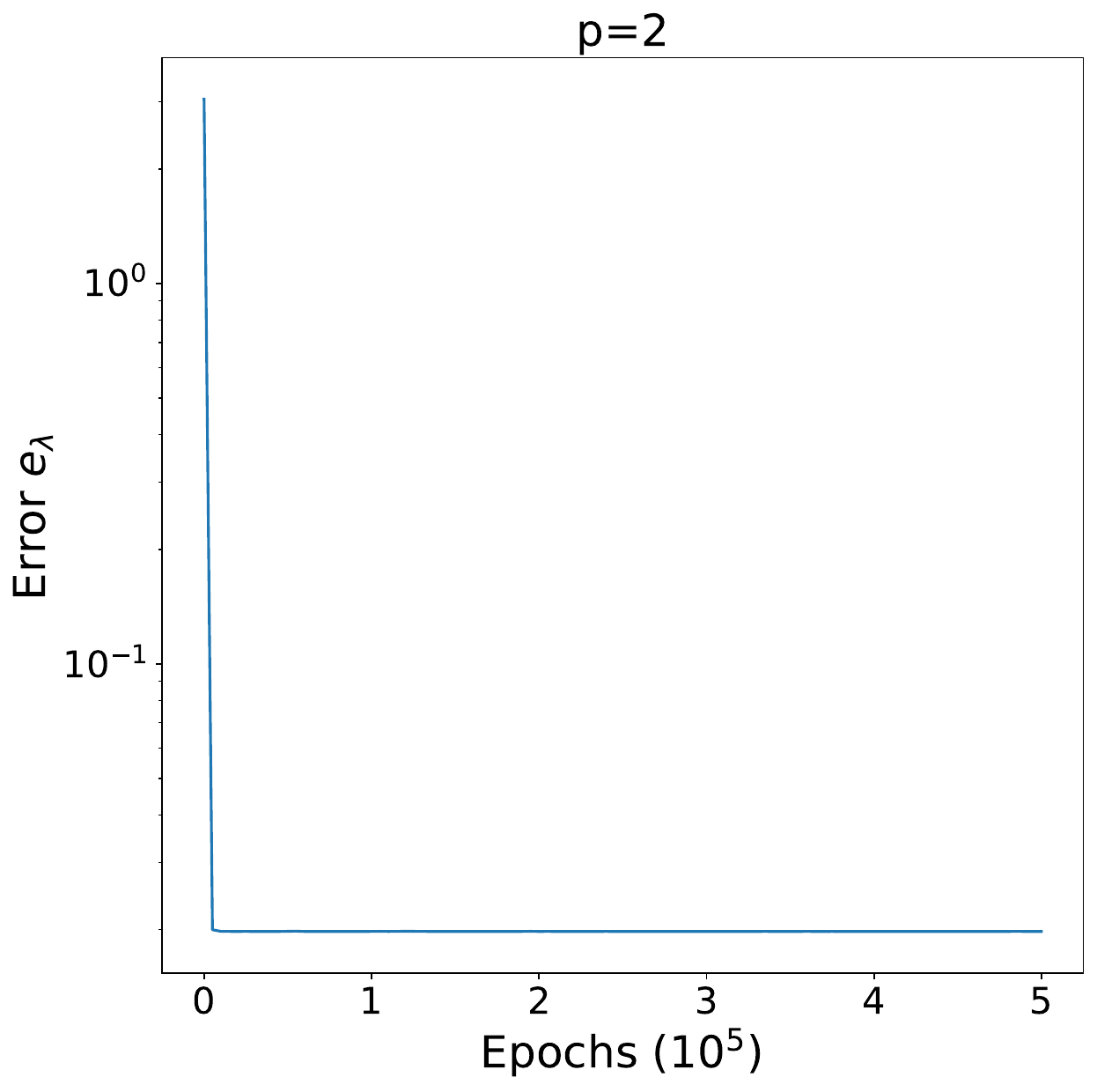}
\includegraphics[width=2.75cm,height=2.5cm]{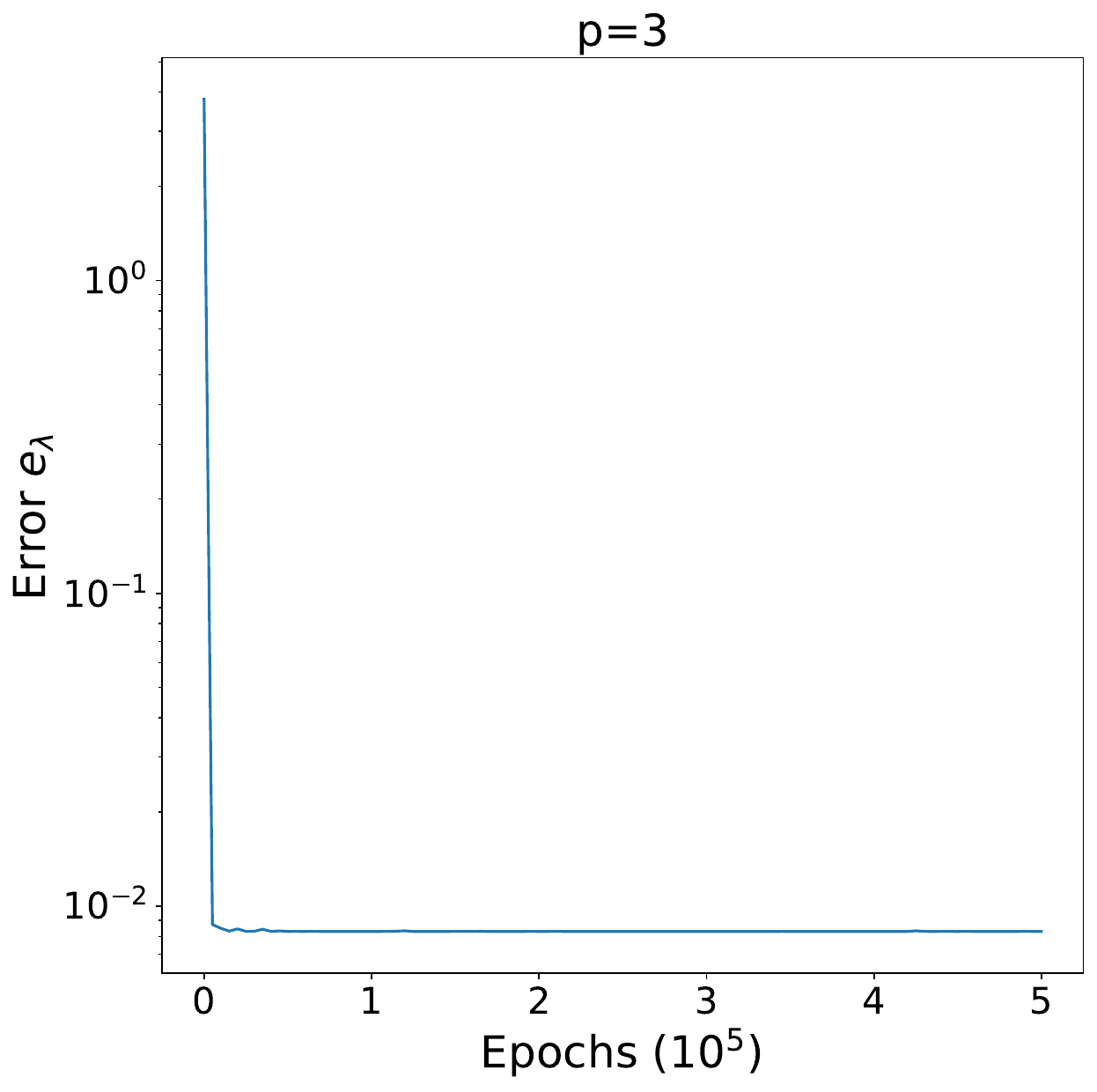}
\includegraphics[width=2.75cm,height=2.5cm]{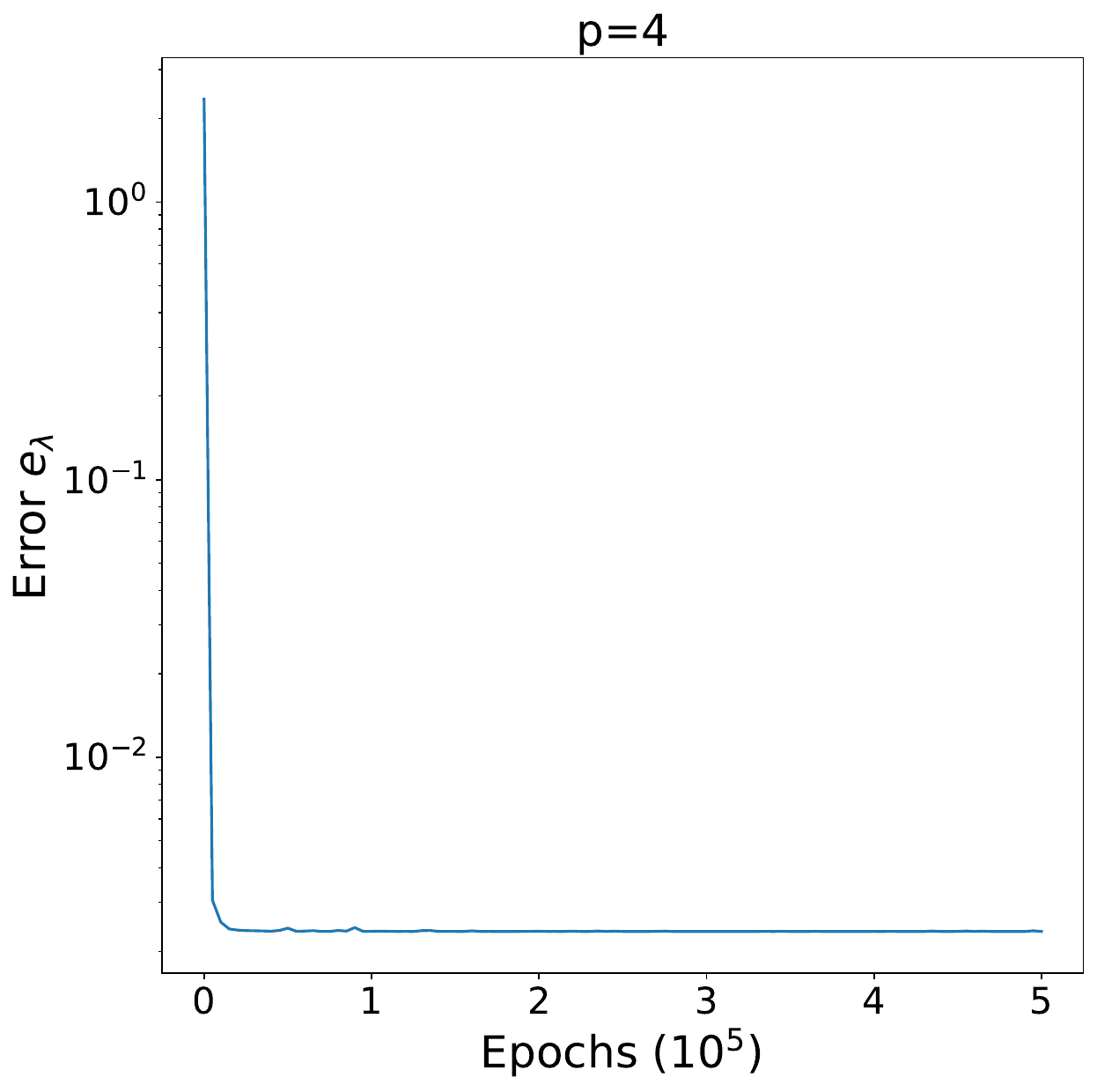}
\includegraphics[width=2.75cm,height=2.5cm]{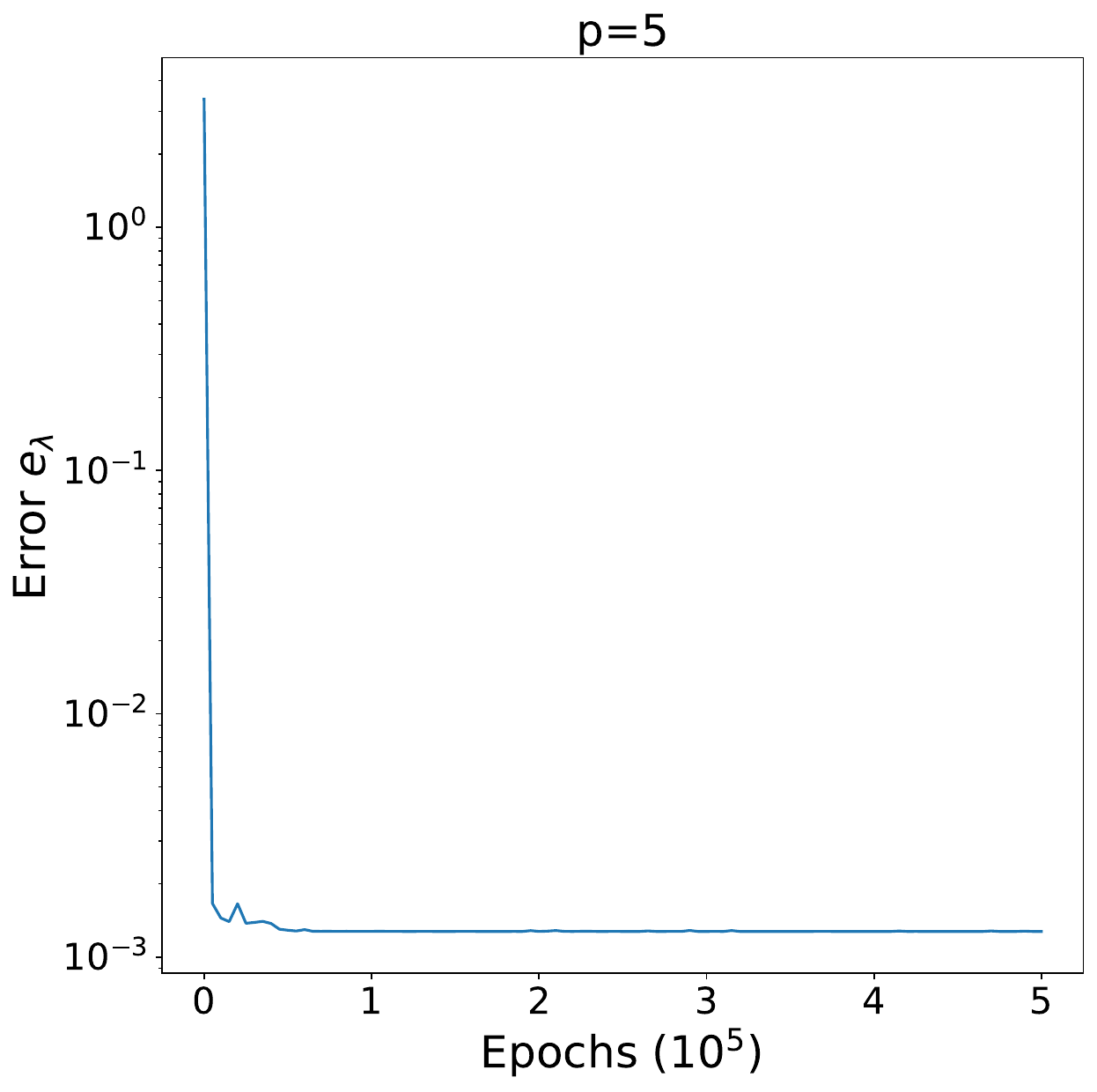}\\
\includegraphics[width=2.75cm,height=2.5cm]{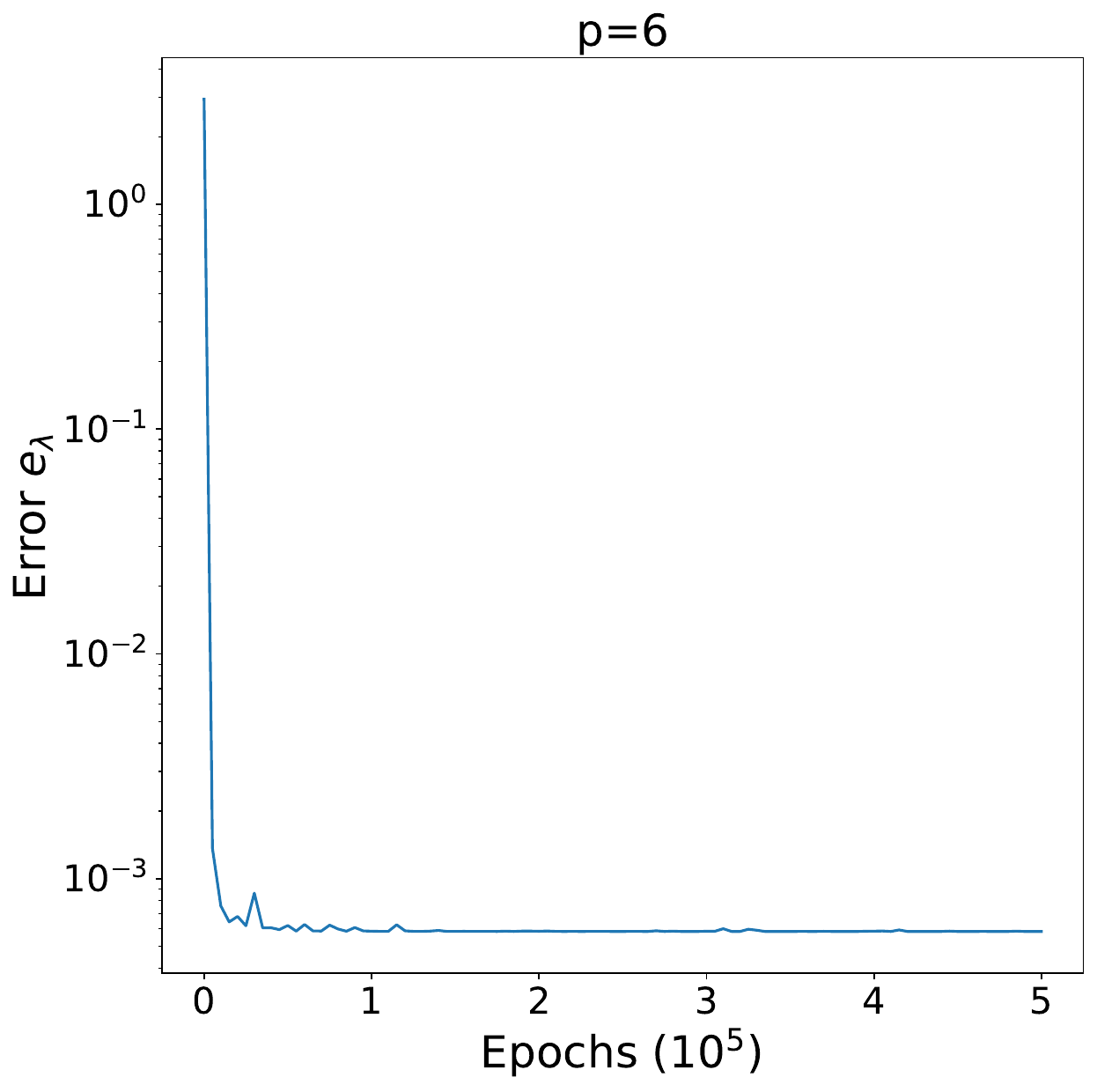}
\includegraphics[width=2.75cm,height=2.5cm]{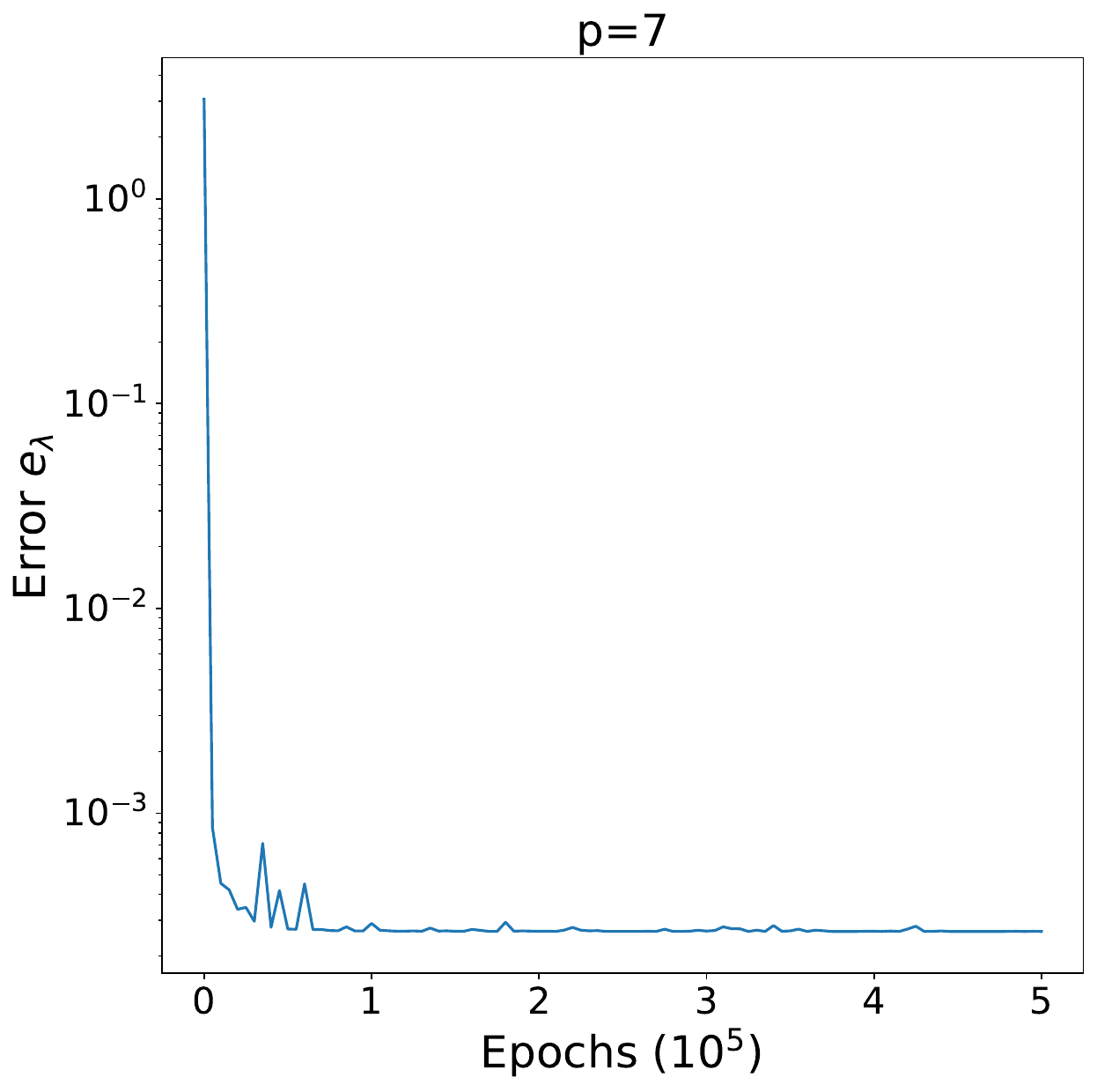}
\includegraphics[width=2.75cm,height=2.5cm]{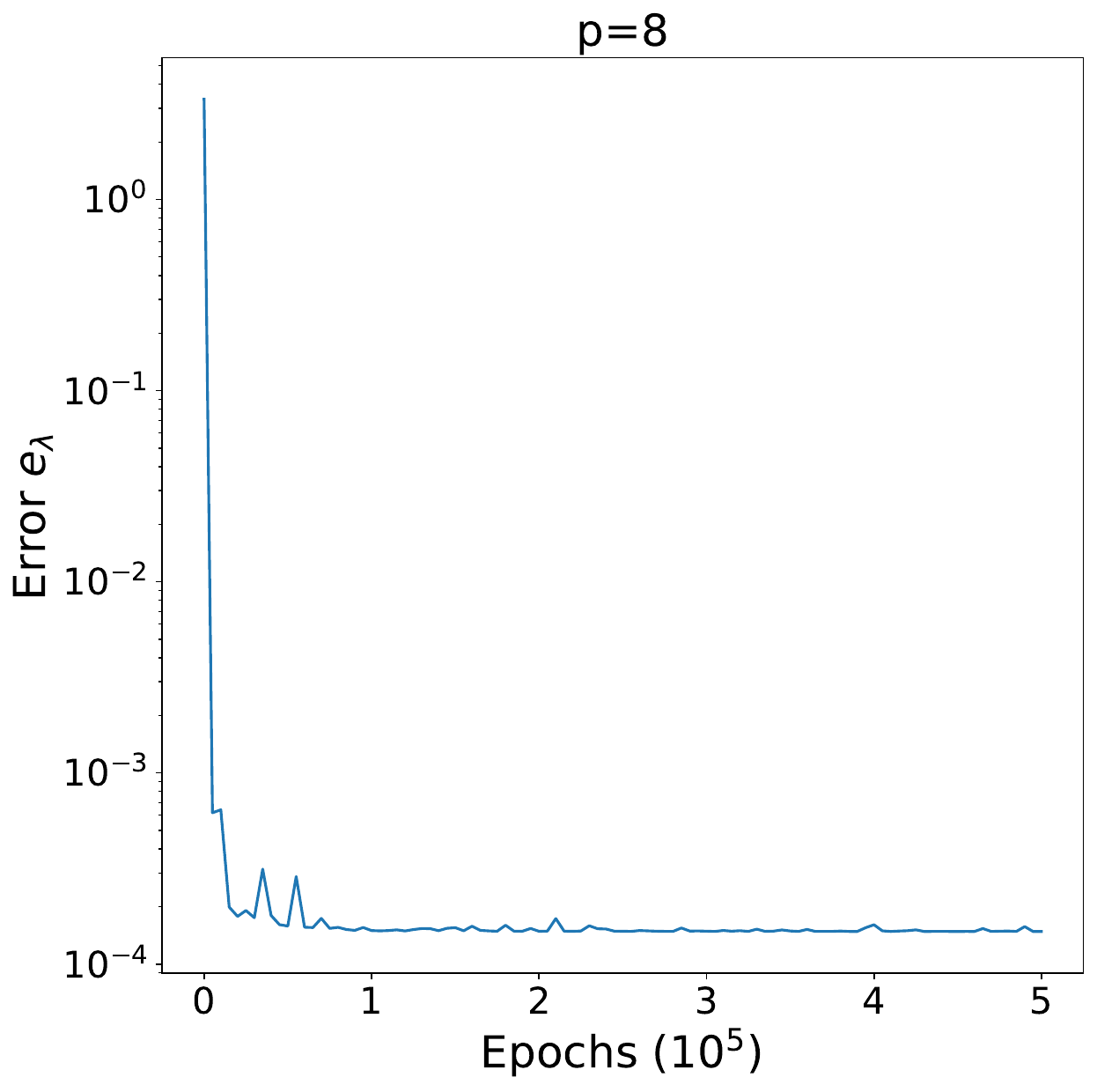}
\includegraphics[width=2.75cm,height=2.5cm]{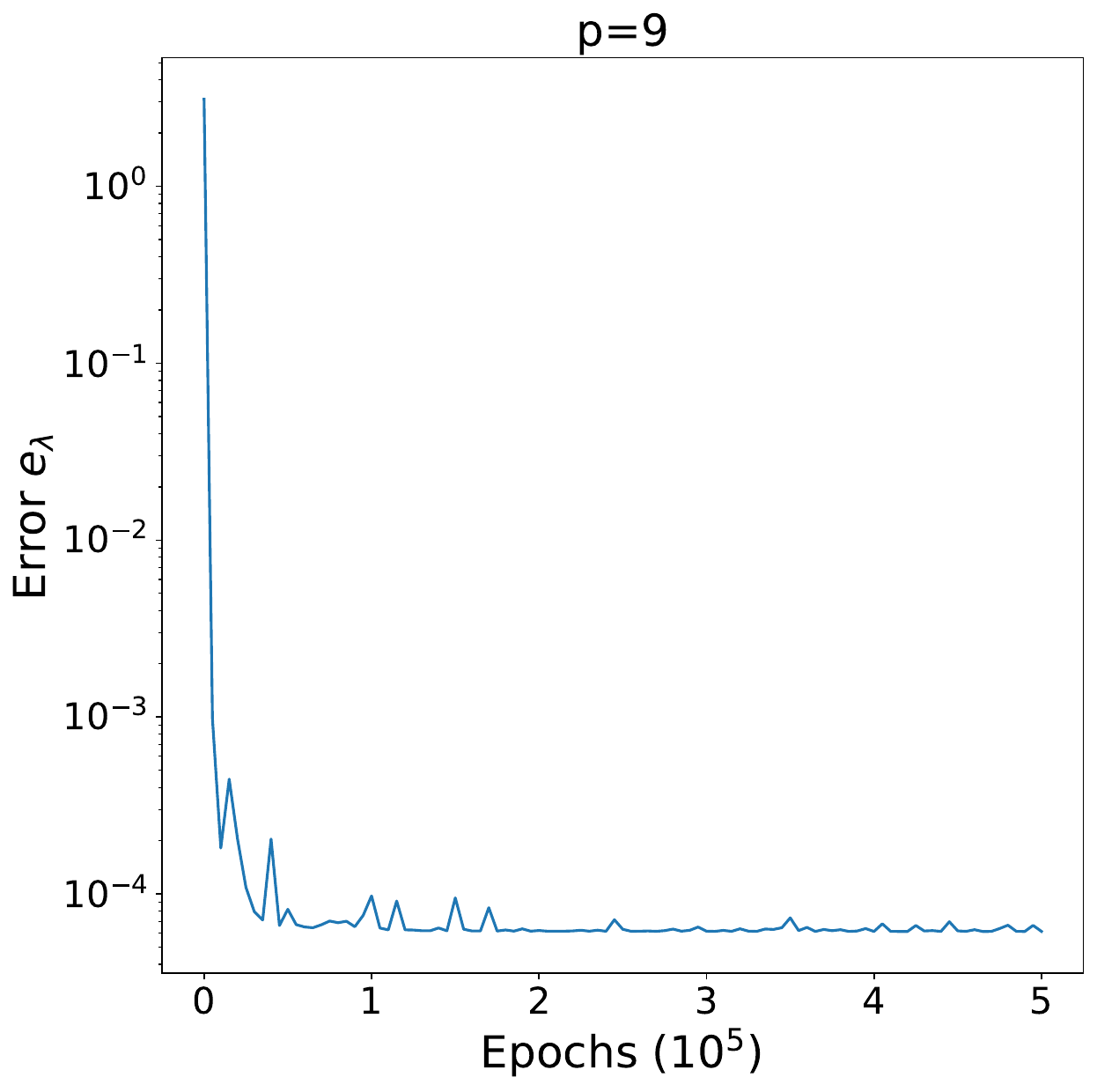}
\includegraphics[width=2.75cm,height=2.5cm]{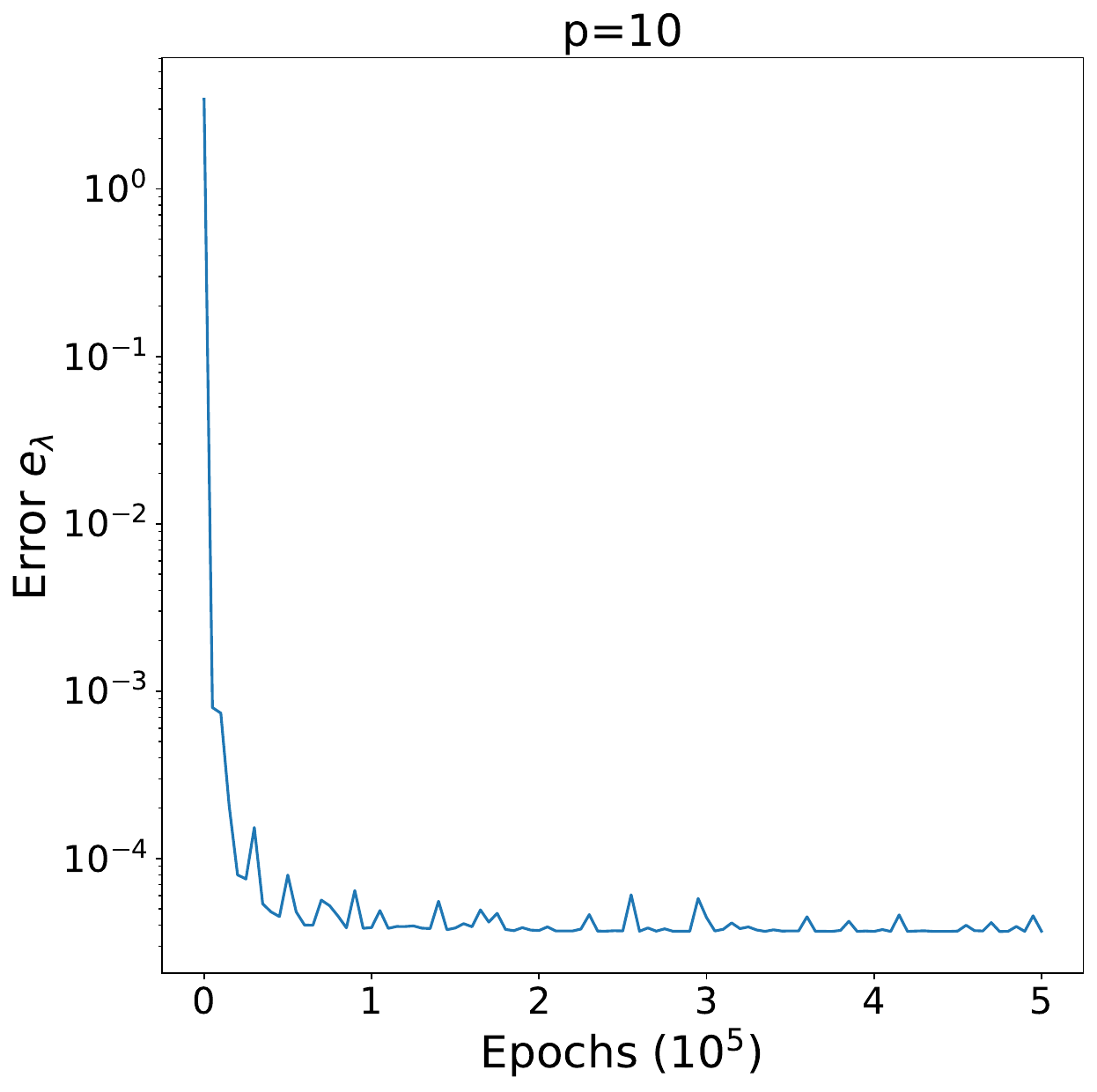}\\
\includegraphics[width=2.75cm,height=2.5cm]{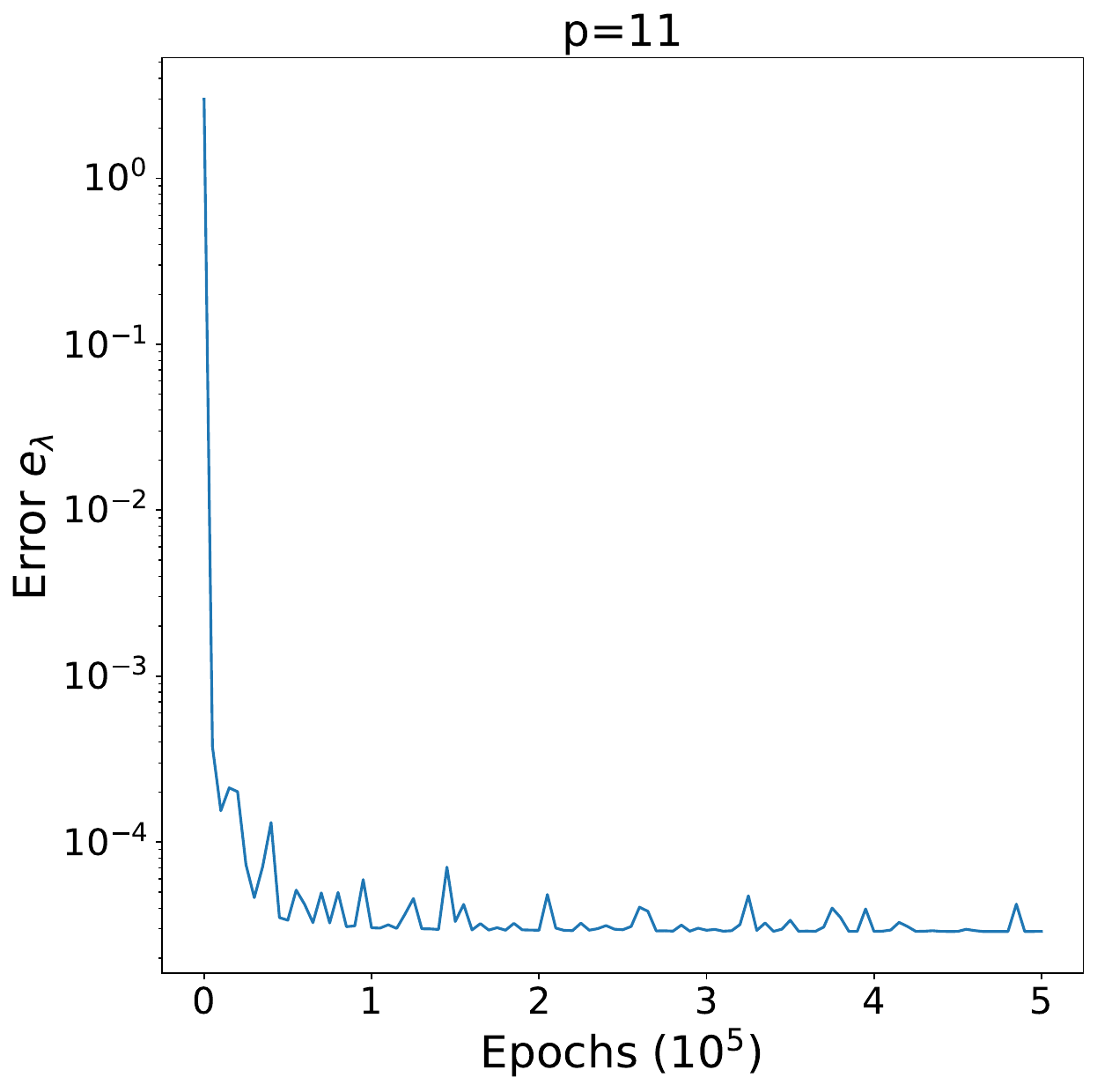}
\includegraphics[width=2.75cm,height=2.5cm]{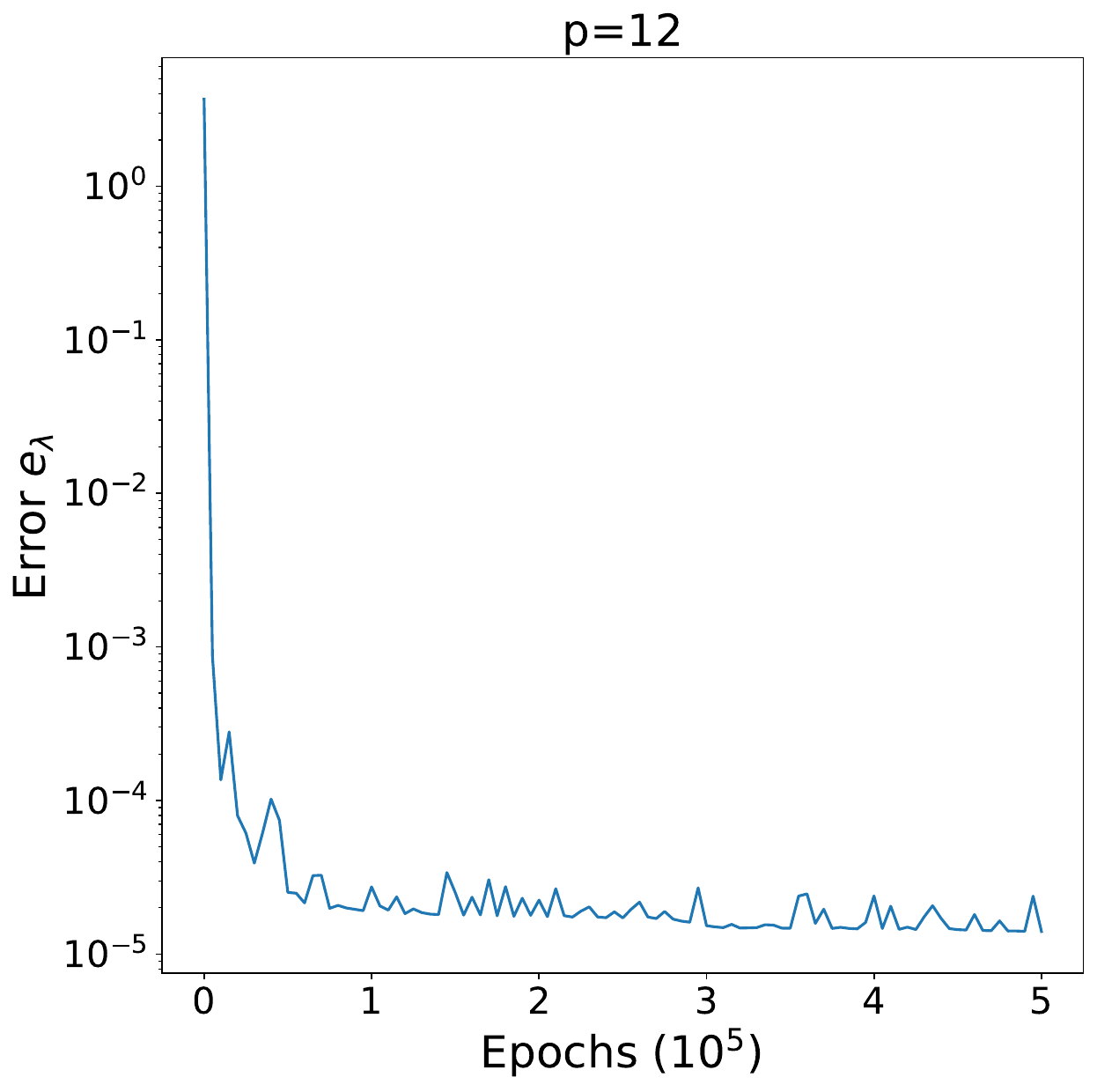}
\includegraphics[width=2.75cm,height=2.5cm]{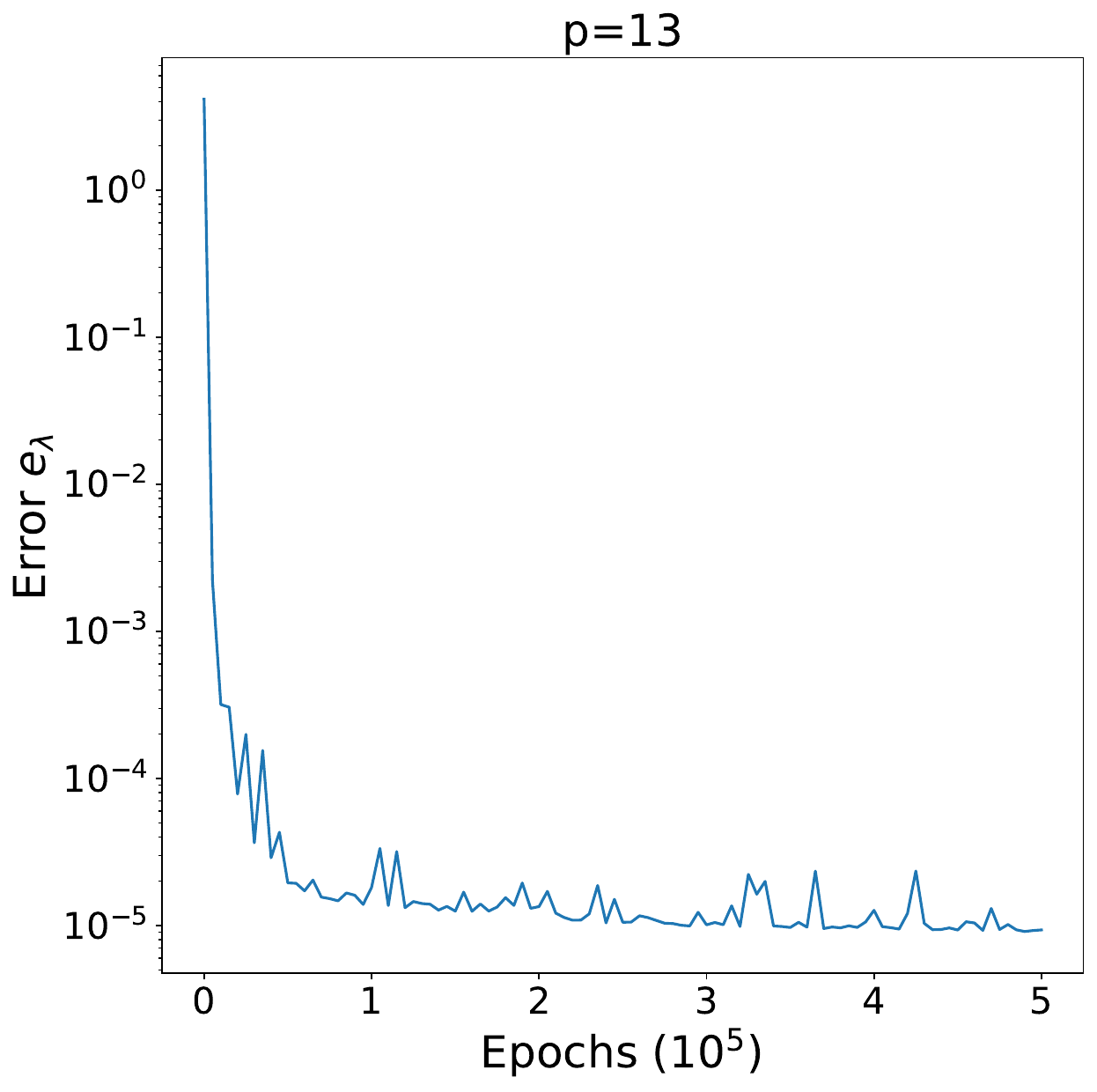}
\includegraphics[width=2.75cm,height=2.5cm]{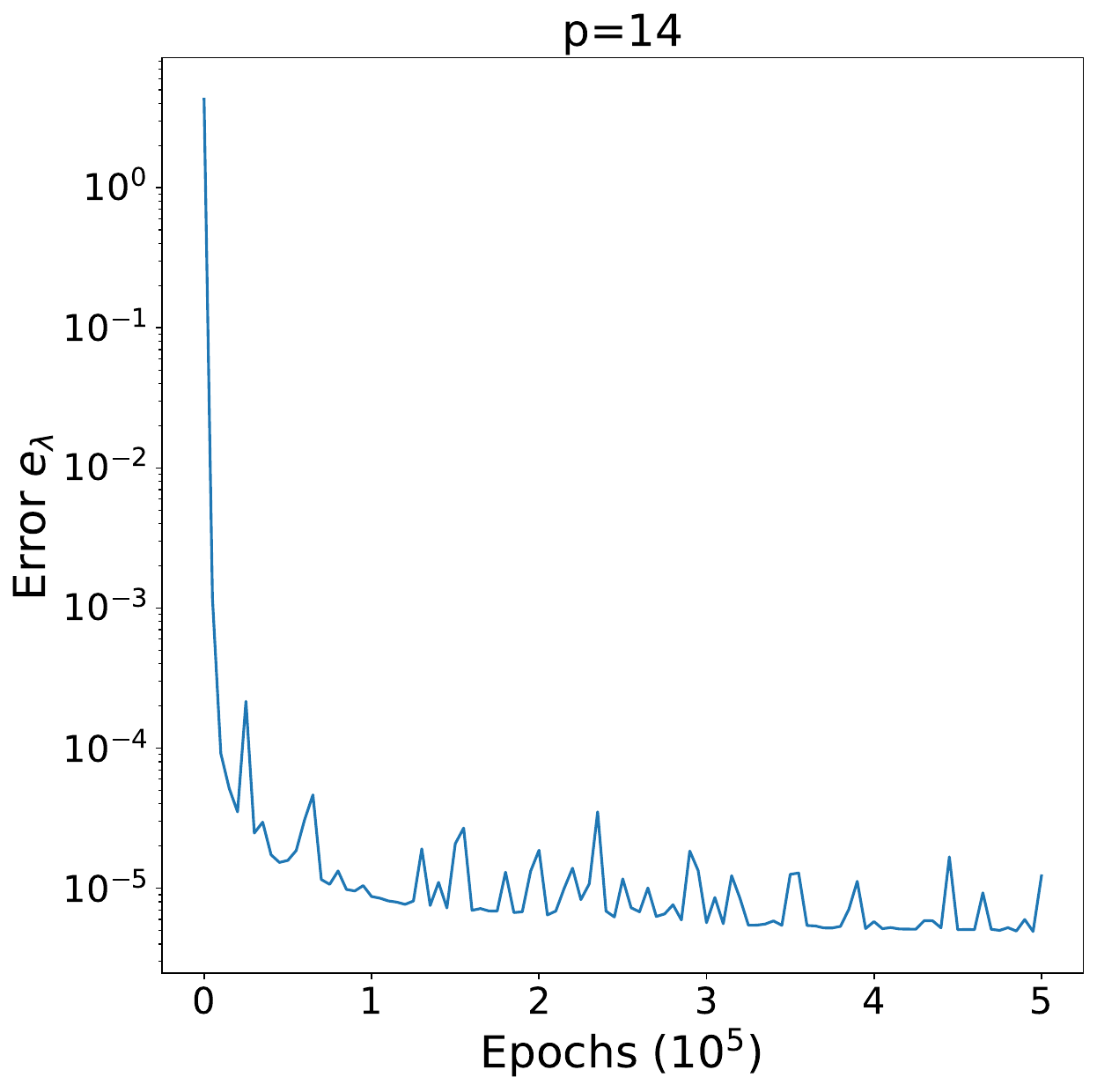}
\includegraphics[width=2.75cm,height=2.5cm]{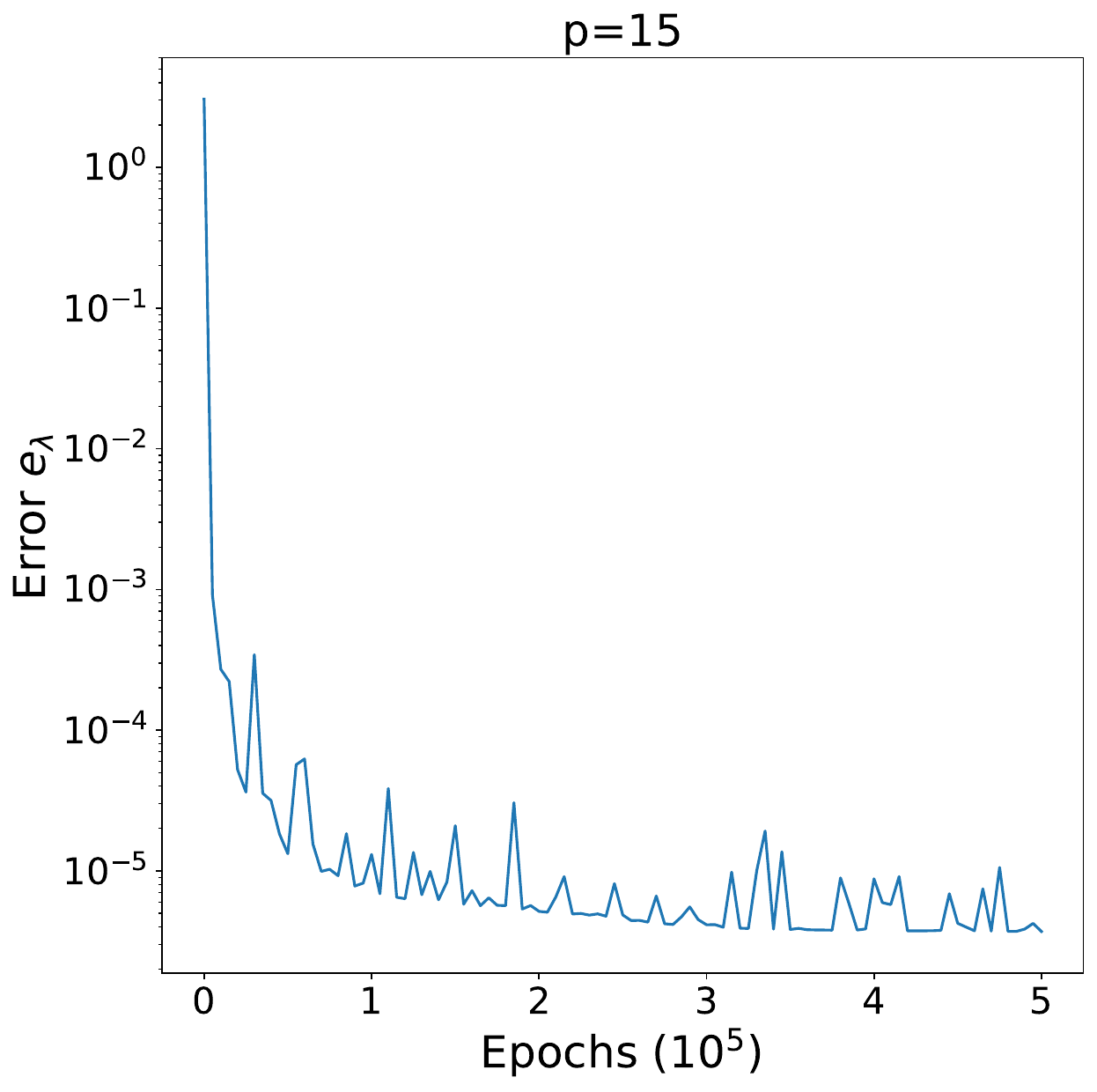}\\
\includegraphics[width=2.75cm,height=2.5cm]{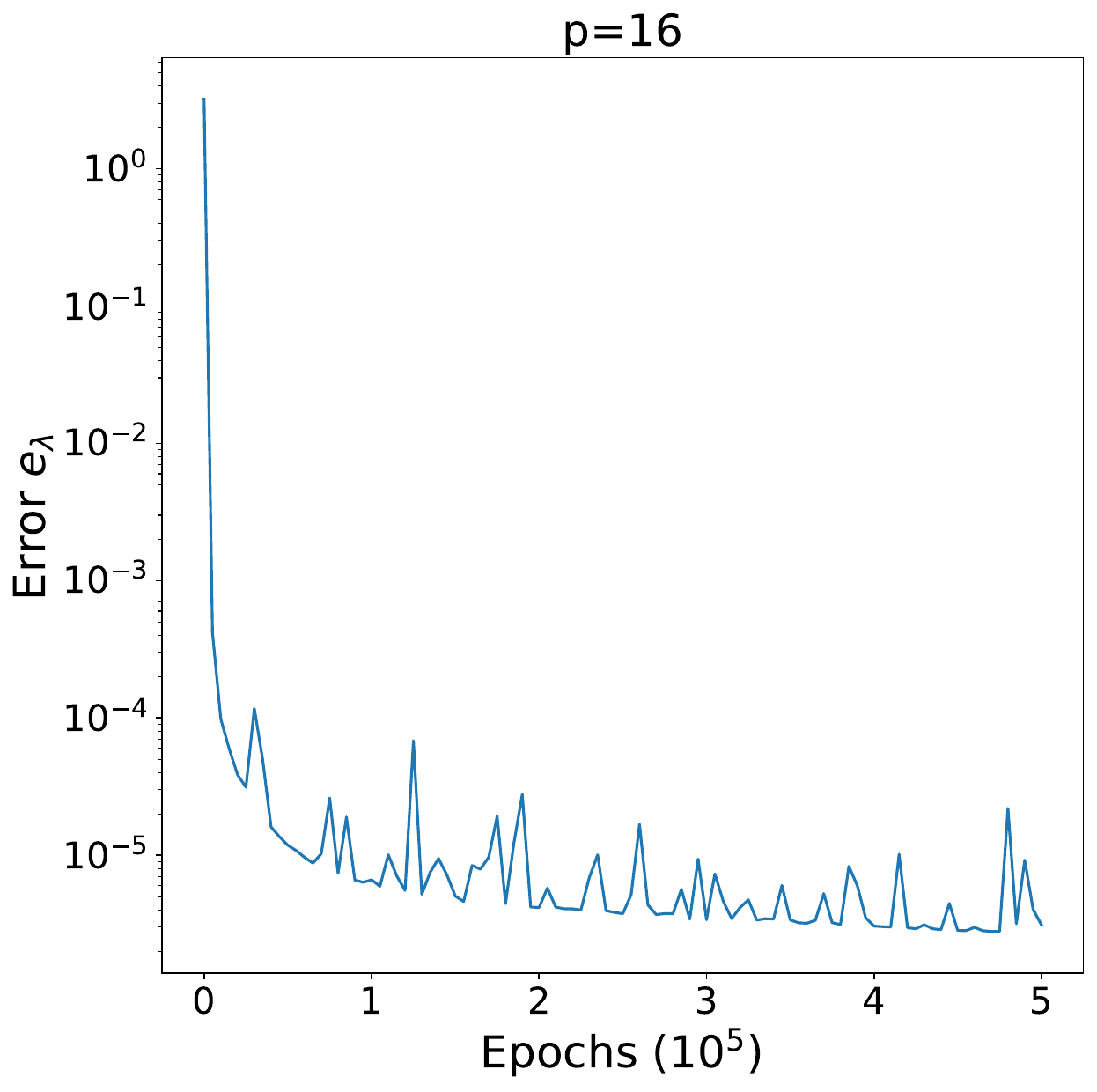}
\includegraphics[width=2.75cm,height=2.5cm]{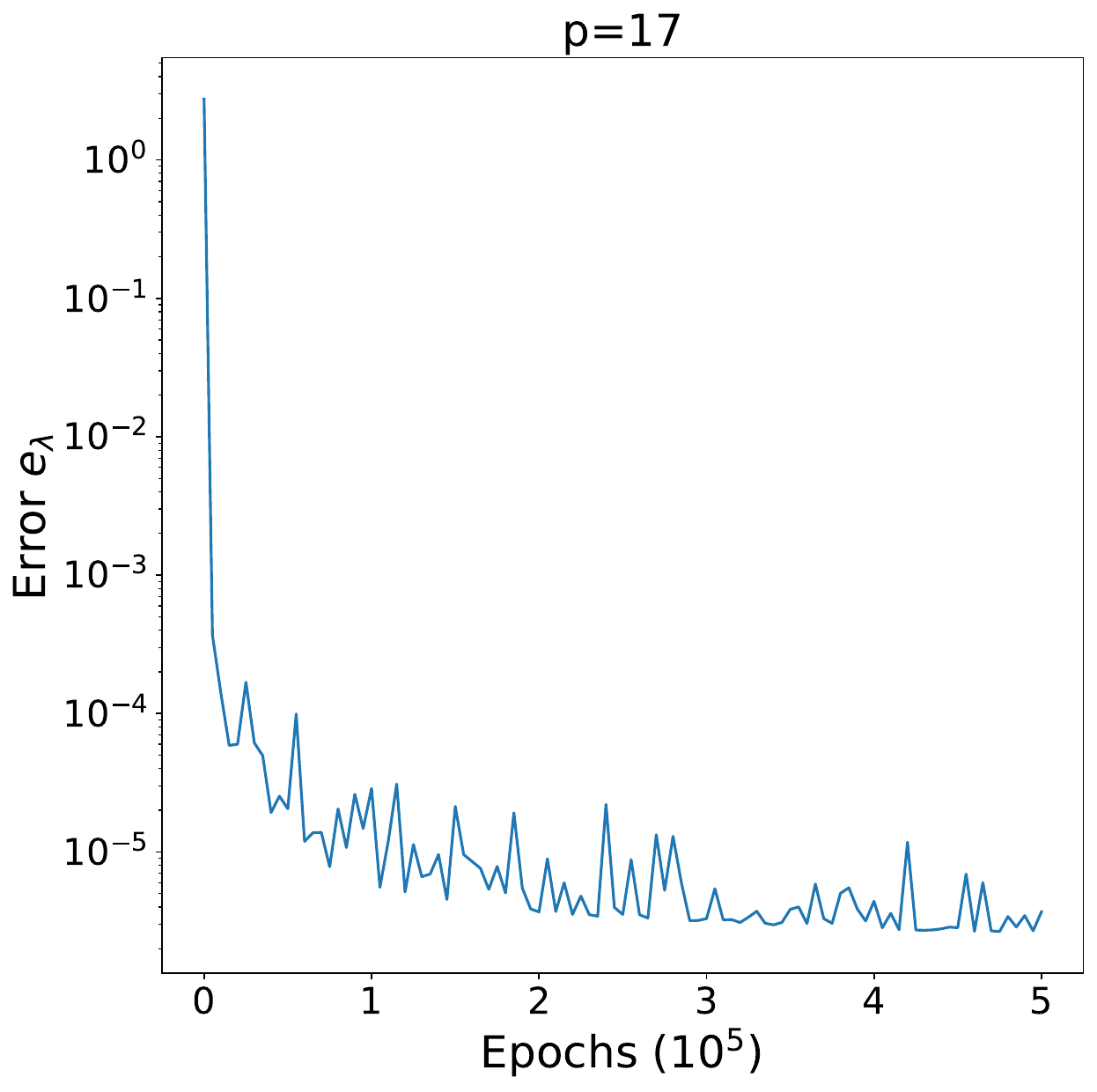}
\includegraphics[width=2.75cm,height=2.5cm]{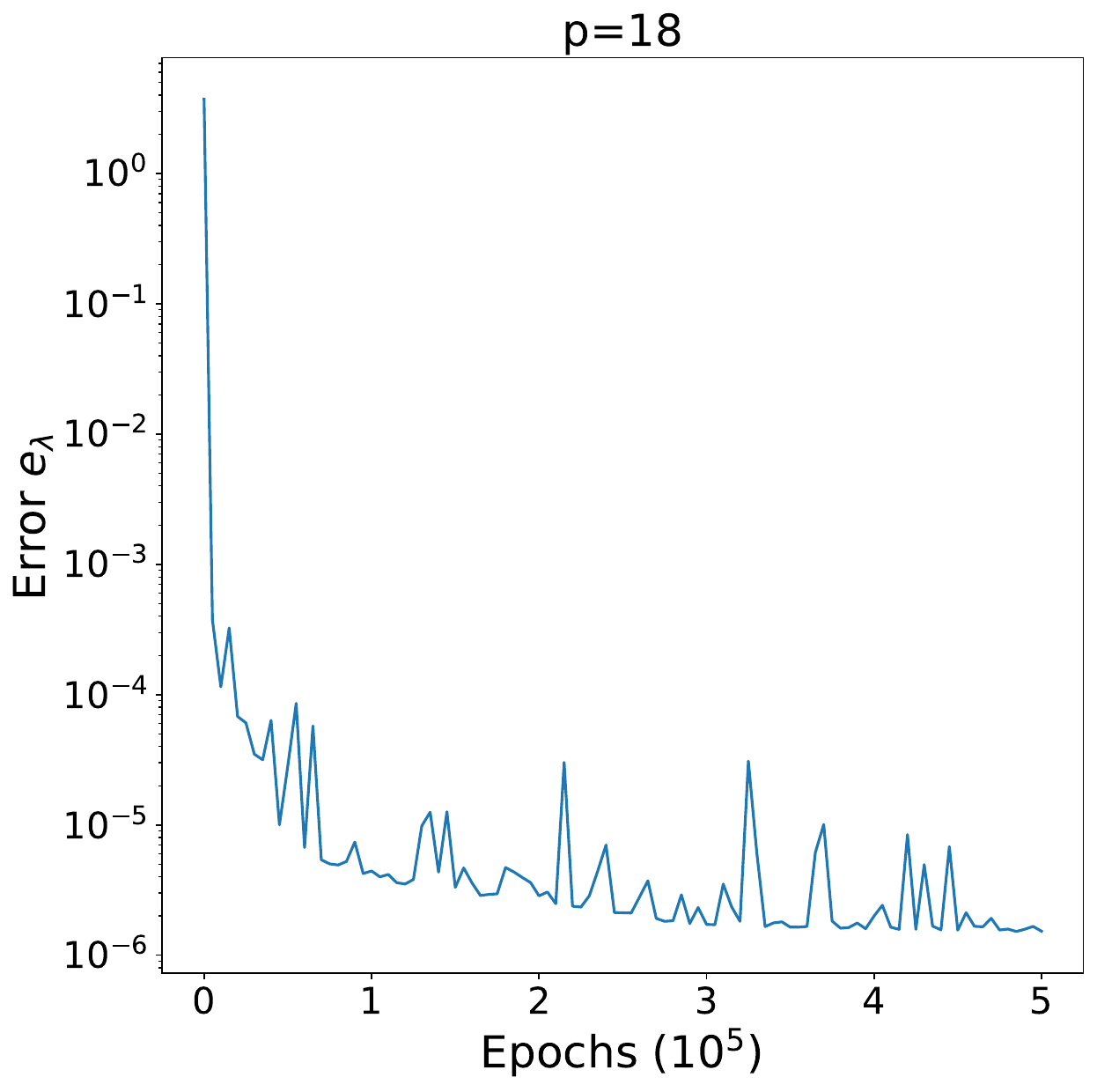}
\includegraphics[width=2.75cm,height=2.5cm]{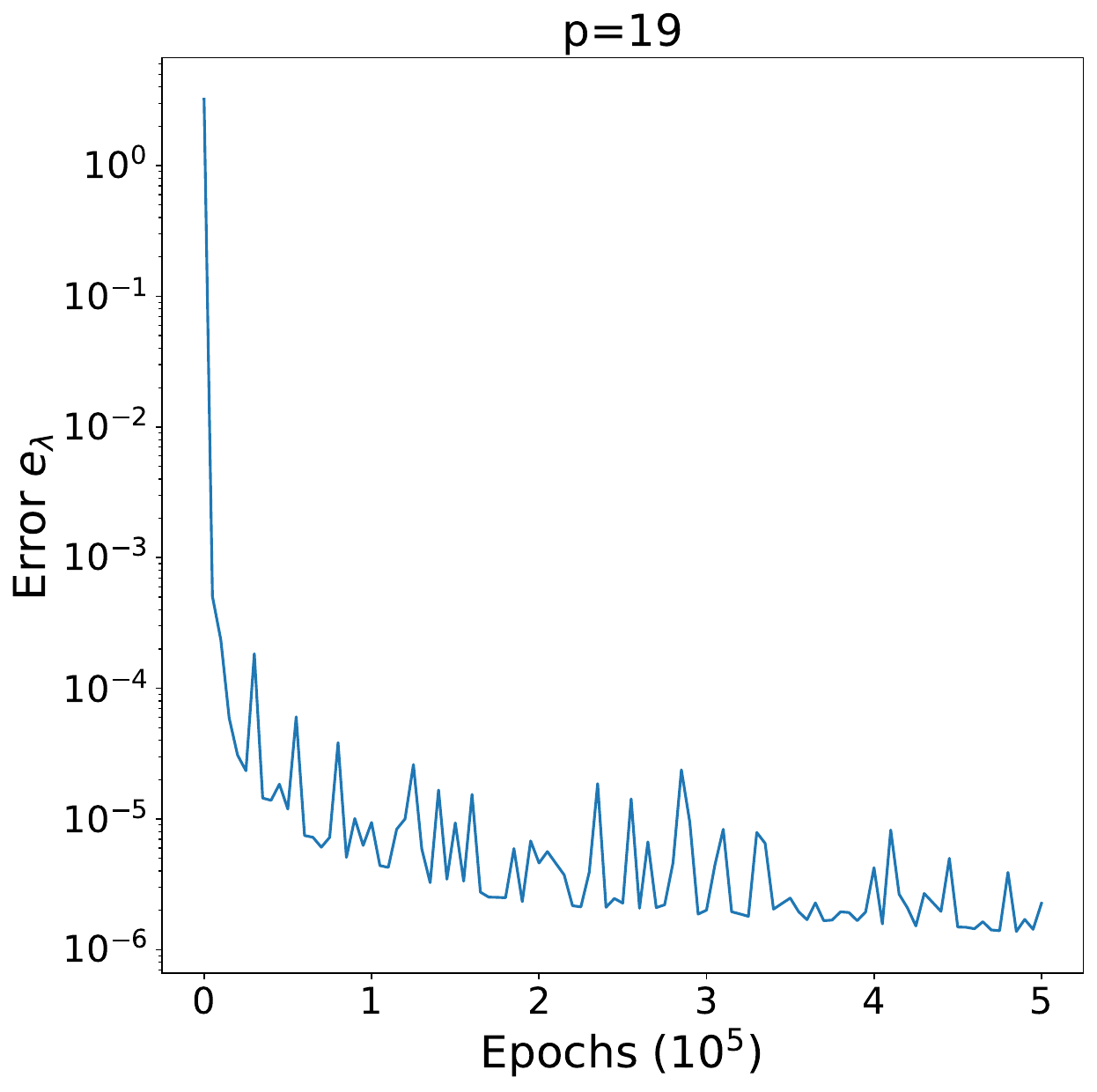}
\includegraphics[width=2.75cm,height=2.5cm]{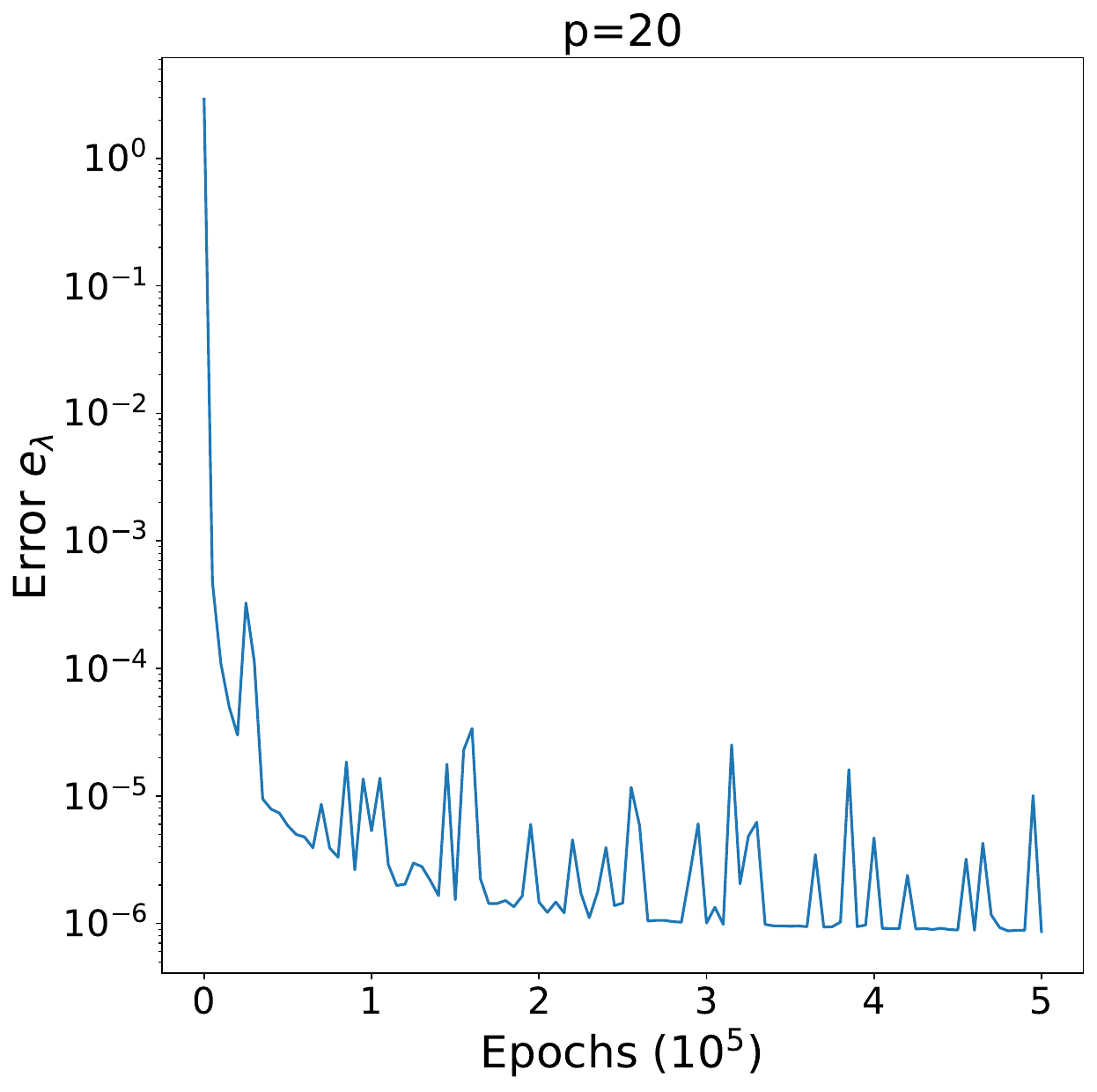}\\
\includegraphics[width=2.75cm,height=2.5cm]{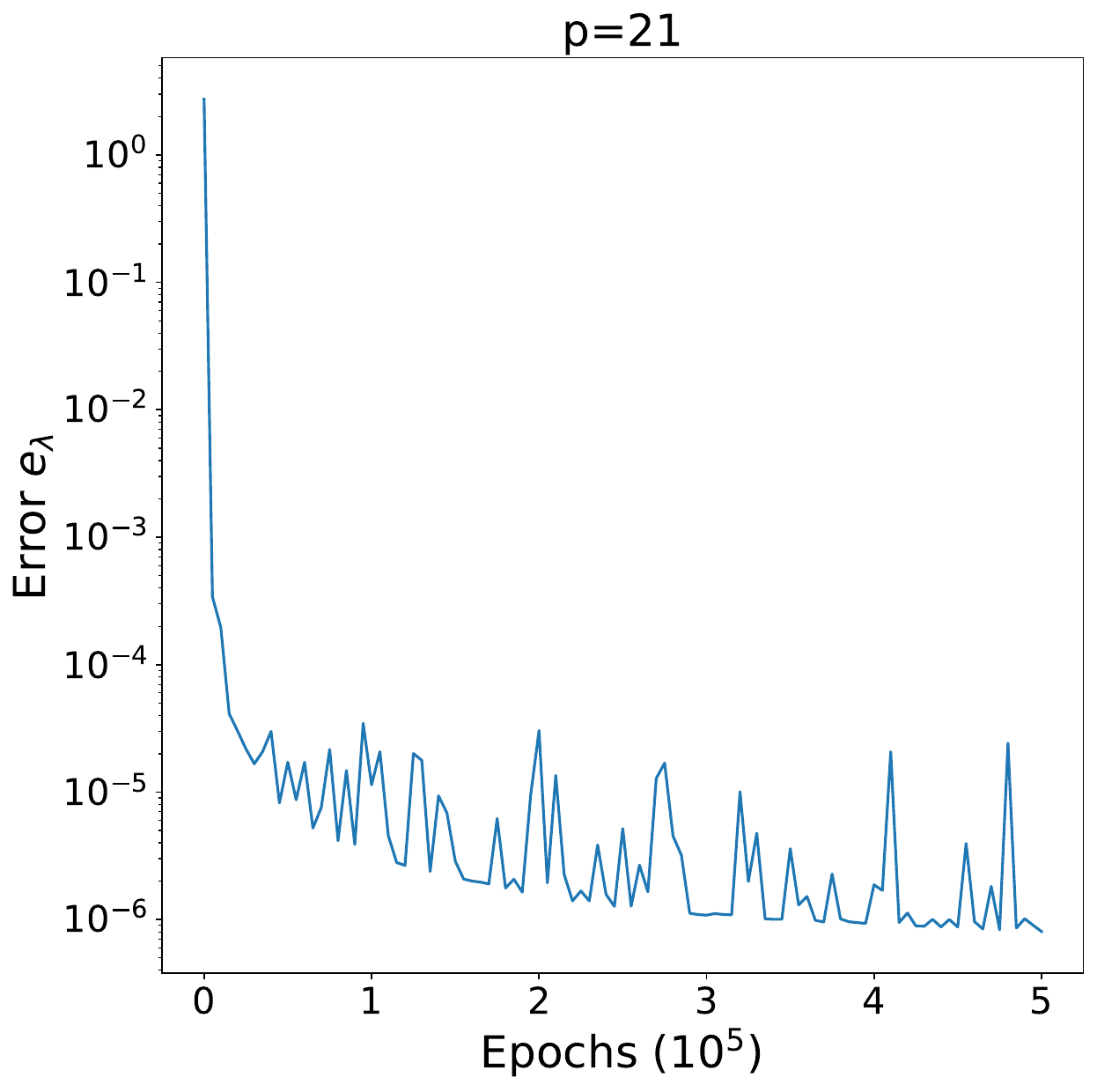}
\includegraphics[width=2.75cm,height=2.5cm]{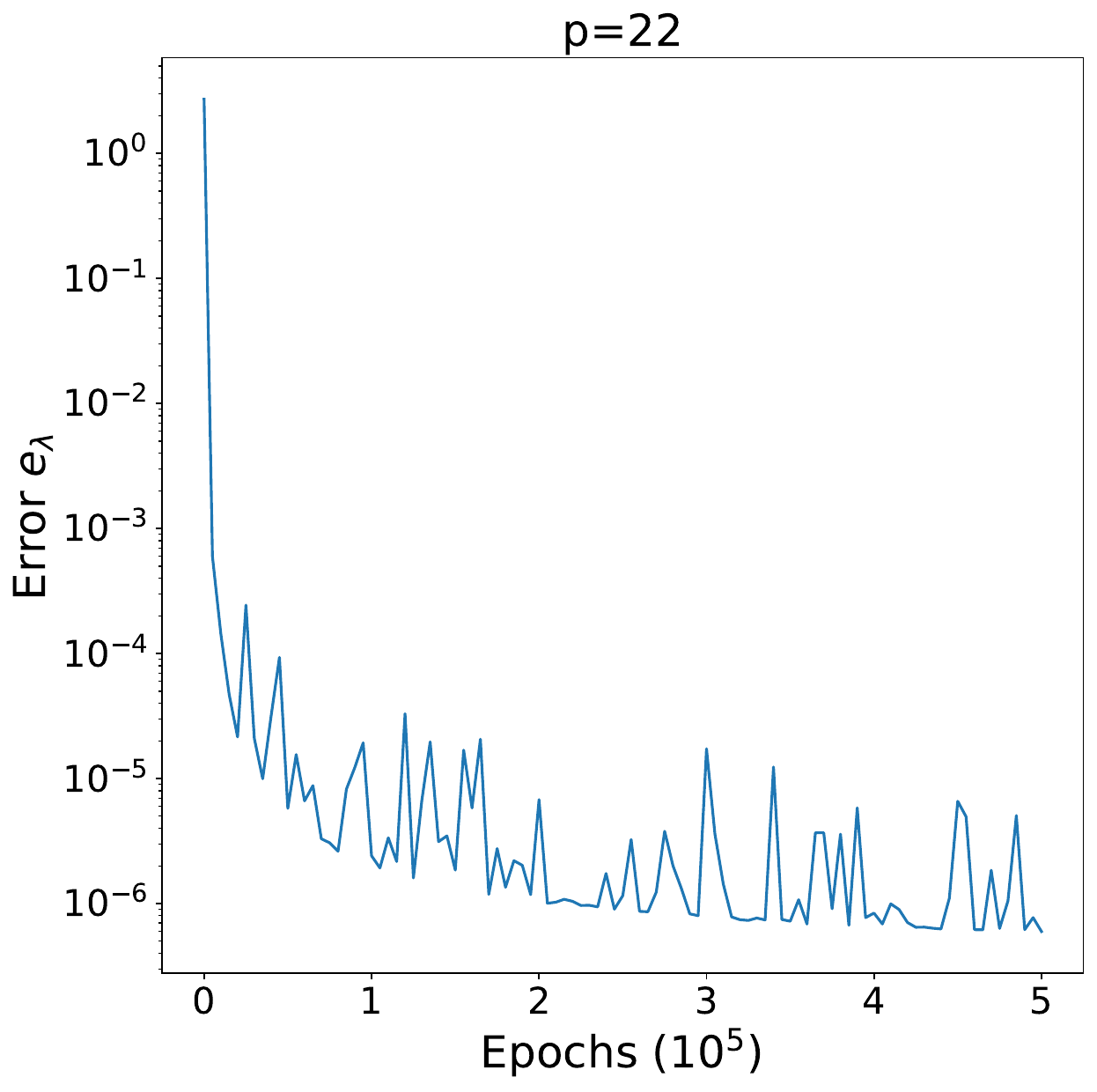}
\includegraphics[width=2.75cm,height=2.5cm]{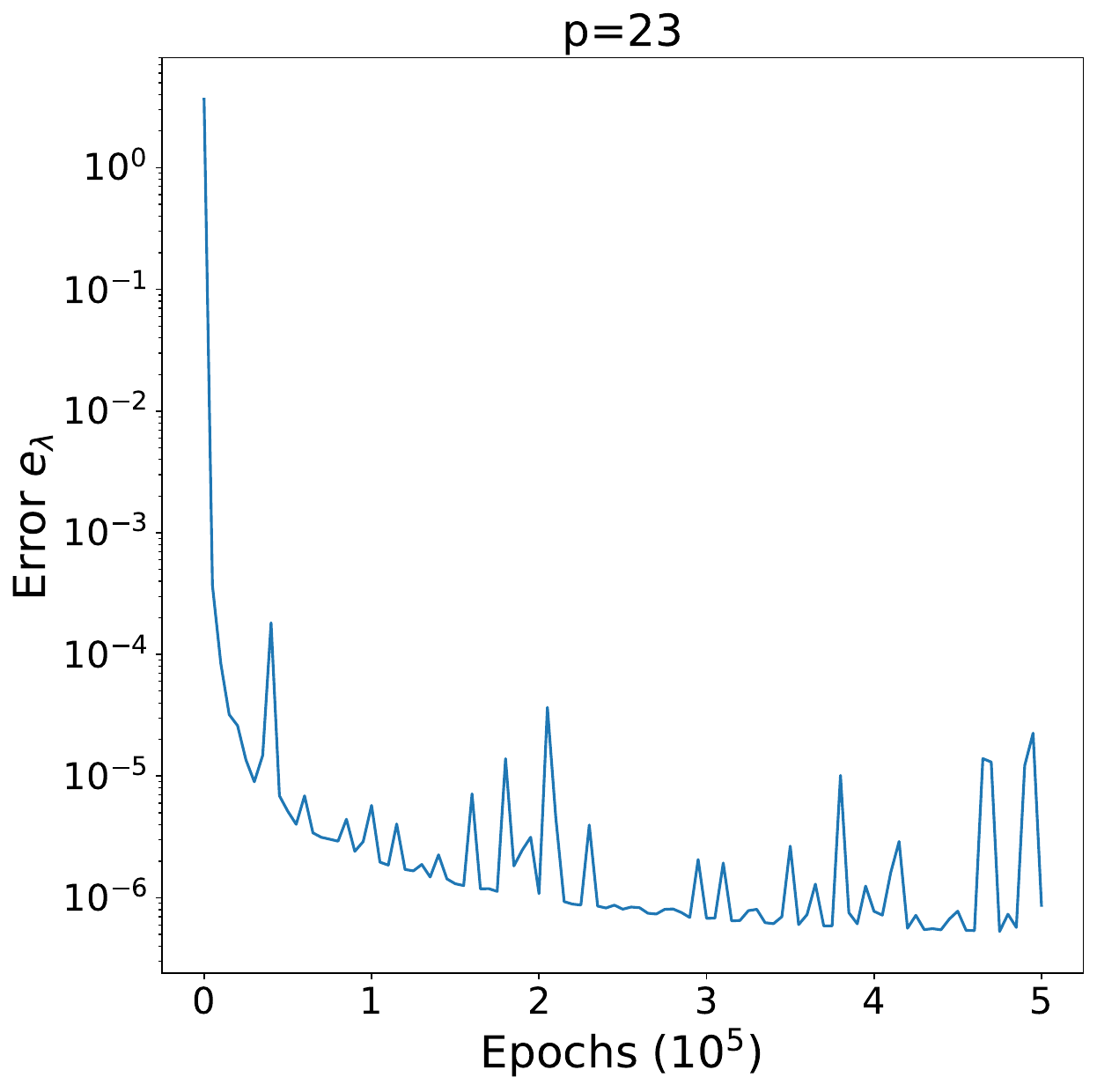}
\includegraphics[width=2.75cm,height=2.5cm]{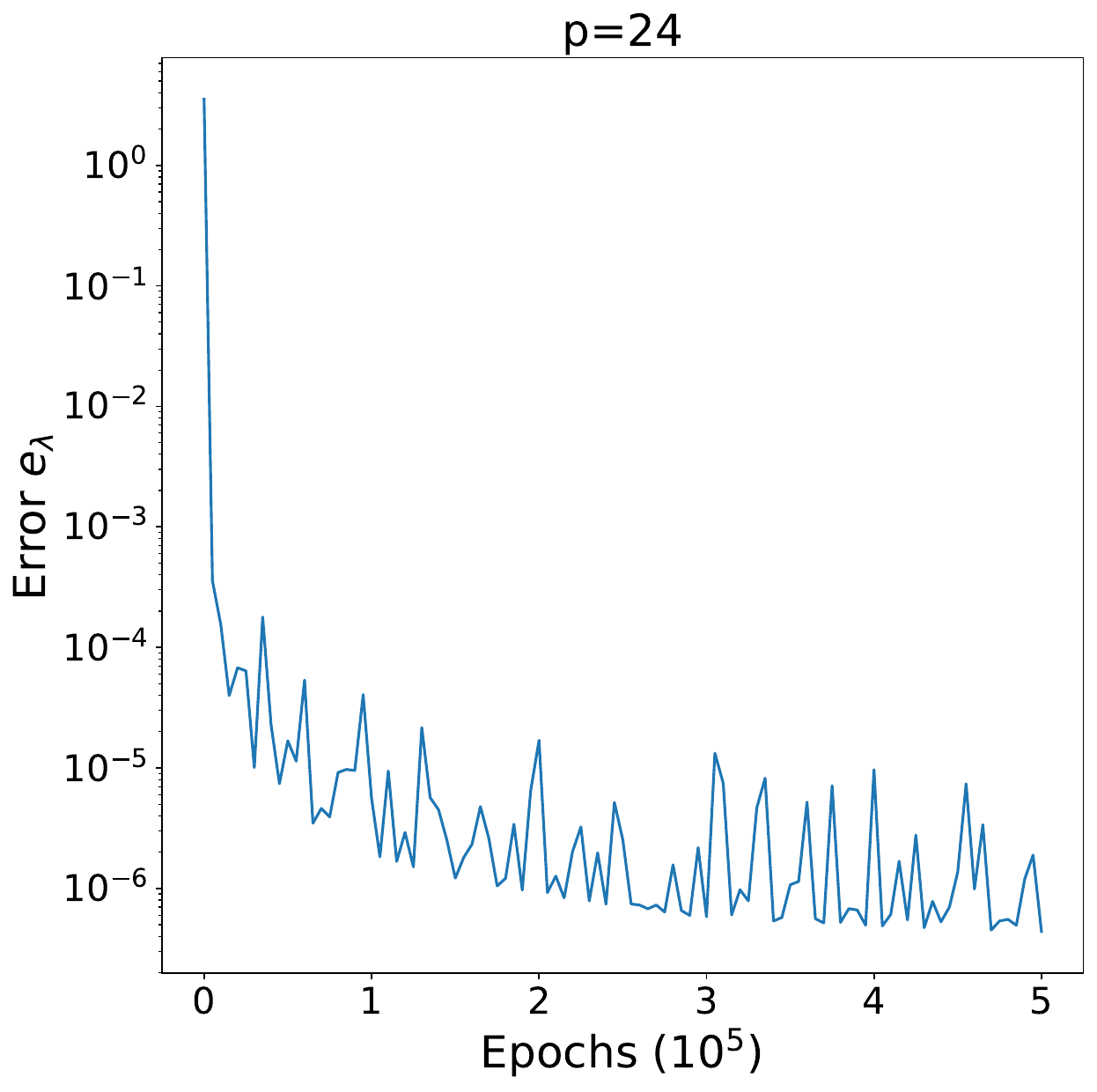}
\includegraphics[width=2.75cm,height=2.5cm]{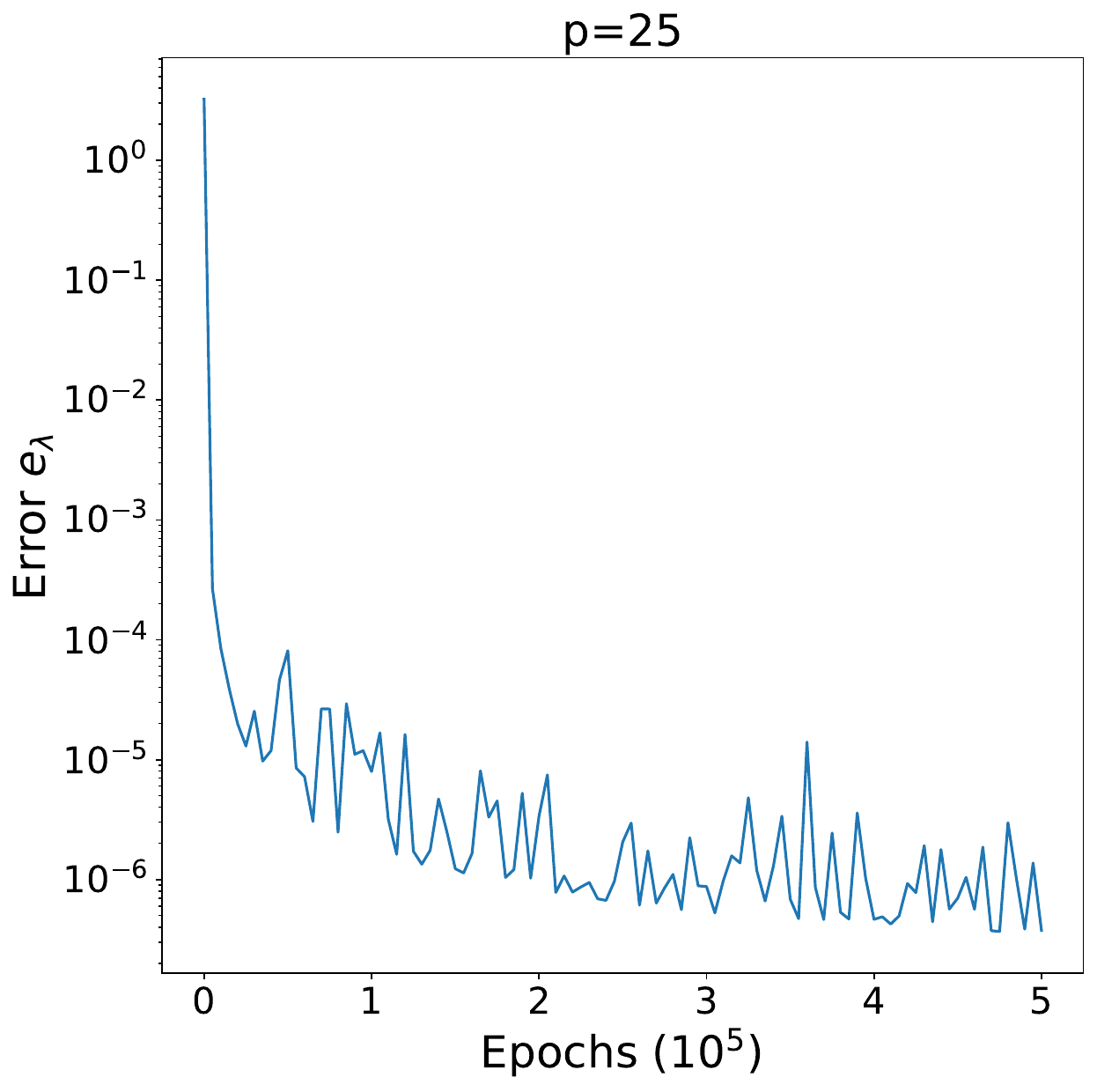}\\
\includegraphics[width=2.75cm,height=2.5cm]{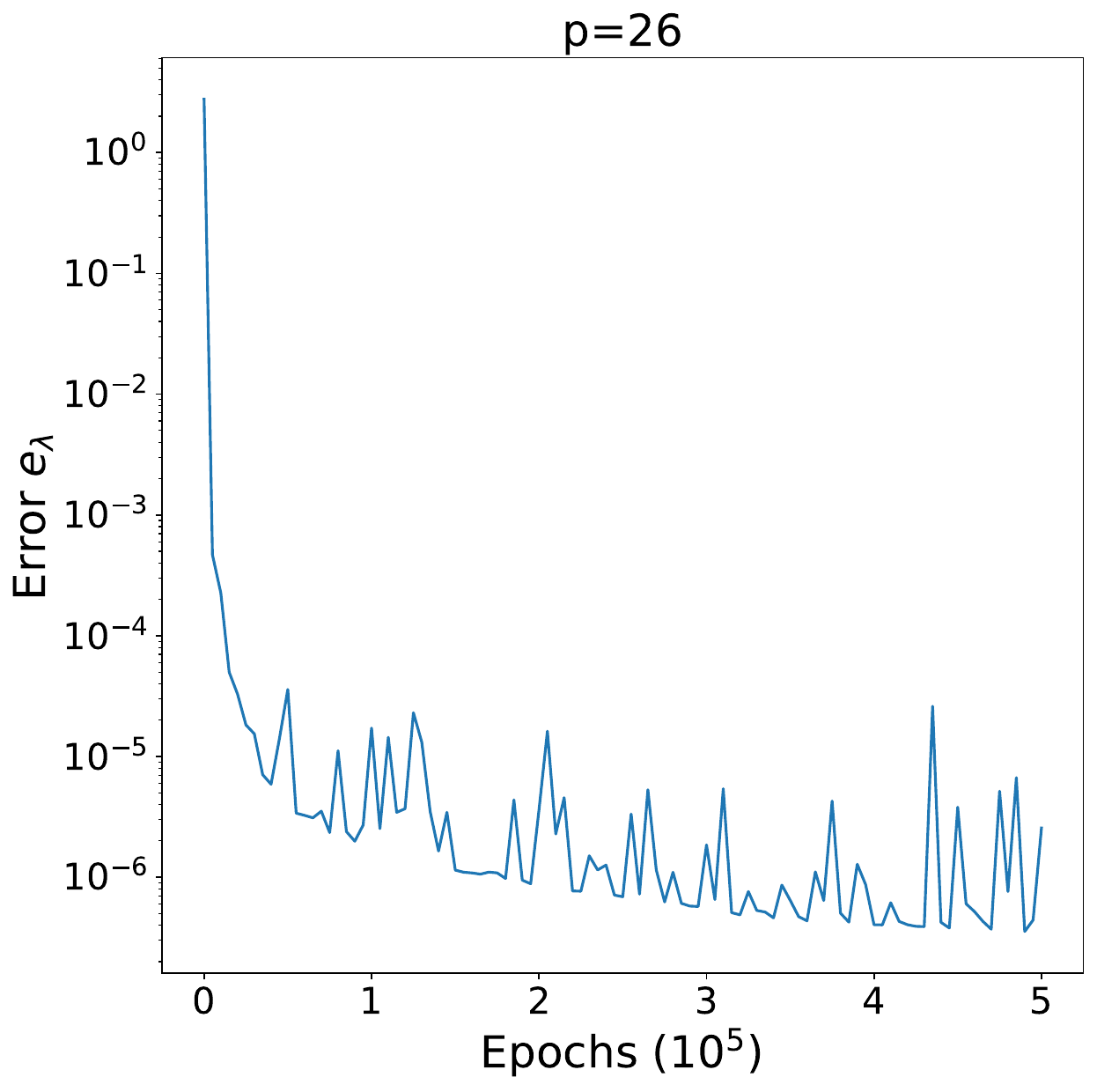}
\includegraphics[width=2.75cm,height=2.5cm]{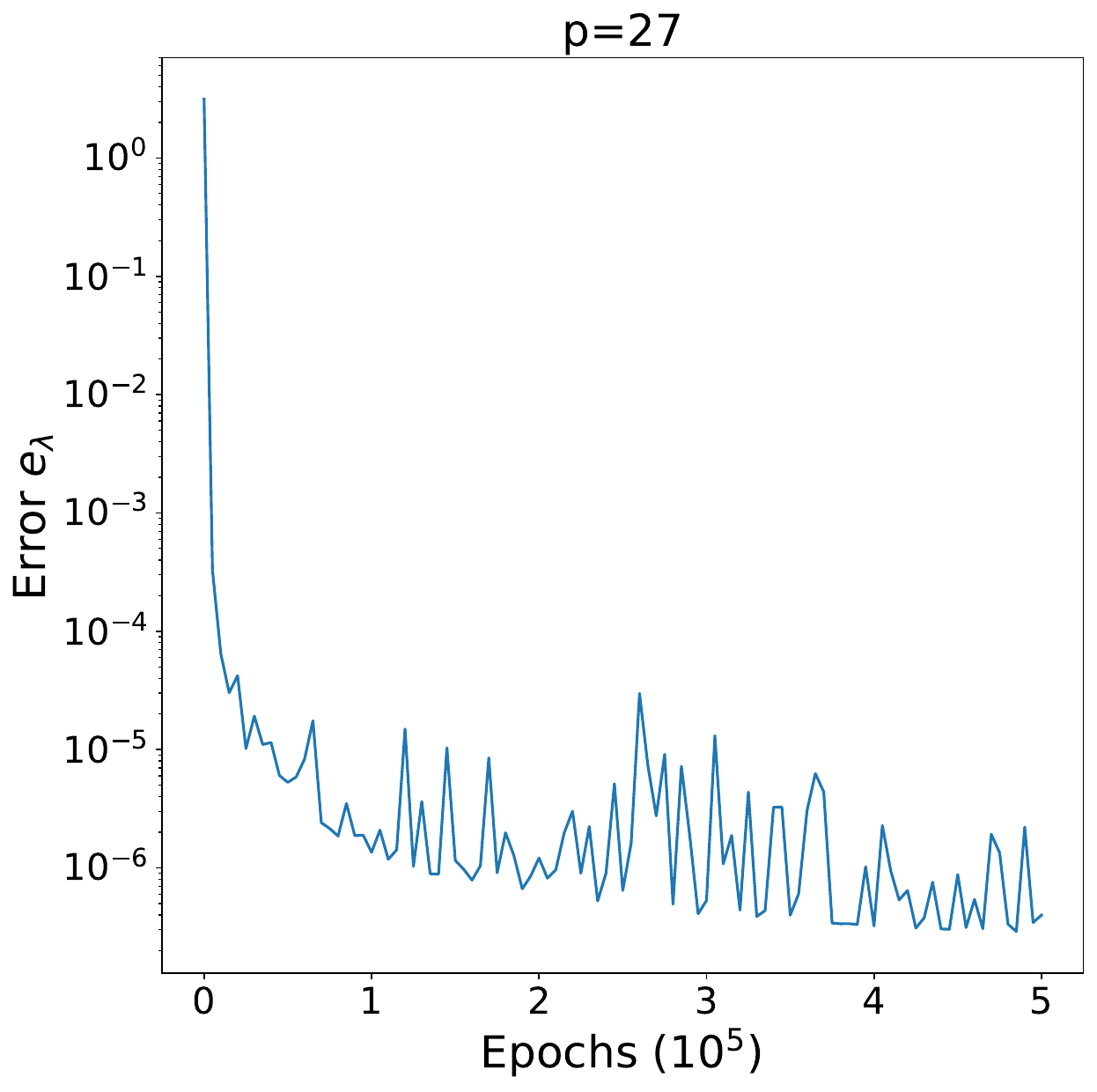}
\includegraphics[width=2.75cm,height=2.5cm]{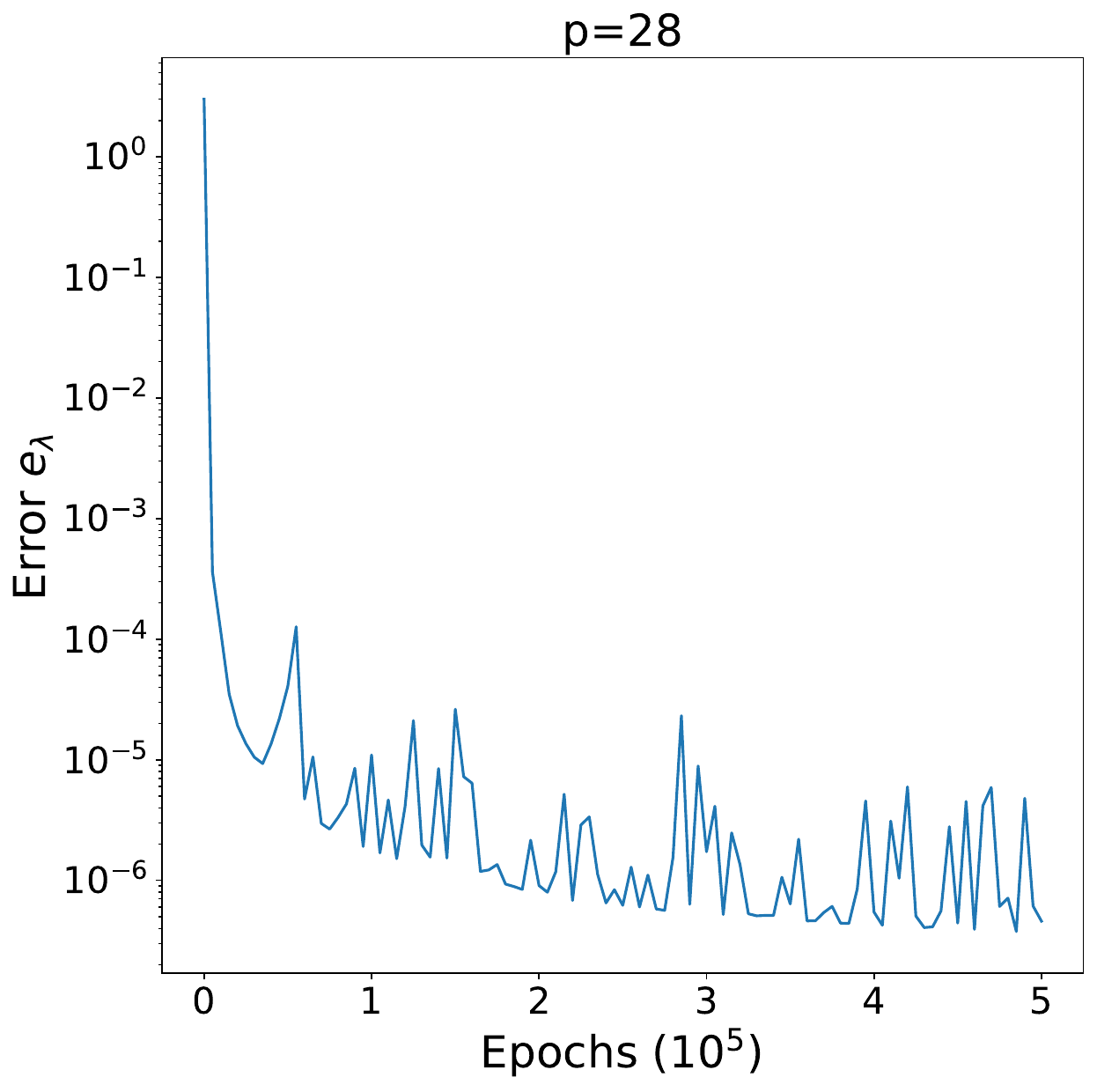}
\includegraphics[width=2.75cm,height=2.5cm]{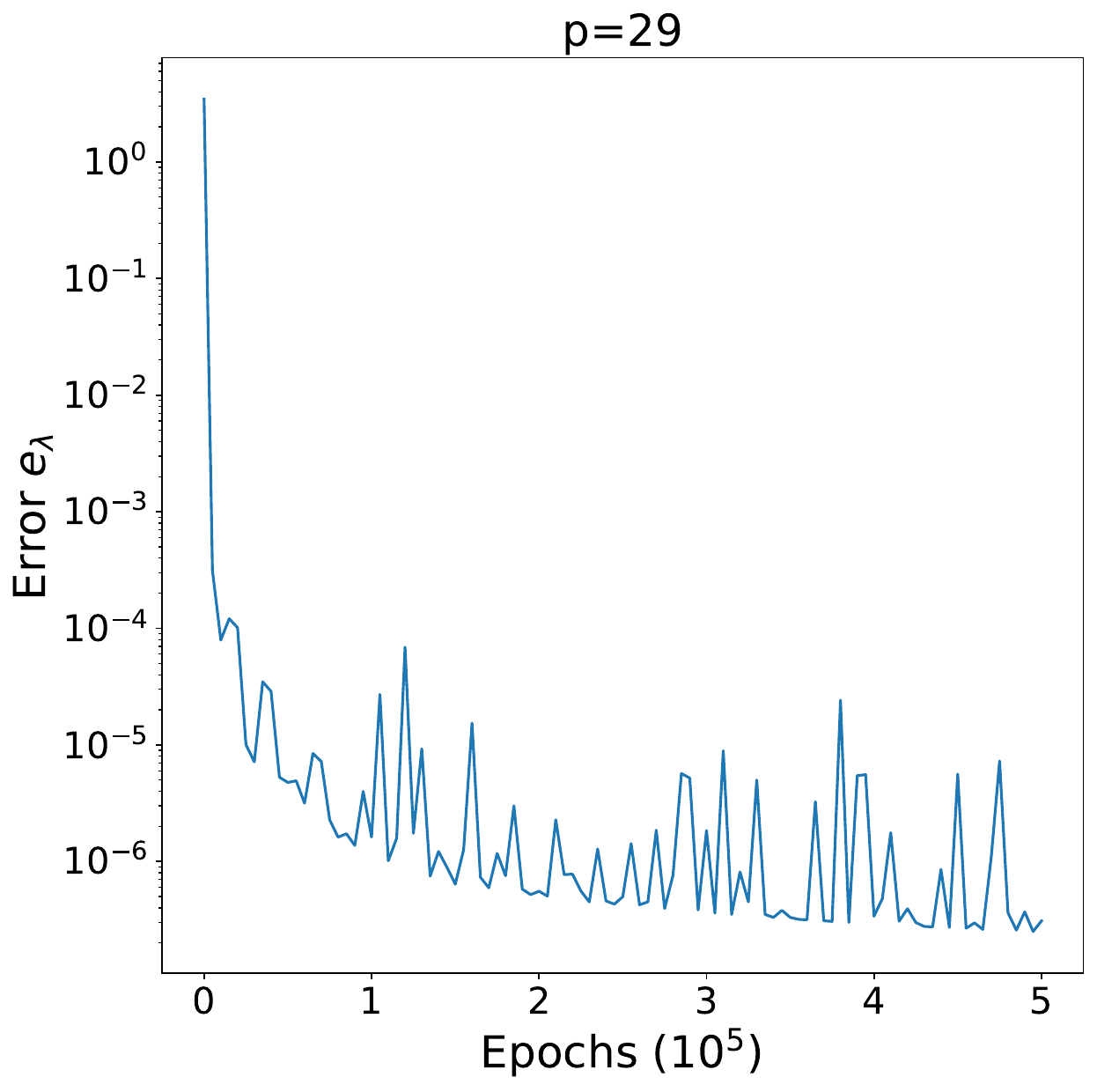}
\includegraphics[width=2.75cm,height=2.5cm]{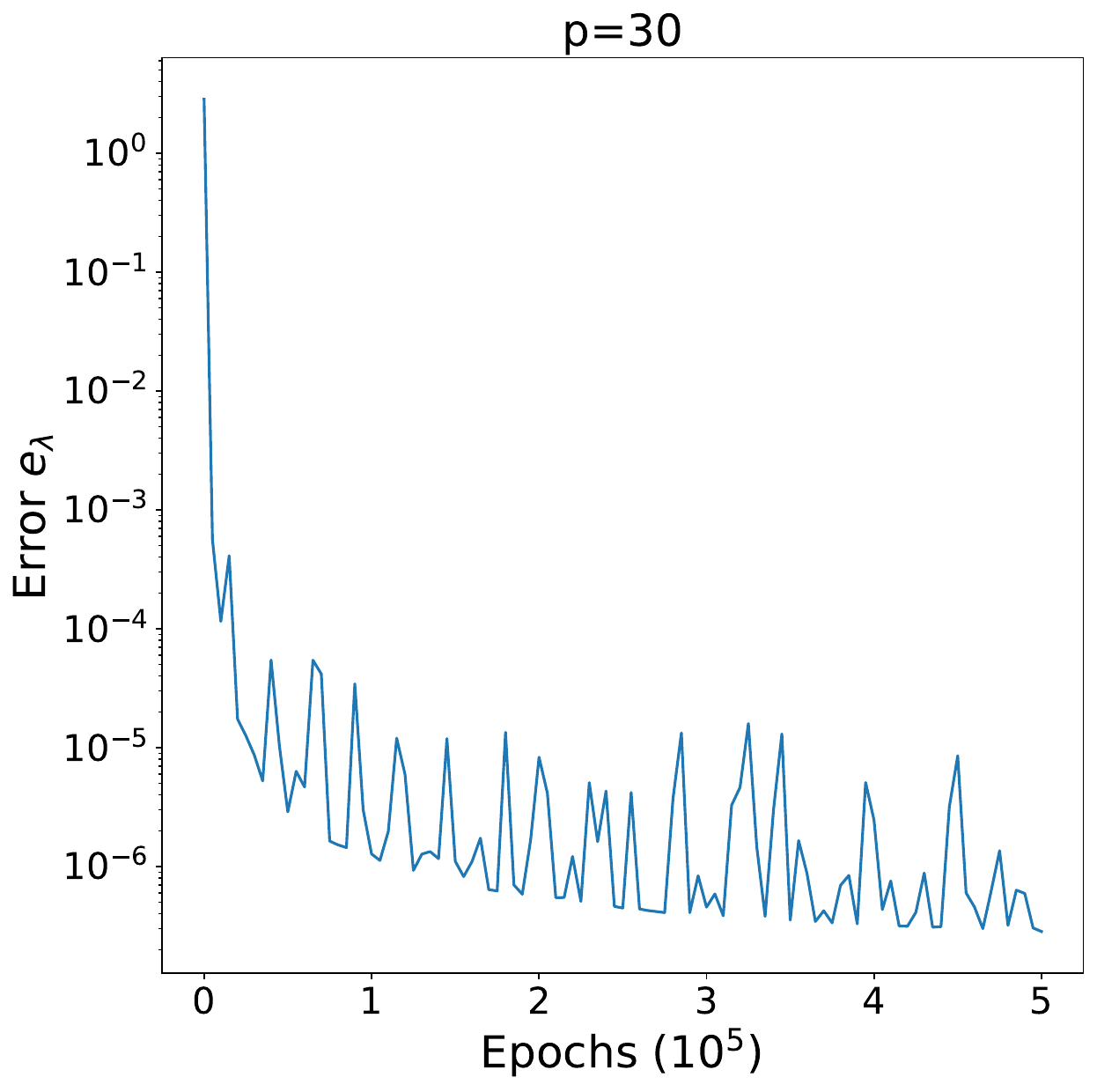}\\
\caption{Relative errors during the training process for the coupled harmonic oscillator:
the rank $p$ increases from 1 to 30.}\label{fig_coupled}
\end{figure}

\begin{figure}[htb!]
\centering
\includegraphics[width=6cm,height=4cm]{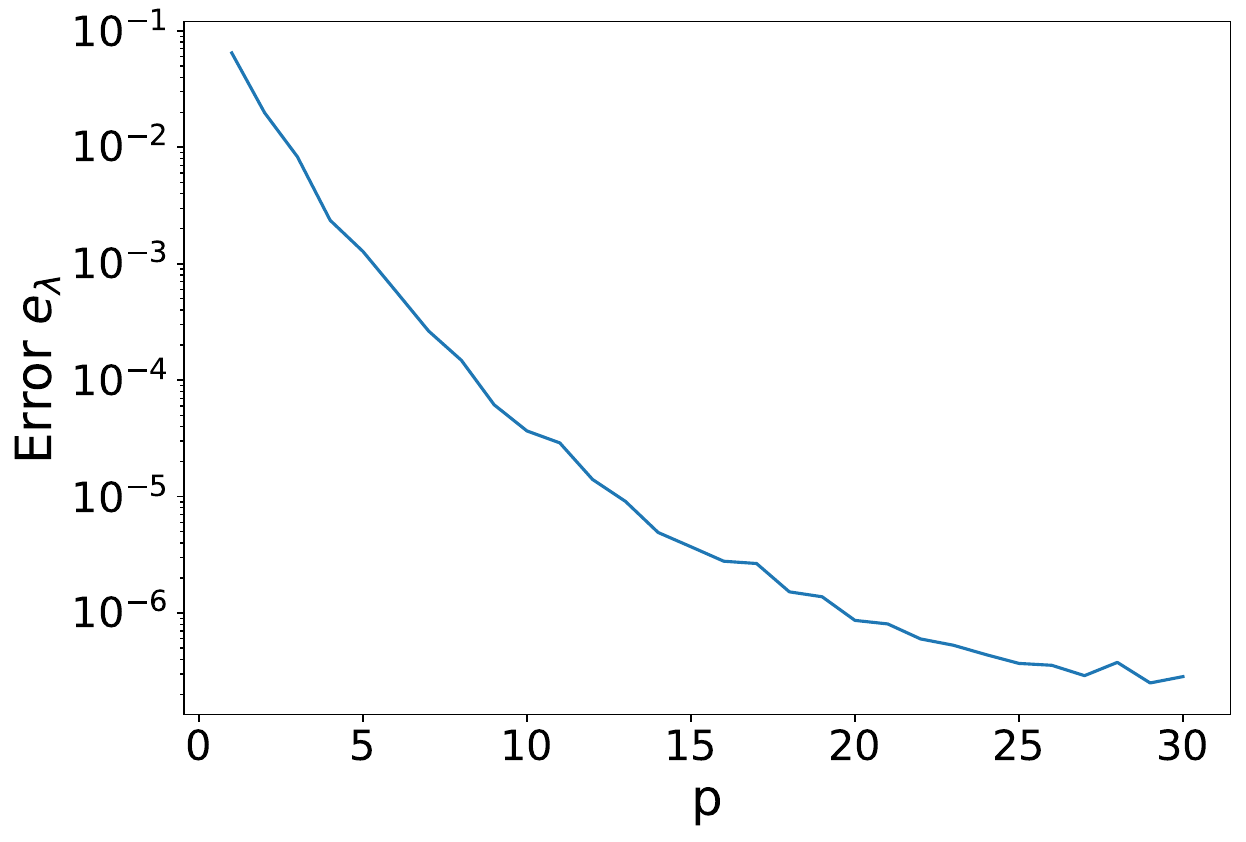}
\caption{Relative errors $e_\lambda$ versus hyperparameter $p$ for the coupled harmonic oscillator ($d=4$). }\label{fig_coupled_error_p}
\end{figure}

\subsection{Ground state of helium atom}
In the fourth example, we consider the Schr\"{o}dinger equation of the helium atom whose potential
cannot be exactly expressed as a CP decomposition of finite rank.
The wave function of the helium atom with the fixed nucleus in Euclidean coordinates
$\Psi(x_1,y_1,z_1,x_2,y_2,z_2)$ satisfies the following eigenvalue problem
\begin{eqnarray}\label{eq_helium}
-\frac{1}{2}\Delta\Psi-\frac{2\Psi}{r_1}-\frac{2\Psi}{r_2}+\frac{\Psi}{r_{12}}=E\Psi,
\end{eqnarray}
where
\begin{eqnarray*}
r_1^2&=&x_1^2+y_1^2+z_1^2, \ \ \ \  r_2^2=x_2^2+y_2^2+z_2^2,\nonumber\\
r_{12}^2&=&(x_1-x_2)^2+(y_1-y_2)^2+(z_1-z_2)^2.
\end{eqnarray*}
Since the potential term $\frac{1}{r_{12}}$ in (\ref{eq_helium}) can not be expressed as a
CP decomposition of finite rank in either Euclidean or spherical coordinates,
it is impossible to give the analytical expressions for exact energy $E$ and wave
function $\Psi$ and it is also difficult to perform the TNN-based machine learning
directly on this potential.
Fortunately, Hylleraas \cite{Hylleraas} chose the three independent variables $r_1,r_2,\theta$, with $\theta$ being the angle between $r_1$ and $r_2$, to determine the form and the size of a triangle that is composed of the nucleus and two electrons.
The coordinates $\{r_1,r_2,\theta\}$ are enough to describe the wave function for the ground state of the helium atom and the corresponding wave function $\Psi(r_1,r_1,\theta)$ satisfies the following eigenvalue problem
\begin{eqnarray*}
-\sum_{i=1}^2\frac{1}{2r_i^2}\frac{\partial}{\partial r_i}\Big(r_i^2\frac{\partial\Psi}{\partial r_i}\Big)-\Big(\sum_{i=1}^2\frac{1}{2r_i}\Big)\cdot\Big[\frac{1}{\sin\theta}
\frac{\partial}{\partial\theta}\Big(\sin\theta\frac{\partial\Psi}{\partial\theta}\Big)\Big]
-\sum_{i=1}^2\frac{2}{r_i}\Psi+\frac{1}{r_{12}}\Psi=E\Psi.
\end{eqnarray*}
The volume of this coordinate is $r_1^2r_2^2\sin\theta$.
The potential $\frac{1}{r_{12}}$ is expanded as functions on the sphere in $\theta$:
\begin{eqnarray}\label{expansion_Legendre}
\frac{1}{r_{12}}=\sum_{\ell=0}^\infty\frac{r_<^\ell}{r_>^{\ell+1}}P_\ell (\cos\theta),
\end{eqnarray}
where $r_>=\max\{r_1, r_2\}$ and $r_<=\min\{r_1, r_2\}$, $P_\ell$ denotes Legendre polynomial of order $\ell$.
In the implementation, we truncate the expression (\ref{expansion_Legendre}) into $20$ terms and the
computational domain from $[0,+\infty)^2\times[0,\pi]$ to $[0,5]^2\times[0,\pi]$.
The benchmark energy for the helium atom is set to be -2.903724377 which is taken from \cite{Nakashima},
at the level of Born-Oppenheimer nonrelativistic ground state energy.
\revise{The TNN is set to be $p=20$. Each subnetwork for the variable $r_1$, $r_2$ or $\theta$, respectively,
is an FNN with two hidden layers and each hidden layer has 50 hidden neurons.}
The boundary condition is guaranteed by multiplying the subnetwork in the $r_i$ direction with $e^{-r_i}$.
We train the TNN epochs with a learning rate of 1e-04 in the first 100000 epochs and with a learning rate of 1e-05
in the subsequent 50000 steps to produce the final result.
The final energy obtained by the TNN method is -2.903781124 and the relative energy error is 1.9542e-05.
Figure \ref{fig_He} shows the radial distribution of electrons for helium atoms.
From Figure \ref{fig_He}, we know that the TNN method can give good simulations of the real electron distribution.
\begin{figure}[http!]
\centering
\includegraphics[width=6cm,height=4cm]{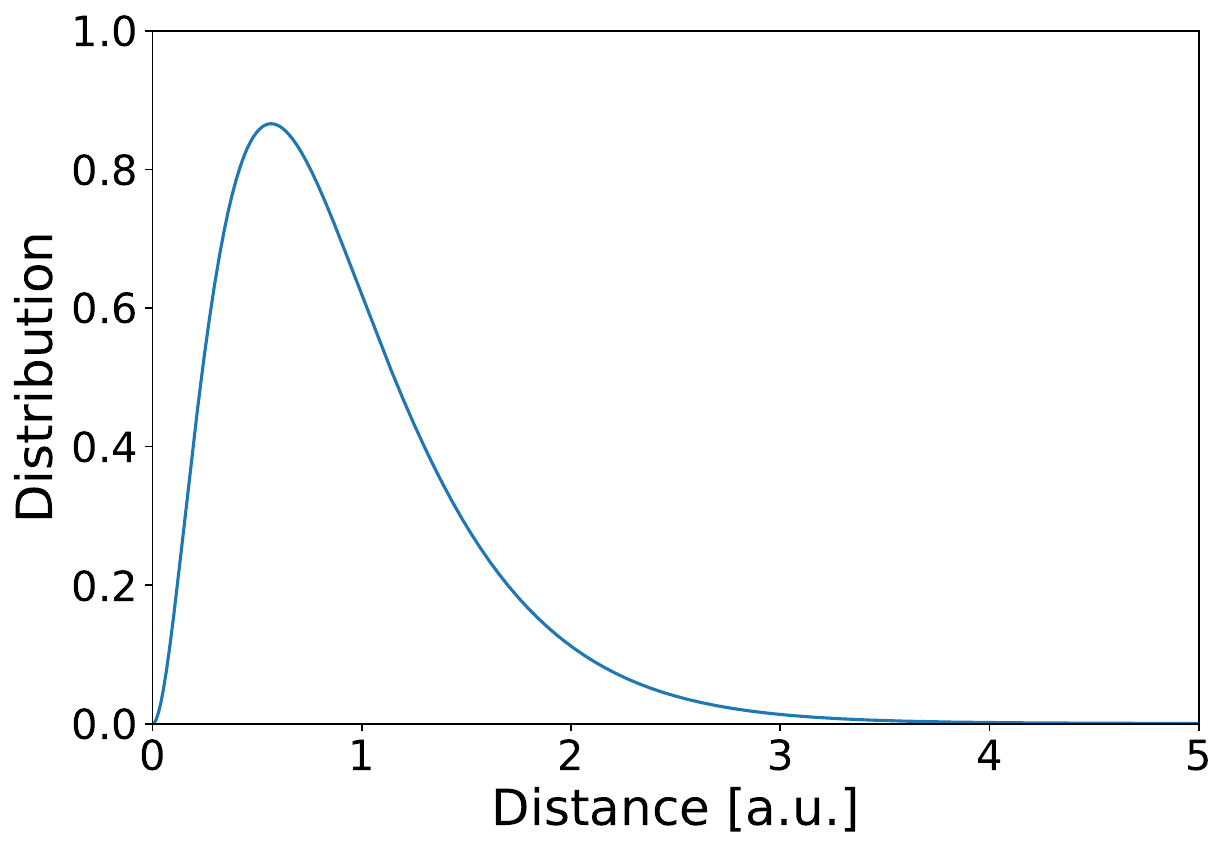}
\caption{Radial distribution of electrons for helium atom.}\label{fig_He}
\end{figure}

\subsection{Boundary value problem with Neumann boundary condition}
In the last example, different from eigenvalue problems discussed in the above subsections,  we tentatively test the performance of TNN for high-dimensional boundary value problems. For this aim, we consider the following boundary
value problem with the Neumann boundary condition
\begin{eqnarray*}\label{neumann}
\left\{
\begin{aligned}
-\Delta u+\pi^2u&=2\pi^2\sum_{i=1}^d\cos(\pi x_i),\ \ \ &x\in&[0,1]^d\\
\frac{\partial u}{\partial \mathbf n}&=0,\ \ \ &x\in&\partial[0,1]^d.
\end{aligned}
\right.
\end{eqnarray*}
Then the exact solution is
\begin{eqnarray}\label{exact_neumann}
u(x)=\sum_{i=1}^d\cos(\pi x_i).
\end{eqnarray}

We use the same loss function as \cite{EYu} and test cases with $d=5,10,20$, respectively.
\revise{For all cases, each subnetwork of TNN is a FNN with two hidden layers and each hidden layer has 50 hidden neurons.}
Referring to
the optimization tips in \cite{MIM}, we use the Adam optimizer in the first 100000 steps and
then the L-BFGS in the subsequent 50000 steps to produce the final result. The corresponding
numerical results for $p=2d$ are reported in Figure \ref{fig_neumann} and Table \ref{table_neumann}.
The final relative errors have almost the same order of magnitude for different dimensions.

\begin{figure}[htb!]
\centering
\includegraphics[width=4cm,height=4cm]{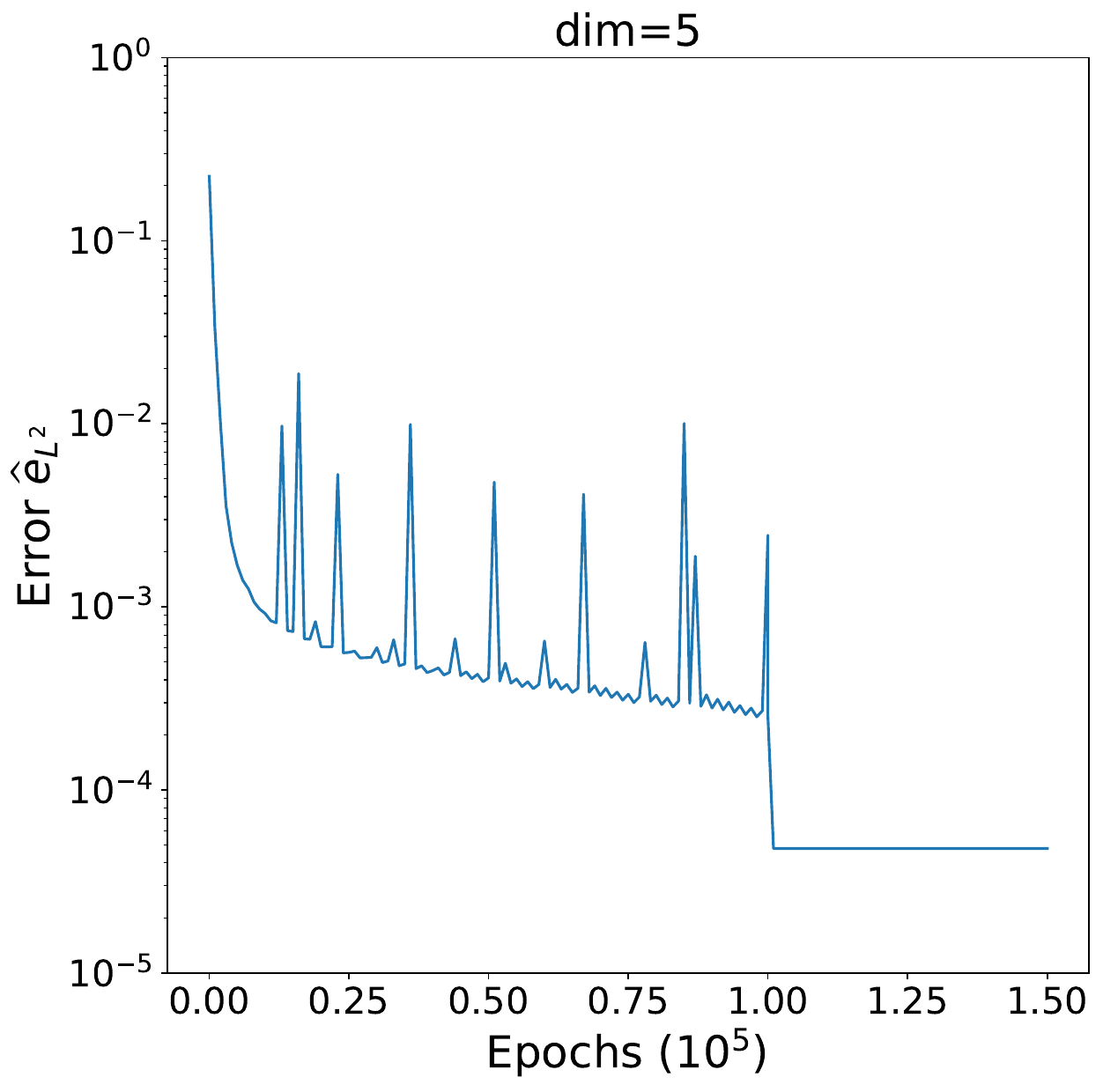}
\includegraphics[width=4cm,height=4cm]{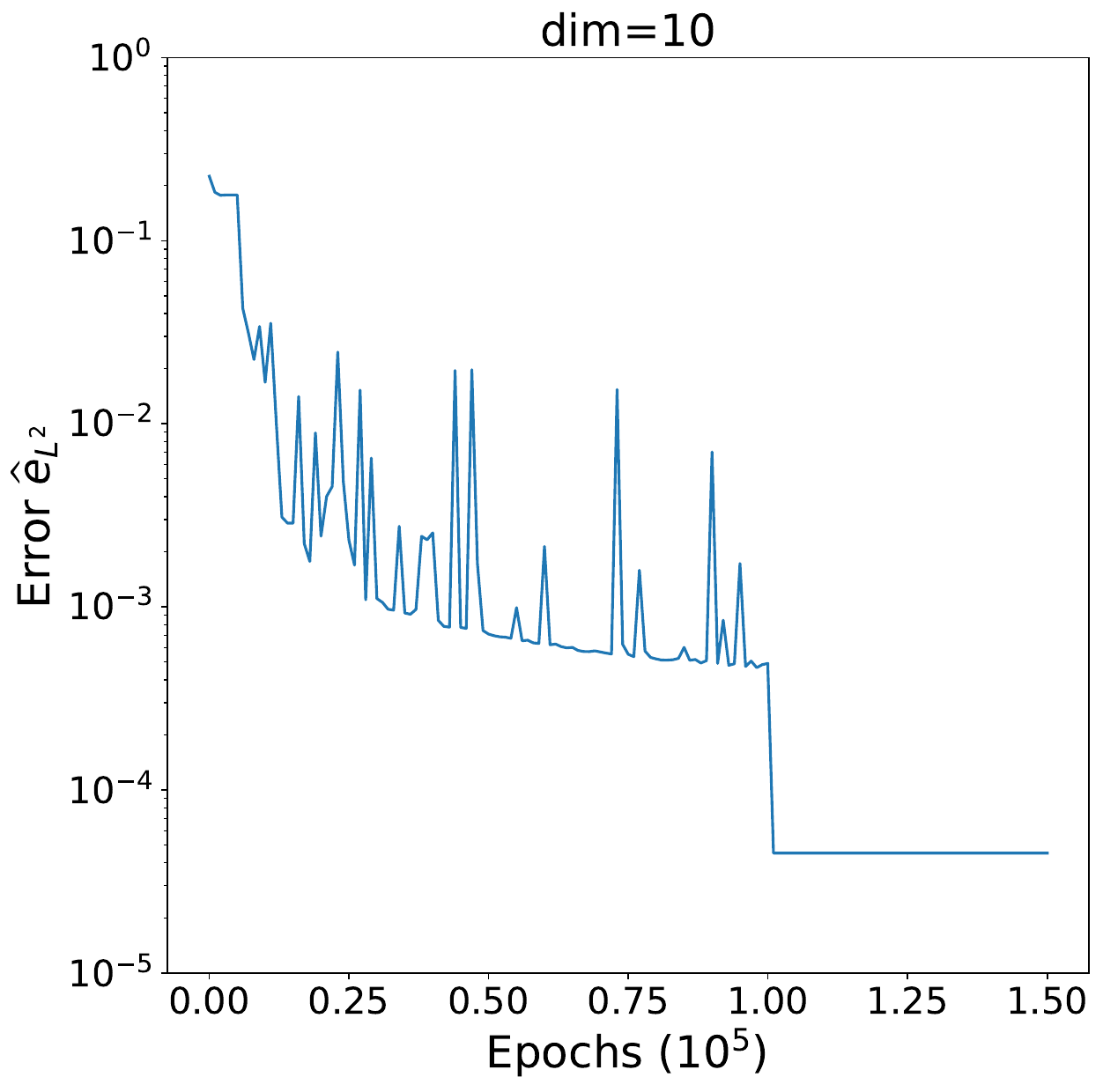}
\includegraphics[width=4cm,height=4cm]{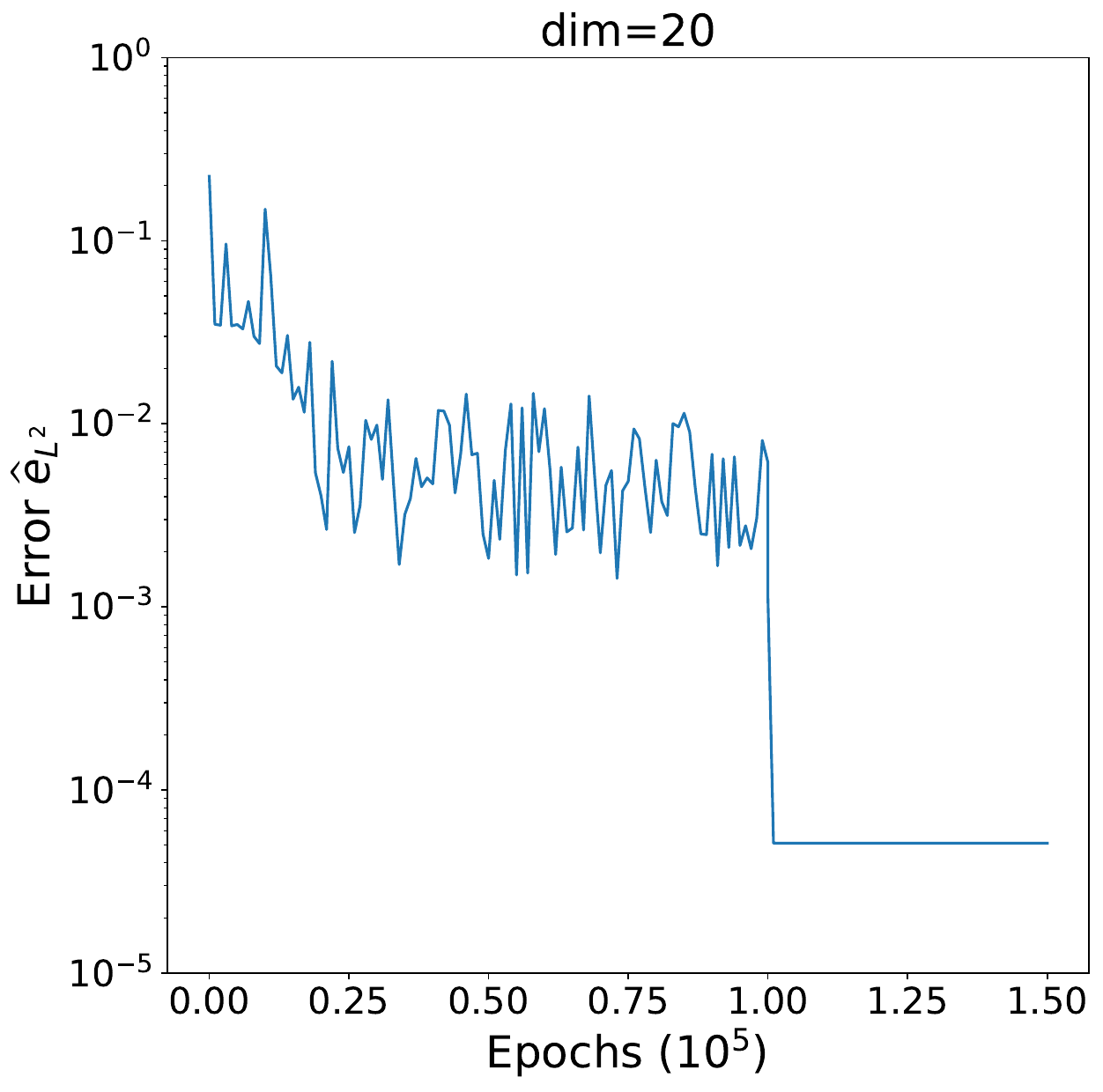}\\
\includegraphics[width=4cm,height=4cm]{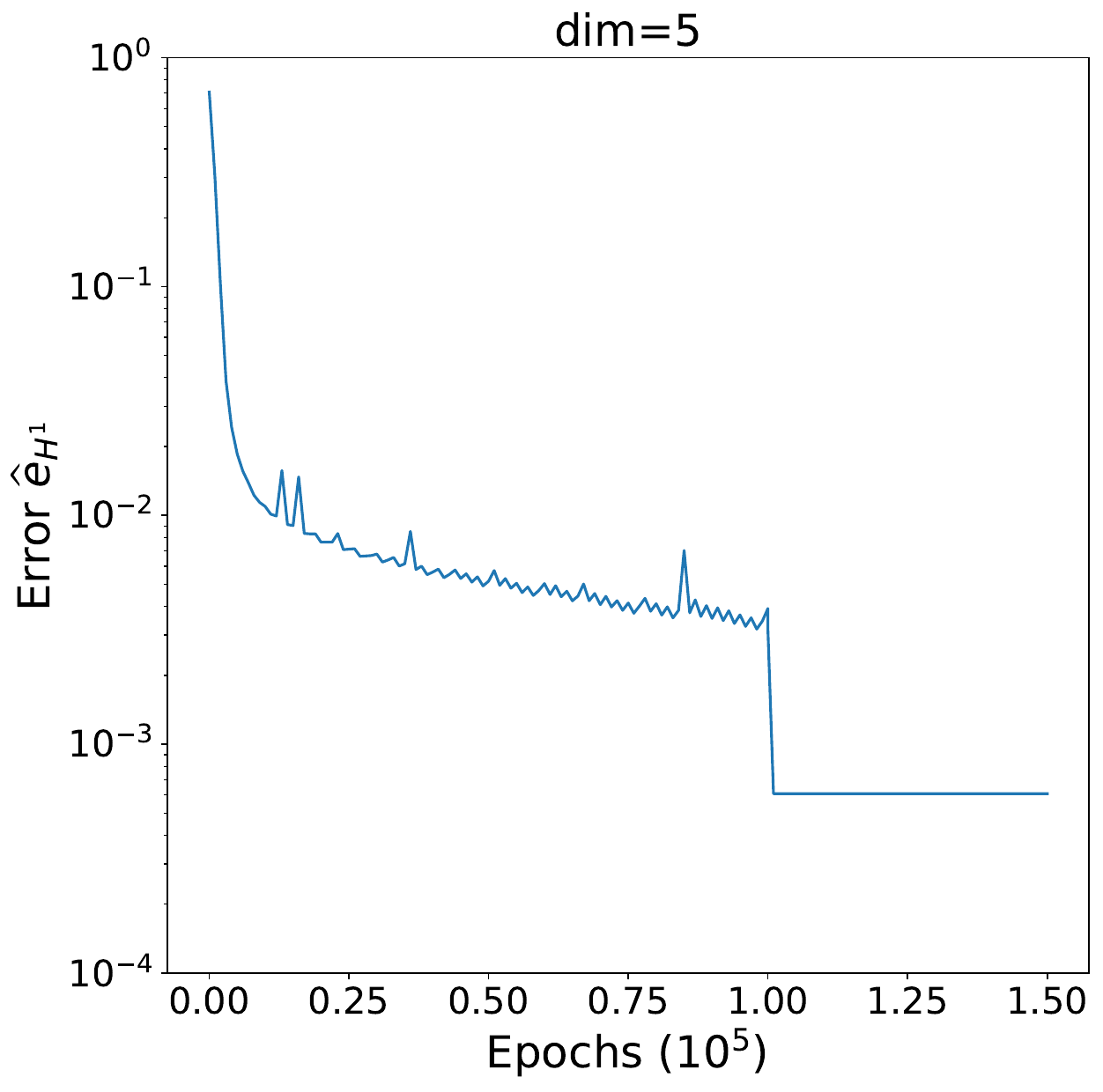}
\includegraphics[width=4cm,height=4cm]{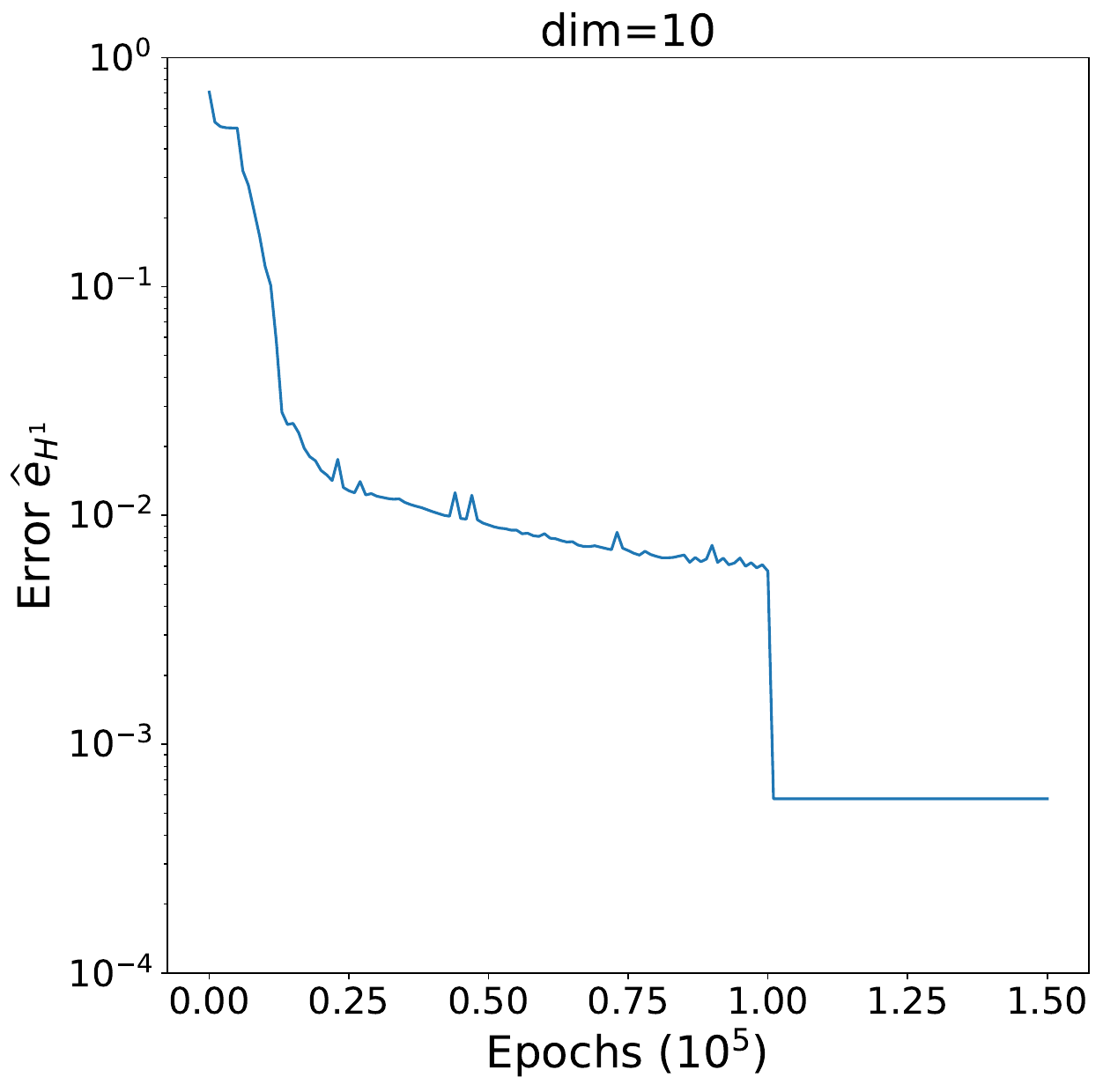}
\includegraphics[width=4cm,height=4cm]{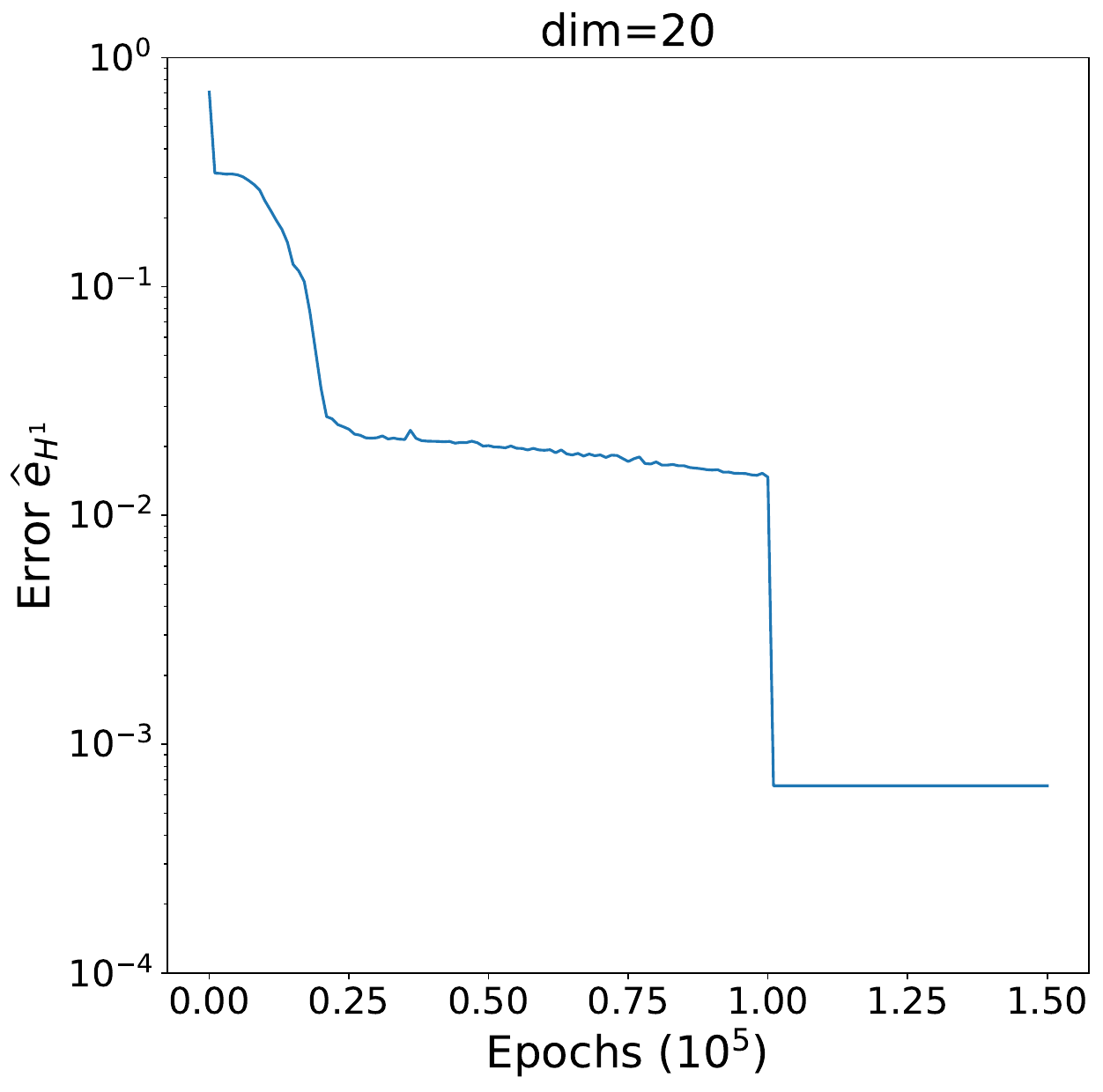}
\caption{Relative errors during the training process for the Neumann boundary problem: for $d=5,10,20$.
The top row shows the relative $L^2$ errors and the bottom one shows the relative $H^1$ errors
of the approximate solution.}\label{fig_neumann}
\end{figure}

\begin{table}[!htb]
\caption{Errors of the Neumann boundary value problem for $d=5,10,20$.}\label{table_neumann}
\begin{center}
\begin{tabular}{ccccc}
\hline
$d$&      $e_{L^2}$&   $e_{H^1}$\\
\hline
5&       4.791e-05&   6.079e-04\\
10&      4.520e-05&   5.778e-04\\
20&      5.122e-05&   6.586e-04\\
\hline
\end{tabular}
\end{center}
\end{table}

From (\ref{exact_neumann}), the exact solution can be represented as CP-decomposition with rank $d$. We can at least claim that the rank of the exact solution is no more than $d$. For the case $d=10$, we take hyperparameter $p$ from 1 to 20 and train the TNN with a learning rate of 0.01. Figure \ref{fig_neumann_error_p} shows the final relative errors $\widehat e_{L^2}$ and $\widehat e_{H^1}$ after 100000 epochs versus $p$. From Figure \ref{fig_neumann_error_p}, we can find an interesting phenomenon that the explicit CP representation (\ref{exact_neumann}) may not describe the effect of low-rank approximation properly. From (\ref{exact_neumann}), it looks like the real rank of the exact solution is $p=10$, but there is no significant error reduction from $p=5$ to $p=20$.

\begin{figure}[htb!]
\centering
\includegraphics[width=4cm,height=4cm]{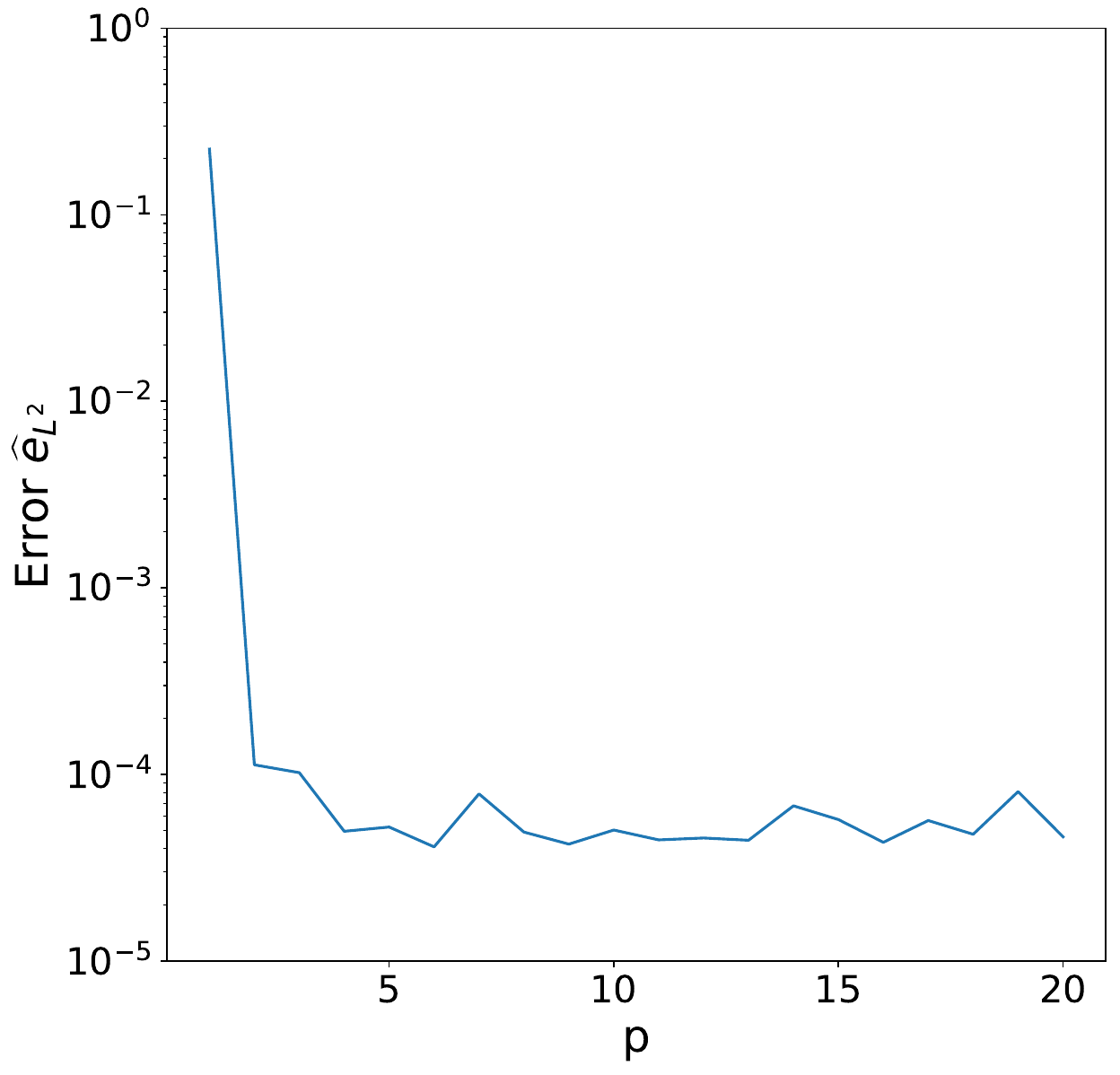}
\includegraphics[width=4cm,height=4cm]{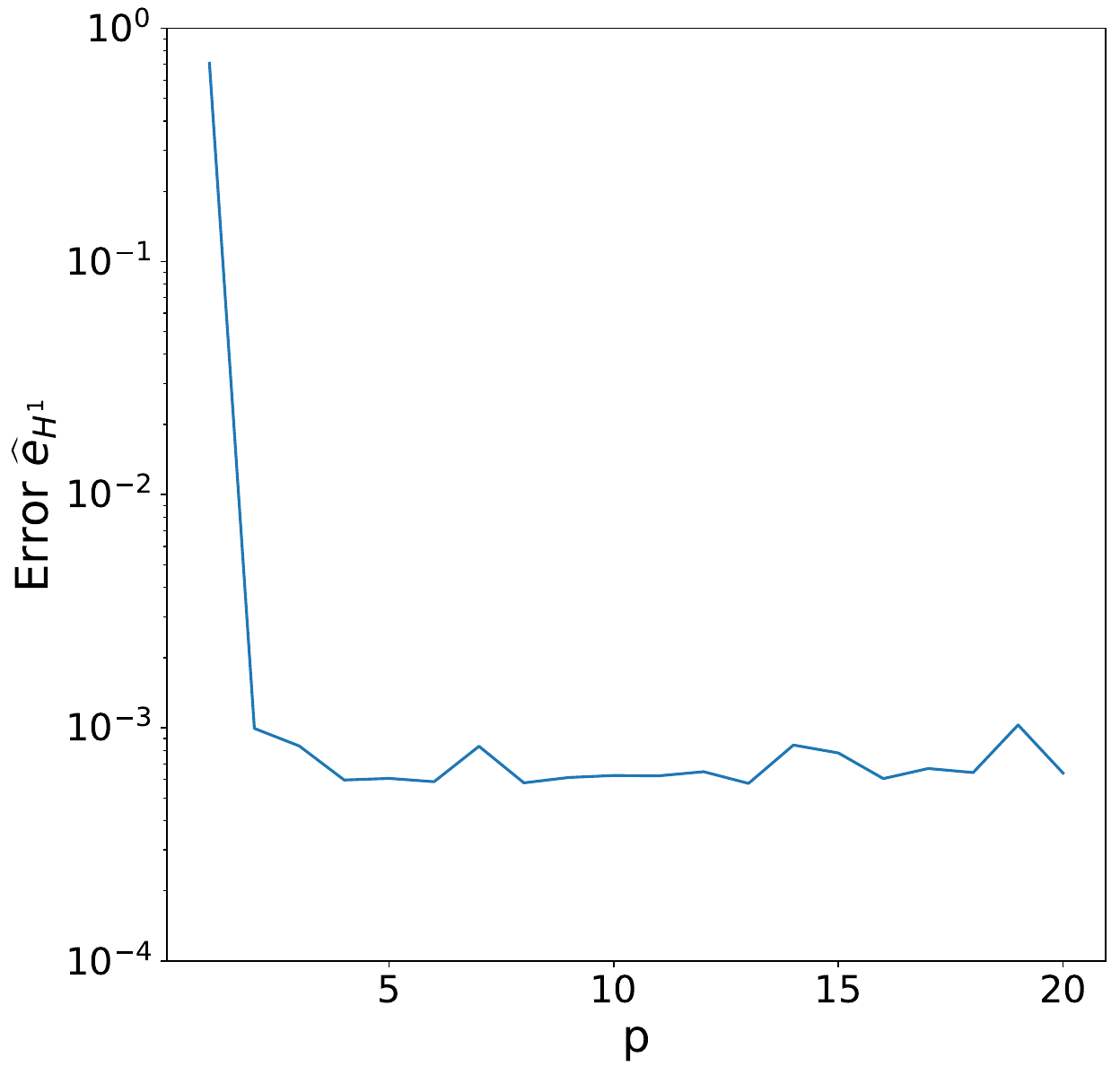}
\caption{Relative errors vs. hyperparameter $p$ for Neumann boundary value problem ($d=10$).
The left subfigure shows the relative $L^2$ errors and the right one shows the relative $H^1$ errors of the approximate solution.}\label{fig_neumann_error_p}
\end{figure}

\section{Conclusions}
In this paper, we present the TNN and corresponding machine-learning methods for solving high-dimensional PDEs. Different from the well-known FNN-based machine learning methods,  TNN has a tensor product form and its numerical integration can use the fixed quadrature points in each dimension. Benefiting from the tensor product structure, we can design an efficient integration scheme for the functions defined by TNN. These properties lead to TNN-based machine learning that can do the direct inner product computation with the polynomial scale of work for the dimension. We believe that the ability of direct inner production computation will bring more applications in solving high-dimensional PDEs.

Based on the ideas of CP decomposition for tensor product Hilbert space and representing the trial functions by deep neural networks, we introduce the TNN structure, its corresponding approximation property, and an efficient numerical integration scheme. The theoretical results, algorithms, and numerical investigations show that this type of structure has the following advantages:
\begin{enumerate}
\item With the help of the straightforward tensor product representation way, we can integrate
this type of function separately in the 1-dimensional interval. This is the reason that the
TNN can overcome the exponential dependence of the computational work for
high-dimensional integrations on the dimension.

\item Instead of randomly sampling data points, the training process uses fixed quadrature points.
This means that the TNN method can avoid the random sampling process to produce the GD
direction in each step and then has better stability.
\end{enumerate}

Besides above observations, there should exist some interesting topics that need to be addressed in future work:
\begin{enumerate}
\item The choice of the subnetwork structure, the activation function, and the more important hyperparameter $p$.

\item When the computing domain is not tensor-product type, further strategies are demanded to maintain the high efficiency and accuracy of the numerical integration.

\item  Since the TNN uses fixed quadrature points, we should design more efficient numerical methods to solve the included optimization problems in the machine learning process.
\end{enumerate}

In addition, more applications to other types of problems should be investigated in the future.

\end{document}